\RequirePackage{amsmath}
\documentclass[ijoc]{informs3}

\usepackage{amsmath}
\usepackage{amssymb}
\usepackage{graphicx}
\usepackage{epsfig}
\usepackage{mathtools}
\usepackage{makecell}
\usepackage{enumerate,dcolumn,latexsym,epstopdf,color}
\usepackage[shortlabels]{enumitem}
\usepackage{lscape}
\usepackage{algorithm}
\usepackage{tabularx}
\usepackage[usenames,dvipsnames,svgnames]{xcolor}
\usepackage{amssymb}
\usepackage{multirow}
\usepackage[flushleft]{threeparttable}
\usepackage{booktabs}
\newcommand{\real}{\mathbb{R}} %Real numbers
\newcommand{\norm}[1]{\left\Vert {#1} \right\Vert} %Norm
\newcommand{\qedsymbol}{$\square$}
\newcommand{\ie}{{\em i.e.}, }
\newcommand{\eg}{{\em e.g.}, }

\newcommand{\bzero}{{\mathbf 0}}

\newcommand{\bb}{{\mathbf b}}
\newcommand{\bc}{{\mathbf c}}
\newcommand{\bd}{{\mathbf d}}

\let\be\relax  
\newcommand{\be}{{\mathbf e}}

\newcommand{\bh}{{\mathbf h}}
\newcommand{\bk}{{\mathbf k}}
\newcommand{\bp}{{\mathbf p}}
\newcommand{\bq}{{\mathbf q}}
\newcommand{\br}{{\mathbf r}}

\newcommand{\bu}{{\mathbf u}}
\newcommand{\bv}{{\mathbf v}}
\newcommand{\bw}{{\mathbf w}}
\newcommand{\bx}{{\mathbf{x}}}
\newcommand{\by}{{\mathbf y}}
\newcommand{\bz}{{\mathbf z}}
\newcommand{\bba}{{\mathbf A}}
\newcommand{\bbb}{{\mathbf B}}

\newcommand{\bbi}{{\mathbf I}}

\newcommand{\bbp}{{\mathbf P}}
\newcommand{\bbq}{{\mathbf Q}}

\newcommand{\balpha}{{\boldsymbol \alpha}}
\newcommand{\bbeta}{{\boldsymbol \beta}}

\newcommand{\bchi}{{\boldsymbol \chi}}

\newcommand{\bkappa}{{\boldsymbol \kappa}}

\newcommand{\bomega}{{\boldsymbol \omega}}
\newcommand{\bzeta}{{\boldsymbol \zeta}}
\newcommand{\blambda}{{\boldsymbol \lambda}}
\newcommand{\btheta}{{\boldsymbol \theta}}
\newcommand{\bnu}{{\boldsymbol \nu}}

\newcommand{\bxi}{{\boldsymbol \xi}}

\newcommand{\bell}{{\boldsymbol \ell}}
\newcommand{\bvarpi}{{\boldsymbol \varpi}}

\DeclareMathOperator*{\dist}{Dist}
\DeclareMathOperator*{\diag}{Diag}
\DeclareMathOperator*{\pinf}{inf\vphantom{p}}
\newcommand{\bigzero}{\mbox{\normalfont\Large\bfseries 0}}
\newcommand{\bigI}{\mbox{\normalfont\Large\bfseries I}}
\newcommand{\rvline}{\hspace*{-\arraycolsep}\vline\hspace*{-\arraycolsep}}

\usepackage{accents}
\newcommand{\ubar}[1]{\underaccent{\bar}{#1}}
\newcommand{\utilde}[1]{\underaccent{\tilde}{#1}}

\newcommand\shimrit[1]{{\color{black}#1}}

\OneAndAHalfSpacedXI
\usepackage{natbib}
\bibpunct[, ]{(}{)}{,}{a}{}{,}%
\def\bibfont{\fontsize{8}{9.5}\selectfont}%
\def\BIBand{and}%

\TheoremsNumberedThrough     % Preferred (Theorem 1, Lemma 1, Theorem 2)
\ECRepeatTheorems

\EquationsNumberedThrough    % Default: (1), (2), ...
\begin{document}
	
	%%%%%%%%%%%%%%%%
	
	%\setcounter{page}{1}
	%\DOI{10.1287/xxxx.0000.0000}%
	
	\RUNAUTHOR{Postek and Shtern}
	
	\RUNTITLE{First-order algorithms for robust optimization}
	\TITLE{First-order algorithms for robust optimization problems via convex-concave saddle-point Lagrangian reformulation}
	
	\ARTICLEAUTHORS{%
		\AUTHOR{Krzysztof Postek}
		\AFF{Faculty of Electrical Enginnering, Mathematics and Computer Science,
			Delft University of Technology, Delft, The Netherlands, \EMAIL{k.s.postek@tudelft.nl}, \URL{}}
		\AUTHOR{Shimrit Shtern}
		\AFF{Faculty of Industrial Engineering and Management, Technion - Israel Institute of Technology, Haifa, Israel, \EMAIL{shimrits@technion.ac.il}, \URL{}}
	}
	
	\ABSTRACT{%
		Robust optimization (RO) is one of the key paradigms for solving optimization problems affected by uncertainty. Two principal approaches for RO, the robust counterpart method and the adversarial approach, potentially lead to excessively large optimization problems. For that reason, first order approaches, based on online-convex-optimization, have been proposed \cite{bental2015oracle,ho2018online} as alternatives for the case of large-scale problems. However, these methods are either stochastic in nature or involve a binary search for the optimal value. We show that this problem can also be solved with deterministic first-order algorithms based on a saddle-point Lagrangian reformulation that avoid both of these issues. Our approach recovers the other approaches' $\mathcal{O}(1/\epsilon^2)$ convergence rate in the general case, and offers an improved $\mathcal{O}(1/\epsilon)$ rate for problems with constraints which are affine both in the decision and in the uncertainty. Experiment involving robust quadratic optimization demonstrates the numerical benefits of our approach.
	}%
	
	%\FUNDING{The research of the second author is supported by the Israeli Science Foundation (ISF) Grant 1460/19.}
	
	%SUBJECTCLASS{}
	
	%\AREAOFREVIEW{Optimization.}
	
	%\KEYWORDS{}
	
	\maketitle
	
	\section{Introduction}
	When solving optimization problems, one often has to deal with uncertainty present in the parameters of the objective and constraint functions. This uncertainty may stem from measurement, implementation or prediction errors. A common paradigm used to ensure that the solution remains feasible under uncertainty is \emph{robust optimization} (RO) \cite{bental2009robust}. In RO, the uncertainty is assumed to be adversarial to the decision maker and to lie in a predefined \emph{uncertainty set}, and the decision maker finds the best solution which remains feasible for all parameter values within this set.
	
	Any RO algorithm involves simultaneous solving of two problems: the decision maker's problem of finding the best-possible decision which is feasible for the given uncertainty set, and nature's implicit adversarial problem of selecting the worst possible realization of the parameters.% given the decision maker's decision. 
	
	Denoting the decision maker's decision as $\bx$, taken from a predefined set $X$, and the uncertain parameter for constraint $i$ as $\bz_i$, taken from a predefined uncertainty set $Z^i$, a general RO problem is given by:
	\begin{align}
		\min\limits_{\bx\in X} \ &  \bc^\top\bx  \label{eq:basic_problem} \\
		\text{s.t.} \ &  \sup\limits_{\bz_i \in Z^i} g_i(\bx, \bz_i) \leq 0, & i \in [m]. \nonumber
	\end{align}
	In this problem, the decision maker minimizes over $\bx\in X$, while the nature selects $\bz\in Z=Z^1\times Z^2\times\ldots \times Z^m$ such that the left-hand sides of the inequality constraints are maximized, in an attempt to violate the constraints.
	
	\vspace{5pt}
	\noindent{\bf{Standard methods.}} There are two key methods for handling the opposite optimization problems simultaneously: (1) the \emph{robust counterpart} (RC) reformulation method \cite{BenTal2015,bental2009robust}, (2) \emph{adversarial approach}, also known as cutting planes \cite{bienstock2007histogram,mutapcic2009cutting}. In the RC approach, nature's maximization problem for each constraint is dualized, and the objective of the dual problem is required to be nonpositive. In this way, a constraint is substituted by a system of inequalities that are satisfied for a given $\bx$ if and only if the original left-hand sides are nonpositive. This method's key advantage is that it requires to solve only one optimization problem and it ensures by-design the robustness of the solution. Its disadvantage is the increase of the problem size due to the added dual variables and constraints generated by dualizing nature's problems. Moreover, for strong duality to hold for nature's problem in each constraint, it has to be convex, or equivalent to a convex problem.
	
	In the adversarial approach, a finite subset $\ubar{Z} \subset Z$ of scenarios is iteratively built up, until it contains enough points to ensure that $\bx$ is feasible for all realizations in the subset if and only if it is (almost) feasible for all realizations in $Z$. The set $\ubar{Z}$ is intialized using an arbitrarily chosen realization $\bz \in Z$. Then two steps are repeated alternatingly: an optimization step, in which a solution $\bx$ which minimizes the objective function and is feasible for $\bz \in \ubar{Z}$ is found; a \emph{pessimization step}, in which a new realization $\bz$ violating at least one constraint is found and added to $\ubar{Z}$. This iterative procedure continues until no violating scenario is found. While this method is simple and enables solving problems in which nature's problem is not necessarily convex, the size of $\ubar{Z}$ increases at each iteration, which may result in extremely large optimization problems in $\bx$ during the optimization step.
	
	\vspace{5pt}
	\noindent{\bf{Need for lower-order approaches.}} Both of the above methods potentially lead to excessively large optimization problems which creates space for approaches in which the decision maker's and the adversary's problems are simultaneously solved in a lightweight fashion. A recently suggested idea \cite{bental2015oracle,ho2018online} is to use online convex optimization to solve problem \eqref{eq:basic_problem}, proving that the number of iterations needed to obtain an \shimrit{$\epsilon$-feasible $\epsilon$-optimal solution is $O(\log(\frac{1}{\epsilon})\frac{1}{\epsilon^2})$}. Through the online optimization lens, the robust problem is seen as a problem of minimizing a partly unknown objective with likewise constraints, whose shape is learned throughout the algorithm via samples.%subsequent solves.
	
	The approach of \cite{bental2015oracle} consists in iteratively solving a nominal version of \eqref{eq:basic_problem} in which the set $Z$ is replaced by a fixed realization $\bz$. The value of $\bz$ is updated at each iteration through first-order updates/pessimization and randomization. The approach of \cite{ho2018online} is to use binary search to determine the minimal $\tau$ for which the feasibility problem
	$$
	\{ \bc^\top\bx \leq \tau , \ \sup_{\bz \in Z} g_i(\bx, \bz) \leq 0, \quad i \in [m] \}
	$$
	has a solution up to a given accuracy. This requires running an online first-order algorithm for each tested $\tau$. Each of the feasibility problems is solved by a first order iteration on $\bx$, simultaneously with a pessimization/first order steps on the dual parameter $\bz$.
	
	Thanks to the online optimization framework, both \cite{bental2015oracle,ho2018online} work directly with the functions $\{g_i(\cdot, \cdot)\}_{i\in [m]}$, without the need to build an ever-increasing list of scenarios, or to dualize the constraints. This has a price: the functions $g_i$ have to be convex-concave, the set $X$ needs to be bounded, and the maximum achievable convergence rate is $\mathcal{O}(1/\epsilon^2)$ to obtain an $\epsilon$-feasible and $\epsilon$-optimal solution. Also, \cite{bental2015oracle} requires to solve multiple nominal problems, while in \cite{ho2018online} one needs to run binary search for the optimal objective value.
	
	\shimrit{The method in \cite{ho2018online}, was also extended in \cite{ho2019exploiting} to cases in which the functions $g_i$ have additional characteristics to obtain $O(1/\epsilon)$ rate of convergence. Specifically, this result requires that for all $i\in[m]$ either $g_i(\cdot,\bz)$ or $-g_i(\bx,\cdot)$ is strongly convex for any $\bz\in Z$ or $\bx\in X$, respectively, and that the function is smooth in the remaining variable. In this paper, we chose for simplicity, to avoid making such further assumptions on the problem's structure, and instead focus on two cases: the case where $g_i(\cdot,\cdot)$ are continuous real-valued functions, and the case where $g_i$ are biaffine functions. } 
	
	\vspace{5pt}
	\noindent{\bf Research questions.} Two questions that arise from the above review are as follows. 
	\begin{enumerate}
		\item \emph{Under what conditions is \eqref{eq:basic_problem} amenable to applying single-run deterministic, light-weight first order methods?}
		\item \emph{Can the guarantees obtained from such algorithms be transformed into meaningful guarantees for the robust problem \eqref{eq:basic_problem} in terms of the optimality and feasibility gaps?}
	\end{enumerate}
	\vspace{5pt}
	\noindent{\bf Contribution.}
	We address these two questions by leveraging a natural convex-concave saddle-point reformulation of \eqref{eq:basic_problem}, based on its Lagrangian. \shimrit{The lifted Lagrangian of problem~\ref{eq:basic_problem}, generally given by
		\begin{equation}\label{eq:lagrangian}\min_{\bx\in X, \bu\in U} \bar{L}(\bx,\bu),\end{equation}
		where $\mathcal{L}$ is a convex in $\bx$ and concave in the lifted variable $\bu$.} Although the existence of such a formulation was noted and discussed in \cite[Appendix A]{ho2018online}, it has been claimed that the problem loses a lot of convenient structure due to the lifting.  However, we show two simple settings under which \shimrit{\eqref{eq:lagrangian}} is amenable to the use of first order algorithms for convex-concave saddle point problems, not only theoretically but also practically. 
	
	Answering the first research question, we start by carefully deriving conditions under which we can map saddle points of this formulation to optimal solutions of problem~\eqref{eq.robust.problem}. Next we turn to solve this formulation using simple first order algorithms.
	
	\shimrit{There is an abundance of applicable first-order algorithms, usually derived for solving variational inequalities of monotone operators (see \cite[Chapter 6]{LectureNem94}), each demanding a different set of assumptions. Indeed, most algorithms demand that the part of the saddle point function connecting the primal and dual variables be smooth in both (\eg \citep{LectureNem94,nemirovski2004prox,juditsky2011first,gidel2017frank} and references therein), which is not the case in our general Lagrange derivation. The convergence of such algorithms
		is usually given by $\mathcal{L}(\bar{\bx}^N,\bu^*)-\mathcal{L}(\bx^*,\bar{\bu}^N)\leq \epsilon_N$ .
		For convenience, in Table~\ref{tbl:sp_algs}, we summarize the known state-of-the-art rates of convergence under different assumptions on the Lagrangian function, the implied assumptions on the
		structure of the original robust problem (See Section~\ref{sec:algorithms} for details), and well-known algorithms achieving these rates. Faster rates are only known for more restrictive assumptions on the functions involved (such as strong convexity \citep{ouyang2021lower,zhang2022lower}). In this paper, in order to avoid assumptions that do not generally hold in the robust setting, we focus on two simple settings. }
	
	\begin{table}[h]\label{tbl:sp_algs}\caption{Existing Saddle Point Algorithms}
		\begin{center}
			\begin{threeparttable}
				\shimrit{\footnotesize
					\begin{tabular}{p{3cm}p{3cm}llp{4.5cm}l}
						\toprule
						SP Problem  & RO Problem    & Iteration cost & Rate of  & Methods\\
						Assumptions & Assumptions &&$\epsilon_N$&\\
						\hline
						%\multirow{4}{3cm}{$\mathcal{L}(\bx,\bu)$ biaffine, $X=\real^n$} & \multirow{2}{3cm}{$g_i$ are biaffine, $X=\real^n$} &  \multirow{2}{*}{$\mathcal{T}_{\bu}(U)$} %&$O(1/N)$ & PAPC \cite{drori2015209}\\  
						%&  &     
						%& \multirow{2}{*}{$O(1/N^2)$} & Chambole-Pock \citep{chambolle2011first}\\
						%\cline{2-5}
						%& \multirow{2}{3cm}{$g_i$ are biaffine, $X$ double dualized}  &  \multirow{2}{*}{$\mathcal{T}_\bx(X)+\mathcal{T}_{\bu}(U)$}  %&$O(1/N)$\tnote{**} & PAPC \cite{drori2015209}\\
						%&& 
						%&\multirow{2}{*}{$O(1/N^2)$\tnote{**}} & Chambole-Pock \citep{chambolle2011first}\\
						% \hline
						\multirow{4}{3cm}{Compact $X$ and $U$, $\mathcal{L}$ Lipschitz continuous on $X\times U$} & Compact $X$ and $Z$, $g_i$ are Lipschitz continuous on $X\times Z$, Slater condition holds &  \multirow{4}{*}{$\mathcal{T}_{\bx}(X)+\mathcal{T}_{\bu}(U)$} &  \multirow{4}{*}{$O(1/\sqrt{N})$} & SPSG / Arrow-Hurwicz-Uzawa method  \citep{auslender2009projected,nedic2009subgradient}\\
						\hline
						\multirow{2}{*}{$\mathcal{L}(\bx,\bu)$ biaffine} & \multirow{2}{*}{$g_i$ are biaffine} & \multirow{2}{*}{$\mathcal{T}_\bx(X)+\mathcal{T}_{\bu}(U)$} & \multirow{2}{*}{$O(1/{N})$} & Chambole-Pock \citep{chambolle2011first}\\ 
						\hline
						\multirow{9}{3cm}{Compact $X$ and $U$, and  $\mathcal{L}(\bx,\bu)$ has a jointly Lipschitz continuous gradient on $X\times U$} &  \multirow{9}{3cm}{Compact $X$ and $Z$, and $\lambda_ig_i(\bx,\tilde{\bz}_i/\lambda_i)$ have Lipschitz continuous gradient in both $\bx$ and $u_i=(\tilde{\bz}_i,\lambda_i)$} &   \multirow{5}{*}{$2(\mathcal{T}_{\bx}(X)+\mathcal{T}_{\bu}(U))$}  & \multirow{5}{*}{$O(1/N)$} & {Extra gradient / Mirror-prox \citep{korpelevich1977extragradient,tseng1991applications,nemirovski2004prox},\newline  Dual extrapolation \citep{nesterov2007dual}}
						\\
						\cline{3-5}
						&&\multirow{4}{*}{$\mathcal{T}_{\bx}(X)+\mathcal{T}_{\bu}(U)$}&\multirow{4}{*}{$O(1/N)\tnote{**}$}&Forward-Reflected Backward / Optimmistic mirror descent
						\citep{Malitsky2020,alacaoglu2021forward}\\
						\bottomrule
				\end{tabular}}
				\footnotesize
				\begin{tablenotes}
					\item[*] $\mathcal{T}_{\by}(C)$ indicate the computational cost of computing the (sub)gradient of $\mathcal{L}$ with respect to variable $\by$ and performing an orthogonal projection on set $C$.
					%\item[**] Iterates not in $X$, converge to $X$ with same rate.
					\item[**] This rate was proven without the compactness assumption for a weaker result.
				\end{tablenotes}
			\end{threeparttable}
		\end{center}
	\end{table}
	
	In the first setting, we make similar assumptions to \cite{ho2018online}, \ie that for all $i\in[m]$, $g_i(\cdot,\cdot)$ are continuous real-valued functions, and $X$ and $Z_i$ are compact sets. Under this setting we require to use more general saddle-point algorithms, such as the one based on the subgradients of the functions, as the method discussed in \shimrit{\cite[Chapter 9]{LectureNem94} and analyzed in \cite{nedic2009subgradient,auslender2009projected} (SGSP), which achieves a $O(1/\sqrt{N})$} rate of convergence, and can be easily extended to the Bregman setting (Mirror Descent). 
	
	In the second setting considered, we treat $g_i(\cdot,\cdot)$ that are biaffine. It is established that in this case, a first-order algorithm can not obtain a better \shimrit{iteration complexity} than $O(1/\epsilon)$ \cite{ouyang2021lower}, and thus we restrict ourselves to algorithms which achieve this optimal rate. Although the above mentioned algorithms for smooth functions fit this setting, they require additional assumptions such as compactness of the sets and multiple gradient and projection steps per iteration. \shimrit{Therefore, we focus on the algorithm presented in \cite{chambolle2011first} (CP),} which is less restrictive and cheaper computationally, as it only requires one gradient computation and one projection on the sets per iteration. \shimrit{This algorithm can be used to obtain a sequence contained in $X$ with the gap converging at a rate of $O(1/N)$.}
	
	To address the second research question, we need to deal with the fact that saddle-point algorithms give convergence results in terms of ergodic duality gap. As we are interested in feasibility and optimality convergence rates for problem~\eqref{eq:basic_problem}, we defined a general notion of \emph{Ergodically Bounded (EB) algorithms}. We then that applying an EB algorithm to the saddle point formulation of~\eqref{eq:basic_problem} yields the same feasibility and optimality convergence rates as the algorithm itself. This definition allowed us to unify the analysis, as it is sufficient to show that a given algorithm, just like the two we consider, is an EB algorithm. Thus, the framework laid out in this paper can serve as a basis for analysis of other saddle point-algorithms in this context. 
	
	\shimrit{We note that the choice of algorithms for these two settings is somewhat arbitrary, since the aim of this paper is to illustrate the power of using the Lagrangian reformulation, and the general technique to convert a given saddle-point algorithm's gap rate of convergence to rates of obtaining feasibility and optimality for the robust problem \eqref{eq:basic_problem}.}
	Thus, a great deal of this paper is devoted to technical but important issues in applying the above mentioned algorithms in these settings. Specifically, we detailed sufficient and realistic assumptions needed to be satisfied by the robust problem \eqref{eq.robust.problem}, and derive bounds required for either the application of the algorithm or for showing that the algorithm assumptions are indeed satisfied. Moreover, to show the practicality of the suggested algorithms, we prove that when $Z^i$ are projection-friendly sets one may also find the projection onto the lifted space relatively easily, either analytically or by using bi-section. Finally, we show that the saddle-point formulation actually allows for more flexibility and enables to tackle problems where either $Z^i$ or $X$ are more complicated by using splitting techniques, where these problems prove to be more challenging for the previously suggested methods.
	
	%\krzysztof{We specifically} demonstrate that if the problem satisfies the Slater-type condition, it can be solved using a subgradient saddle-point algorithm (SGSP) of \cite{nedic2009subgradient} which requires $\mathcal{O}(1/ \epsilon^2)$ iterations to obtain an $\epsilon$-feasible and $\epsilon$-optimal solution. For the special case where $g_i(\cdot, \cdot)$ are biaffine, we show that the problem can be solved using the proximal alternating predictor-corrector (PAPC) algorithm of \cite{drori2015209}, which takes $\mathcal{O}(1/ \epsilon)$ iterations. Moreover, we show that when $Z^i$ are projection- or proximal-friendly one may also find the projection onto the lifted space relatively easily, either analytically or by using bi-section. Finally, we show that the saddle-point formulation actually allows for more flexibility and enables to tackle problems where either $Z^i$ or $X$ are more complicated by using splitting techniques, where these problems prove to be more challenging for the previously suggested methods.
	\vspace{5pt}
	\noindent{\bf{Paper Structure.}} The remainder of the paper is structured as follows. In Section~\ref{sec:setting}, we introduce the problem we solve along with the corresponding assumptions and the Lagrangian saddle-point reformulation. In Section~\ref{sec:algorithms}, we introduce the two algorithms for the case of simple uncertainty sets and state the corresponding convergence results. In Section~\ref{sec:generalizations}, we present the convergence analysis of the algorithms for a generalized problem form through a unified framework of showing that both are ergodically bounded (EB) algorithms. In Section~\ref{sec:numerics}, we compare the performance of our SGSP approach to the online first-order approaches of \cite{ho2018online} and the adversarial approach on randomly sampled robust quadratic optimization problems with and without constraints. Section~\ref{sec:conclusions} concludes the paper. All proofs not given in the body are given in the Supplementary Material.
	
	\vspace{5pt}
	\noindent{\bf{Notation}}
	Throughout the paper we use bolded small letters $\bx$, $\bz$ for vectors, and bolded capital letters $\bbp$ for matrices, and capital letters for sets. For any $k\in\mathbb{N}$, we use shorthand notation $[k]$ to denote the  set of indexes $\{1, 2,\ldots, k \}$. Unless specified otherwise, $\| \cdot \|$ refers to the Euclidean norm.
	
	\section{Problem setting and assumptions} \label{sec:setting}
	\subsection{Introduction}
	In this paper, we consider the following general RO problem.
	\begin{align}
		\min\limits_{\bx \in X} \ &  \bc^\top\bx \label{eq.robust.problem} \\
		\text{s.t.} \ &  f_i(\bx)\coloneqq \max\limits_{\bz_i \in Z^i } g_i(\bx,\bz_i)\leq 0 && i \in [m] \nonumber \\
		& \bba \bx = \bb \nonumber
	\end{align}
	where $X\subseteq \real^n$ is a closed and convex set, for all $i \in [m]$ the set $Z^i\subset \mathbb{R}^{d_i}$ is a convex and compact set and without loss of generality $\bzero\in Z^i$, $g_i(\cdot,\bz_i):X\rightarrow \real$ is convex for any fixed $\bz_i\in Z^i$, $g_i(\bx,\cdot):Z^i\rightarrow \real$ is concave for any fixed $\bx\in X$, and $\bba\in\real^{r\times n}$, $\bb\in \real^r$.  Note that, contrary to Problem~\ref{eq:basic_problem}, here we allow to separate some of the affine constraints that may be involved in the definition of the domain of $\bx$ from the set $X$. 
	
	Note that formulation \eqref{eq.robust.problem} also encompasses robust problems involving uncertainty in the objective function, since such problems can be transformed to form \eqref{eq.robust.problem} using the epigraph formulation of the objective. Thus, formulation \eqref{eq.robust.problem} is general and includes many useful problems (\emph{c.f.}, \cite{bertsimas2011theory,gabrel2014recent}). 
	
	Functions $f_i(\bx)$ are known as the robust counterpart formulation of the robust constraint
	$$
	g_i(\bx,\bz_i)\leq 0, \;\forall \bz_i\in Z^i.
	$$
	Note that while $f_i(\bx)$ are convex functions of $\bx$ (as a maximum of convex functions), they are not necessarily easily representable due to their implicit formulation as maxima.
	
	Our aim is to solve \eqref{eq.robust.problem} through its saddle point Lagrangian formulation. For this formulation to be well-defined, we make three standard assumptions.
	\begin{assumption}\label{ass:Optimal} 
		Problem~\eqref{eq.robust.problem} has an optimal solution.
	\end{assumption}
	
	\begin{assumption}\label{ass:Slater}
		There exists $\hat{\bx}\in\text{int}(X)$ such that $\bba \hat{\bx} = \bb$ and there exists a $\epsilon_{\hat{\bx}} > 0$ such that $\hat{\bx}+\by\in X$ and $f_{i}(\hat{\bx}+\by) < 0$ for all $\norm{\by}\leq \epsilon_{\hat{\bx}}$.
	\end{assumption}
	
	\begin{assumption} \label{ass:full.rank.A}
		Matrix $\bba$ has full row rank.
	\end{assumption}
	With respect to Assumption~\ref{ass:Optimal}, if the problem does not have an optimal solution there are three options: (i) it is infeasible, in which case the uncertainty sets defined for the problem may be too large, (ii) it is unbounded, \ie, it might not be constrained enough, or (iii) it is bounded but the optimal solution is not attained, in which case we can restrict the set $X$ to a subset containing $\epsilon$-optimal solutions of the original problem. Assumption~\ref{ass:Slater}, known as the Slater condition, ensures that the problem is stable, \ie slight perturbations in the feasible set do not make the problem infeasible. Assumption~\ref{ass:full.rank.A} states that there are no redundant equality constraints.
	
	In order to solve problem~\eqref{eq.robust.problem} we consider its Lagrangian:
	\begin{equation}
		L(\bx,(\blambda, \bw))\equiv \bc^\top\bx +\sum_{i=1}^m\lambda_if_i(\bx) + \bw^\top (\bba \bx - \bb) \label{eq:basic_saddle_point}
	\end{equation}
	Specifically, we are interested in saddle points of function $L(\cdot,\cdot)$ , \ie points $(\bx^*,\blambda^*)$ which satisfy
	$$
	L(\bx^*,(\blambda, \bw))\leq L(\bx^*,(\blambda^*, \bw^*))\leq L(\bx,(\blambda^*, \bw^*)), \quad\forall \bx\in X, \lambda\in\real^m_+, \bw \in \mathbb{R}^r.
	$$
	Under Assumptions~\ref{ass:Optimal} and \ref{ass:Slater}, the Lagrangian function $L(\cdot,\cdot)$ has a saddle point, and $(\bx^*, (\blambda^*, \bw^*))$ is a saddle point of $L(\cdot,\cdot)$ if and only if $\bx^*$ and $(\blambda^*, \bw^*)$ are optimal solutions to the primal and dual problems respectively.
	Thus, instead of solving problem~\eqref{eq.robust.problem}, we want to find a solution to
	\begin{align}\label{eq:saddle_point_simple}
		\pinf_{\bx\in X}\sup_{\blambda\in \real^m_+, \bw \in \mathbb{R}^r} L(\bx,(\blambda, \bw)).
	\end{align}
	This reformulation eliminates the constraints and renders the problem as a saddle point one, enabling the use of various first-order methods. However, such methods require computing at each iteration not only the functions $f_i(\bx)$, but also their sub-gradients or proximal operators, which may be challenging with the implicitly-defined $f_i$ functions. We will therefore consider an alternative formulation where our goal is to work with functions $g_i$, and show how to obtain an $\epsilon$-optimal and $\epsilon$-feasible solution of problem~\eqref{eq.robust.problem}.
	\subsection{Conversion to a convex-concave saddle point problem}
	We begin our transformation of problem \eqref{eq:saddle_point_simple} plugging in the explicit definitions of of functions $f_i$, using functions $g_i$, as follows:
	\begin{align}
		\sup_{\blambda\in\real^m_+, \bw \in \mathbb{R}^r} \pinf_{\bx\in X} L(\bx,(\blambda, \bw))&= \sup_{\blambda\in\real^m_+, \bw \in \mathbb{R}^r}\pinf_{\bx\in X}\bc^\top\bx +\sum_{i=1}^m\lambda_i\sup_{\bz_i\in Z^i} g_i(\bx,\bz_i) + \bw^\top(\bba \bx - \bb) \nonumber\\
		&=\sup_{\blambda\in\real^m_+, \bw \in \mathbb{R}^r}\pinf_{\bx\in X}\max_{\bz_i\in Z^i,\;i \in [m]}\bc^\top\bx +\sum_{i=1}^m\lambda_i g_i(\bx,\bz_i) + \bw^\top(\bba \bx - \bb) \nonumber \\
		&=\sup_{\blambda\in\real^m_+, \bw \in \mathbb{R}^r}\sup_{\bz_i\in Z^i,\;i \in [m]}\pinf_{\bx\in X}\bc^\top\bx +\sum_{i=1}^m\lambda_i g_i(\bx,\bz_i) + \bw^\top(\bba \bx - \bb) \label{eq:nonconvex_saddle}
	\end{align}
	where the equalities follow from Sion's Theorem, the fact $Z^i$ are convex and compact, $X$ is convex, and $g_i$ are convex-concave. However, the resulting saddle point problem \eqref{eq:nonconvex_saddle} is over a function which is convex in $\bx$ but is not jointly concave in $\bz=(\bz_1,\ldots,\bz_m)$ and $\blambda=(\lambda_1,\ldots,\lambda_m)$. Since convergence results for saddle point algorithms typically require a convex-concave structure, we need to reformulate the problem to achieve such a structure. For this, we will use a change of variables $\tilde{\bz}_i=\lambda_i \bz_i$, and inversely 
	$$
	\bz_i=\begin{cases}\frac{\tilde{\bz}_i}{\lambda_i}, & \lambda_i>0,\\
		\bzero, &\text{otherwise}.
	\end{cases}
	$$
	Using this definition we have $\lambda_i g_i(\bx,\bz_i)=\lambda_i g_i(\bx,\tilde{\bz}_i / \lambda_i)$, and since $-g_i(\bx,\cdot)$ is convex for every $\bx$, $-\lambda_i g_i(\bx, \tilde{\bz}_i / \lambda_i)$ is jointly convex in $\bu_i=(\tilde{\bz}_i,\lambda_i)$ for every $\bx$, as a \emph{perspective} of a convex function \cite[Proposition 8.23]{bauschke2011convex}. Moreover,  $\lambda_i g_i(\bx,\tilde{\bz}_i / \lambda_i)$ is continuous for all $\bu_i\in U^i,$ where $U^i=\{\bu_i=(\tilde{\bz}_i,\lambda_i): \tilde{\bz}_i\in \lambda_iZ^i , \lambda_i\geq 0\}$, obtaining a value of $0$ whenever $\lambda_i=0$. Defining $\bu=(\bu_1,\ldots,\bu_m)$ and the set $U=U^1\times\ldots\times U^m$, 
	we have
	\begin{align}\label{eq:saddle_point_lifted}
		\sup_{\bu\in U, \bw \in \mathbb{R}^r}\pinf_{\bx\in X} \bar{L}(\bx,(\bu,\bw))& \coloneqq \sup_{\bu\in U, \bw \in \mathbb{R}^r}\pinf_{\bx\in X} \bc^\top\bx+\sum_{i=1}^m \lambda_i g_i \left( \bx,\frac{\tilde{\bz}_i}{\lambda_i} \right) + \bw^\top (\bba \bx - \bb).
	\end{align}
	For ease of notation, in henceforth we denote the perspective version of $g_i$ as $\tilde{g}_i(\bx,\bu_i)\equiv \lambda_i g_i(\bx, \tilde{\bz}_i / \lambda_i)$ where $\bu_i=( \tilde{\bz}_i,\lambda_i)$.
	
	The following result shows that solving \eqref{eq:saddle_point_lifted} is sufficient for solving \eqref{eq:saddle_point_simple}, \ie, the saddle points of $\bar{L}$ can be reduced to those of $L$.
	\begin{proposition}\label{prop:eq_saddle_points}
		Let $({\bx}^*,({\bu}^*, \bw^*))\in X\times U \times \mathbb{R}^r$, where $\bu^*=(\bu^*_1,\ldots,\bu^*_m)$ and $\bu_i^*=(\tilde{\bz}_i^*,\lambda^*_i)$ for $i \in [m]$. Then $({\bx}^*,({\bu}^*, \bw^*)) \in X\times U$ is a saddle point of $\bar{L}$ over $X\times U\times \mathbb{R}^r$ if and only if $(\bx^*,(\blambda^*, \bw^*))$ is a saddle point of $L$ over $X\times\real^m_+\times \mathbb{R}^r$.
	\end{proposition}
	
	\section{Algorithms for solving saddle point formulation} \label{sec:algorithms}
	With the robust problem \eqref{eq.robust.problem} in the desired convex-concave structure \eqref{eq:saddle_point_lifted}, we move towards introducing two algorithms for solving it. As both algorithms will require boundedness of the optimal dual solution $(\blambda^*,\bw^*)$, either for running or for obtaining the convergence guarantees, we begin with stating the result that uses Assumptions~\ref{ass:Optimal}-\ref{ass:full.rank.A} to provide such bounds.
	\begin{proposition}\label{prop:bounded_lambda}
		Let Assumption~\ref{ass:Optimal} hold, let $\hat{\bx}$ be the point satisfying Assumption~\ref{ass:Slater}, let $(\bx^*,(\bu^*, \bw^*))$ be a saddle point of $\bar{L}$ on the set $X\times U \times \mathbb{R}^r$, and let $\ubar{v}$ be a strict lower bound on the optimal value of problem~\eqref{eq.robust.problem}. Then, $\bu^*=(\bu^*_1,\ldots, \bu^*_m)$, where $\bu_i^*=(\tilde{\bz}^*_i,\lambda^*_i)$, and $\bw^*$ satisfy
		\begin{align*}
			\lambda^*_{i} & \leq \bar{\lambda} := \frac{\bc^\top\hat{\bx}-\ubar{v}}{- \max_{i \in [m]} f_i(\hat{\bx})}\\
			\| \bu^*_{i} \| & \leq \bar{\lambda} \sqrt{1 + R_{i}^2} \\
			\| \bw^* \| & \leq R_\bw := \frac{1}{\sigma_{\min}(\bba)} \left( \frac{\bc^\top \hat{\bx} - \ubar{v}}{ \epsilon_{\hat{\bx}}} + \| \bc \| \right),
		\end{align*}
		where $R_{i} := \max_{\bz \in Z^i} \| \bz \|$ and $\sigma_{\min}(\bba) > 0$ is the smallest singular value of $\bba$.
	\end{proposition}
	\begin{proof}{Proof.}
		{\bf Boundedness of $\lambda^*_{i}$}. Let $\bx^*$ be an optimal solution to \eqref{eq.robust.problem}. Let $\hat{\bx}\in X$ be a Slater point. We have
		\begin{align*}
			\bc^\top\bx^* %&=\min_{\bx\in X} \bc^\top\bx + \sum_{i=1}^m \lambda^*_if_i(\bx)+(\bw^*)^\top(\bba\bx-\bb) \\
			& \leq \bc^\top\hat{\bx}+\sum_{i=1}^m \lambda^*_if_i(\hat{\bx})+(\bw^*)^\top(\bba\hat{\bx}-\bb)\\
			&=\bc^\top\hat{\bx}+\sum_{i=1}^m \lambda^*_if_i(\hat{\bx}) \nonumber \\
			& \leq \bc^\top\hat{\bx}+\norm{\blambda^*}_1\max_{i\in[m]} f_i(\hat{\bx})
		\end{align*}
		whence
		$$
		\norm{\blambda^*}_1\leq \frac{\bc^\top(\hat{\bx}-\bx^*)}{- \max_{i \in [m]} f_i(\hat{\bx})}=\bar{\lambda}.
		$$
		Thus, we have that $\lambda^*_{i}\leq \norm{\blambda^*}_1\leq \bar{\lambda}$ for all $i\in[m]$.
		
		{\bf Boundedness of $\bu_{i}^*$}. Since $\bu^*_{i} = (\tilde{\bz}^*_{i}, \lambda^*_{i}) \in U^{i}$ we have that $\tilde{\bz}^*_i\in \lambda^*_i Z^i$, by the Cauchy-Schwartz inequality we have that $\norm{\tilde{\bz}_i^*}\leq |\lambda^*_i|R_i\leq \bar{\lambda}R_i$. Thus, $\norm{\bu^*_i}^*= \norm{\tilde{\bz}_i^*}^2+|\lambda_i^*|^2\leq  \bar{\lambda}^2(R_i^2+1)$.
		
		{\bf Boundedness of $\bw^*$} Consider the saddle point formulation:
		\begin{align*}
			\inf\limits_{\bx \in X} \sup\limits_{\bu_i \in U^i, \bw} \bc^\top \bx  + \bw^\top (\bba \bx - \bb) + \sum\limits_{i = 1}^m \tilde{g}_i(\bx, \bu_i) = \inf\limits_{\bx} \sup\limits_{\bu_i \in U^i, \bw} \bc^\top \bx  + \bw^\top (\bba \bx - \bb) + \sum\limits_{i = 1}^m \tilde{g}_i(\bx, \bu_i) + \delta_X(\bx)
		\end{align*}
		Assume $\bx^*$, $\bw^*$, $\bu_i^\ast$ are the saddle point of this problem. The optimality conditions of the problem are:
		\begin{align*}
			\bba \bx^* & = \bb \\
			\bc + \bba^\top \bw^* + \sum\limits_{i = 1}^m \balpha_i + \bbeta & = \bzero
		\end{align*}
		where $\balpha_i \in  \partial_{\bx} \tilde{g}_i(\bx^*, \bu_i^*)$ and $\bbeta \in  \partial_{\bx} \delta_X(\bx^*)$.  Let $\hat{\bx} \in \text{int}(X)$ be the Slater point so we have 
		\begin{equation} \label{eq.slater}	
			\hat{\bx} + \bkappa \in X, \quad \tilde{g}_{i}(\hat{\bx} + \bkappa, \bu_i^*) \leq 0 \ \qquad \ \forall \bkappa: \| \bkappa \| \leq \epsilon_{\hat{\bx}}
		\end{equation}
		Since $\balpha_i \in \partial_{\bx} \tilde{g}_i(\bx^*, \bu_i^*)$, we have that 
		$$
		%\begin{equation} \label{eq.subgradients.g.tilde}
		\tilde{g}_i(\bx^*, \bu_i^*) + \balpha_i^\top \left( \hat{\bx} + \bkappa - \bx^* \right) \leq \tilde{g}_i(\hat{\bx}+\bkappa, \bu_i^*)
		$$
		%\end{equation}
		so that
		$$
		\underbrace{\sum\limits_{i \in [m]} \tilde{g}_i(\bx^*, \bu_i^*)}_{= 0} + \sum\limits_{i \in [m]} \balpha_i^\top \left( \hat{\bx} + \bkappa - \bx^* \right) \leq \underbrace{\sum\limits_{i \in [m]} \tilde{g}_i(\hat{\bx} + \bkappa, \bu_i^*)}_{\leq 0}
		$$
		where the first sum is equal to zero due to complementary slackness (since $\bx^*$ and $\lambda^*$ are optimal $\tilde{g}_i(\bx^\ast, \bu_i^\ast) = \lambda_i^\ast g_i(\bx^\ast, \bz_i^\ast) = \lambda_i^\ast f_i(\bx^*)= 0$) and the second sum is nonpositive due to \eqref{eq.slater}. Since $\bbeta \in  \partial_{\bx} \delta_X(\bx^*)$, we also have that 
		$$
		\delta_X(\bx^*) + \bbeta^\top ( \hat{\bx} + \bkappa - \bx^* ) \leq  \delta_X(\hat{\bx} + \bkappa)=0,
		$$ 
		which, combined with the optimality conditions, implies
		$$
		0 \leq \left( \bc + \bba^\top \bw^* + \sum\limits_{i = 1}^m \balpha_i \right)^\top \left( \hat{\bx} + \bkappa - \bx^* \right) \leq \left( \bc + \bba^\top \bw^* \right)^\top \left( \hat{\bx} + \bkappa - \bx^* \right), \quad \forall \| \bkappa \| \leq \epsilon_{\hat{\bx}}.
		$$
		This gives us the property
		$$
		\epsilon_{\hat{\bx}} \left\| \bc + \bba^\top \bw^* \right\| \leq \left( \bc + \bba^\top \bw^* \right)^\top \left( \hat{\bx} - \bx^* \right)
		$$
		Because $\bba \bx^* = \bba \hat{\bx} = \bb$, we obtain $\epsilon_{\hat{\bx}} \left\| \bc + \bba^\top \bw^* \right\| \leq \bc^\top \left( \hat{\bx} - \bx^* \right) $. Since by Assumption~\ref{ass:full.rank.A}, $\bba$ has full row rank, by the reverse triangle inequality  we can bound $\bw^*$ as follows:
		$$
		\epsilon_{\hat{\bx}} \left( \sigma_{\text{min}}(\bba^\top) \left\| \bw^* \right\| -  \left\| \bc \right\| \right) \leq \epsilon_{\hat{\bx}} \left( \left\| \bba^\top \bw^* \right\| -  \left\| \bc \right\| \right) \leq \epsilon_{\hat{\bx}} \left\| \bc + \bba^\top \bw^* \right\| \leq \bc^\top \left( \hat{\bx} - \bx^* \right).
		$$
		\qedsymbol
	\end{proof}
	
	We now move to presenting the two algorithms, summarized in Table~\ref{tab:comparison.algorithms}. We will first state both for problem \eqref{eq:saddle_point_lifted} and give their convergence without proofs. In Section~\ref{sec:generalizations}, we shall prove the convergence of a generalized problem in a unified framework from which the `simple cases' will follow as straightforward corollaries.
	
	The first algorithm, presented in Section~\ref{sec:general}, applies to the case with the additional assumptions that $X$ is bounded and the functions $g_i$ have bounded subgradients over $X\times Z^i$. For this setting, \shimrit{an iteration complexity of $O(1/{\epsilon}^2)$ is attained,} similar to the one obtained by \cite{ho2018online} under almost identical assumptions. In Section~\ref{subsubsec:CP.biaffine}, we consider the case where $g_i(\bx,\bz_i)$ are biaffine functions (or can be transformed to this form), and show that in this setting we can obtain a superior \shimrit{iteration complexity of $O(1/\epsilon)$.} %\red{However, this algorithm does not guarantee that the iterates of $\bx$ reside in the set $X$, but rather the ergodic sequence converges to set $X$ at an $O(1/\epsilon)$ rate.}
	
	\begin{table}
		\TABLE
		{Comparison of the two saddle point algorithms. \label{tab:comparison.algorithms}}
		{\scriptsize 
			\begin{tabular}{|l|c|c|} \hline
				Algorithm & SGSP & \shimrit{CP} \\  \hline
				Domain $X$ & Bounded & No restriction \\ 
				Structure $g_i(\cdot)$ & Any & Biaffine or reducible to biaffine \\
				Optimality convergence rate & $\mathcal{O}(1/\sqrt{N})$ & $\mathcal{O}(1/N)$ \\
				Feasibility convergence rate & $\mathcal{O}(1/\sqrt{N})$ & $\mathcal{O}(1/N)$ \\
				Slater point needed to compute the stepsize & Yes & No \\  \hline
		\end{tabular}}
		{}
	\end{table}
	
	\subsection{Subgradient Saddle Point algorithm}\label{sec:general}
	In this section, we show how problem \eqref{eq:saddle_point_lifted} can be solved using the SGSP suggested in \cite{nedic2009subgradient}. This algorithm requires that the both primal and dual variables be contained in compact sets. Thus, we first need to replace each set $U^i$ by its compact counterpart $\tilde{U}^i=\{(\tilde{\bz}_i,\lambda_i)\in U^i: \lambda_i\leq \bar{\lambda}\}$. Indeed, by the property of $\lambda^*$ presented in Proposition~\ref{prop:bounded_lambda}, this restriction of $U^i$ would not change the set of saddle points of problem~\eqref{eq:saddle_point_lifted}. Similarly, we can restrict $\bw$ to reside in a set $W = \{ \bw \in \mathbb{R}^r: \ \norm{\bw} \leq R_\bw \}$. Using these new sets, the algorithm is as follows.
	\begin{center}
		\vspace{10pt}
		\fcolorbox{black}{gray!20}{
			\begin{minipage}{0.85\textwidth}
				\noindent \textbf{SGSP: SubGradient Algorithm for Saddle Point}\\
				\textbf{Input:} ${\tau>0}, {\theta_{i}>0}, \theta_{\bw} > 0$ for $i \in [m]$, and $N\in\mathbb{N}$\\
				\textbf{Initialization}. Initialize $\bx^0 \in X$ and $\bu^0_i \in \tilde{U}^i$, for $i \in [m]$, $\bw^0 \in W$. \\
				\textbf{General step:} For $k \in [N]$\\
				Compute subgradients $\bv^k_x\in \partial_\bx\bar{L}(\bx^k,(\bu^k, \bw^k))$, $\bv^k_{i}\in \partial_{\bu_i} \left(-\bar{L}(\bx^k,(\bu^k, \bw^k))\right)$, for all $i \in [m]$
				\begin{align*}
					\bx^{k+1}  &= P_X(\bx^{k} - \tau \bv_x^k), \\
					\bu^{k+1}_i &= P_{\tilde{U}^i}(\bu_i^{k} - \theta_i \bv_i^k) && i \in [m] \\
					\bw^{k+1} & = P_{W}(\bw^{k} + \theta_\bw (\bba \bx^{k} - \bb))
				\end{align*}
			\end{minipage}
		}
		\vspace{10pt}
	\end{center}
	Note that the algorithm can be applied whenever the projections over $\{\tilde{U}^i\}_{i\in[m]}$ and $X$ can be easily computed. We will see in Section~\ref{sec:generalizations}, that if any of the sets are intersections of multiple simpler sets, we can utilize splitting methods that enable the use of projections only on the components of the intersection. Moreover, note that at each iteration the steps for all $\bu_i$, $\bx$ and $\bw$ can be done in parallel.
	
	As in most algorithms solving saddle point problems, the SGSP algorithm's convergence is given in terms of the ergodic sequences, \ie denoting 
	$$
	\bar{\bx}^N=\frac{1}{N}\sum_{k=1}^N \bx^k, \quad \bar{\bu}^N=\frac{1}{N}\sum_{k=1}^N \bu^k, \quad \bar{\bw}^N=\frac{1}{N}\sum_{k=1}^N \bw^k
	$$
	the convergence result states the rate at which the sequence $\{\left(\bar{\bx}^N, \bar{\bu}^N, \bar{\bw}^N \right)\}_{N\in\mathbb{N}}$  converges to a saddle point. Here, we present the convergence in terms of the total constraint violation and the distance from the optimal value.
	\begin{theorem} \label{theorem.sgsp.simple.convergence}
		Let $\{\bx^k, \bu^k, \bw^k\}_{k\in\mathbb{N}}$ be the sequences generated by the SGSP algorithm with step sizes 
		\begin{align*}
			\tau\coloneqq\ \tilde{\tau} / \sqrt{N} ,\quad \theta_i\coloneqq \tilde{\theta}_i / \sqrt{N},\quad \theta_\bw \coloneqq \tilde{\theta}_\bw / \sqrt{N}
		\end{align*}
		for $i \in [m]$. Assume that $X$ is compact and define $R_\bx:=\max_{\bx\in X} \norm{\bx}$. Further assume there exist constants $G_\bx$, $G_i$ for $i \in [m]$ such that the subgradients $\{\bv_x^k\}_{k\in\mathbb{N}}$, $\{\bv_i^k\}_{k\in\mathbb{N}}$, $i \in [m]$ generated by the algorithm satisfy 
		$$  \norm{\bv^k_x}\leq G_\bx,\;\norm{\bv^k_i}\leq G_i,\;i \in [m],\quad \forall k\in \mathbb{N}.
		$$
		Then, we have the following feasibility and optimality convergence guarantees.
		\begin{align*}
			& \sum_{i=1}^m [f_i(\bar{\bx}^N)]_+ + \| \bba \bar{\bx}^N - \bb \|\leq  \nonumber \\
			&  \frac{\max\{2, \max_i \{ 1 + 4R_i^2\} \}}{2\sqrt{N}} \left( \frac{2 \max\{\norm{\bx^0}^2,R_\bx^2\}}{\tau} + \sum_{i=1}^m \frac{\max\{\bar{\lambda}+1, \lambda_i^0\}^2}{\theta_i} + \frac{\max \{ R_\bw + 1, \| \bw^0 \| \}^2}{\theta_\bw} + \phi \right), \nonumber
		\end{align*}
		and
		\begin{align*}
			& \left| \bc^\top(\bar{\bx}^N-\bx^*) \right| \leq  \nonumber \\
			&  \frac{\max\{2, \max_i \{ 1 + 4R_i^2\} \}}{2\sqrt{N}} \left(\frac{2 \max\{\norm{\bx^0}^2,R_\bx^2\}}{\tau} + \sum_{i=1}^m \frac{\max\{2\bar{\lambda}, \lambda_i^0\}^2}{\theta_i} + \frac{\max \{2R_\bw, \| \bw^0 \| \}^2}{\theta_\bw} + \phi \right), \nonumber
		\end{align*}
		where
		$$
		\phi := \tilde{\tau} G_\bx^2 + \sum\limits_{i=1}^m \tilde{\theta}_i G_i^2 + \tilde{\theta}_\bw G_\bw^2, \quad %\norm{\bv^k_\bw} := \norm{\bba \bx^k - \bb} \leq \
		G_\bw:=\norm{\bba}R_\bx +\norm{\bb}, \quad R_{i} := \max_{\bz \in Z^i} \| \bz \|.
		$$
	\end{theorem}
	Note that almost all conditions used in Theorem~\ref{theorem.sgsp.simple.convergence} are also needed to apply the online first-order (OFO) approach of \cite{ho2018online}. In fact, the two approaches give similar convergence results with a few differences:
	\begin{enumerate}
		\item \emph{Assumptions.} SGSP requires the existence of a Slater point $\hat{\bx}$ while the OFO does not. Secondly, OFO requires boundedness of the subgradients of $g(\bx,\cdot)$ while SGSP requires boundedness of the subgradients of its prespective function. In Section~\ref{sec:subgradients} we show that under mild assumptions on the problems, these requirements are equivalent.
		%\item \emph{Feasibility and Optimality Guarantees.} The OFO directly provides feasibility and optimality guarantees, while the SGSP provides a duality gap type result.  
		\item \emph{Implementation.} Since the OFO approach is meant to solve a feasibility, rather than an optimality, problem, it requires to perform a binary search to approximate the optimal value of \eqref{eq.robust.problem}. Thus, the number of needed iteration to obtain feasibility and optimality guaranties is increased by a factor of $\log(1/\epsilon)$. In the SGSP in turn, we do not need to perform bi-section, however we must find a Slater point $\hat{\bx}$, the values of $f_i(\hat{\bx})$ for $i \in [m]$, and $\epsilon_{\hat{\bx}}$, as well as a lower bound on the objective function $\ubar{v}$, which are needed to compute both $\bar{\lambda}$ and $R_{\bw}$. While in some cases it may be easy to find these quantities, in general it requires solving an auxiliary optimization problem, as we discuss in Section~\ref{sec:search.slater.point}.  
		\item \emph{Projections.} While OFO requires projections on sets $X$ and $Z^i$, SGSP requires projection onto $X$ and the lifted set $\tilde{U}^i$. In Section~\ref{sec:projection} we show that for standard simple sets $Z^i$ the projections onto $\tilde{U}^i$ can be simply computed. 
		\item \emph{Constants.} Although the convergence rate of both methods is $O(1/\epsilon^2)$, the constants obtained by the SGSP algorithm are worse then those of OFO, if the same first order method (subgradient/mirror decent) is used. 
	\end{enumerate}
	
	\subsection{\shimrit{Chambolle-Pock algorithm}}
	\label{subsubsec:CP.biaffine}
	In this section, we present an algorithm with a superior rate of convergence of $O(1/\epsilon)$ which does not require boundedness of $X$. This algorithm requires the additional assumption that the functions $g_i(\bx,\bz_i)$ have a biaffine form:
	\begin{equation} \label{eq:biaffine.g_i}
		g_i(\bx,\bz_i)=\bx^\top\bbq_i\bz_i+\bd_i^\top\bx+\bq_i^\top\bz_i+\gamma_i.
	\end{equation}
	and that the primal variable is not constrained. In Remark~\ref{rem:non.biaffine}, we show how more general problems of the form $g_i(\bx,\bz_i)\coloneqq\bh_i(\bx)^\top\bk_i(\bz_i)$ with convex and concave $\bh_i$ and $\bk_i$, respectively, can be reformulated to fit this case.
	
	To state the algorithm, we first simplify its form due to the biaffine structure. Indeed, under \eqref{eq:biaffine.g_i} function $\bar{L}$ reduces to
	\begin{align}
		\bar{L}(\bx,(\bu, \bw)) & \coloneqq \bc^\top\bx+\sum_{i=1}^m \left(\bx^\top\tilde{\bbq}_i\bu_i+\tilde{\bq}_i^\top\bu_i\right) + \bw^\top (\bba \bx - \bb) \nonumber \\
		& = \bc^\top\bx + \bx^\top \bba^\top \bw + \bx^\top\tilde{\bbq}\bu+\tilde{\bq}^\top\bu - \bb^\top \bw, \label{eq:basic.ChP.problem}
	\end{align}
	where, 
	\begin{equation}\label{eq:lin_new_param}
		\tilde{\bbq}_i=\begin{bmatrix}
			\bbq_i & \bd_i
		\end{bmatrix},\; \tilde{\bq}_i=\begin{bmatrix}
			\bq_i \\ \gamma_i
		\end{bmatrix},\; \tilde{\bbq}=(\tilde{\bbq}_1,\ldots,\tilde{\bbq}_m),\; \tilde{\bq}=(\tilde{\bq}_1^\top,\ldots,\tilde{\bq}_m^\top)^\top,\;\bbb=(\bba^\top,\tilde{\bbq}).\end{equation}
	
	\shimrit{We now state the CP algorithm of \cite{chambolle2011first} for solving problem~\eqref{eq:basic.ChP.problem}.}
	\begin{center}
		\vspace{10pt}
		\fcolorbox{black}{gray!20}{
			\begin{minipage}{0.7\textwidth}
				\shimrit{	\noindent \textbf{CP: Chambolle-Pock first-order primal-dual algorithm}\\
					\textbf{Input:} ${\tau>0}, \sigma >0, \tau\cdot\sigma\leq \frac{1}{\norm{\bbb}^2}$\\
					\textbf{Initialization}. Initialize $\bar{\bx}^0 \in \real^n$, $\by^0=(\bu^0, \bw^0) \in U \times \real^m$. \\
					\textbf{General step:} For $k \in [N]$
					\begin{align*}
						\bu^{k}_{i} & = P_{U_{i}} \left(  \bu^{k - 1}_{i} + \sigma (\tilde{\bq}_i +\tilde{\bbq}_i^T\bar{\bx}^{k-1})\right) && i \in [m] \\
						\bw^k & =  \bw^{k - 1} + \sigma (\bba \bar{\bx}^{k-1} - \bb)  \\
						\bx^k & = P_X \left( \bx^{k - 1} - \tau \left(\bc + \bba^\top \bw^{k} - \tilde{\bbq}_i \bu_{i}^k\right)  \right) \\
						\bar{\bx}^k&=2\bx^k-\bx^{k-1}
				\end{align*}}
		\end{minipage}}
		\vspace{10pt}
	\end{center}
	Thus, similarly to the SGSP algorithm, CP can be applied whenever the projections over $U^i$ and $X$ are easily computed. Similarly to the SGSP, convergence results for the CP are in terms of the ergodic sequence.
	\shimrit{\begin{theorem} \label{theorem.CP.simple.convergence}
			Let Assumptions~\ref{ass:Optimal} and \ref{ass:Slater} hold and let $\{\bx^k, \bw^k \}_{k \in \mathbb{N}}$ be the sequence generated by the CP algorithm with some $\tau>0, \sigma > 0$ satisfying $\tau\sigma\norm{\bbb}^2\leq 1$. Then, we have the following convergence guarantees:
			\begin{align*}
				& \sum_{i=1}^N [f_i(\bar{\bx}^N)]_+ + \norm{\bba \bar{\bx}^N - \bb} \leq \\
				&\frac{\max \{ 2, \max_i (1 + 4 R_i^2) \}}{2N} \left({\small \frac{{\max\{\norm{\bx^*},\norm{\bx^0}\}^2}}{\tau} + \sum_{i=1}^m\frac{\max\{\bar{\lambda}+1, \lambda_i^0\}^2}{\sigma} + \frac{\max \{ R_\bw + 1, \| \bw^0 \| \}^2}{\sigma} } \right),
			\end{align*}
			and
			\begin{align*}
				& \left| \bc^\top(\bar{\bx}^N-\bx^*) \right| \leq  \nonumber \\
				& \frac{\max \{ 2, \max_i (1 + 4 R_i^2) \}}{2N} \left({\small  \frac{{\max\{\norm{\bx^*},\norm{\bx^0}\}^2}}{\tau} + \sum_{i=1}^m \frac{\max\{2\bar{\lambda}, \lambda_i^0\}^2}{\sigma} + \frac{\max \{ 2 R_\bw, \| \bw^0 \| \}^2}{\sigma} }\right).
			\end{align*}
			% where $\phi= 4 \sum_{i=1}^m (s_i-1)\bar{\mu}_i^2\left( 1+\frac{1}{\epsilon_i} \right)^2 $ with $\bar{\mu}_i$ and $\epsilon_i$ are defined in Assumptions~\ref{ass:constraints_bounded_below} and \ref{ass:zero_in_int_Zi}, respectively.
	\end{theorem}}
	
	\subsection{Numerical implementation of saddle point algorithms}
	In this section, we discuss some technical aspects related to the implementation of the above algorithms. These aspects relate to (i) finding a Slater point for the SGSP algorithm using SGSP algorithm of an auxiliary problem, (ii) computing subgradients of the lifted functions $\tilde{g}_i$ from the subgradients of the original functions $g_i$, (iii) computing projections onto the lifted sets $U^i$.
	\subsubsection{Search for a Slater point in the SGSP algorithm} \label{sec:search.slater.point}
	To run the SGSP algorithm, we require parameter values dependent on the features of a Slater point of the problem. Therefore, before SGSP is run, we first need to identify a Slater point and plug the appropriate values into the algorithm. To find such a point, we need to solve the following optimization problem 
	\begin{align}
		\min\limits_{\bx, t} \ & t \label{eq.problem.slater.search} \\
		\text{s.t.} \ & \max\limits_{\bz_i \in Z^i} g_i(\bx,\bz_i) \leq t && \forall i \in [m] \nonumber \\
		& \bba \bx = \bb \nonumber \\
		& \bx \in X \nonumber
	\end{align}
	This problem satisfies the Slater condition for any $\bx^0 \in \text{int}(X)$, $\bba \bx^0 = \bb$, and $t=t^0+\delta$ such that 
	$$
	{t}^0 = \max_{i \in [m]} \max\limits_{\bz_i \in Z^i} g_i({\bx}^0,\bz_i).
	$$
	For that reason, we can first apply the SGSP algorithm to this problem by transforming it to a saddle point form. To keep the domain of $(\bx, t)$ compact, we can restrict $t$  to belong to the interval $[-1,\bar{t}]$ where $\bar{t}={t}^0+\delta$.
	
	If we assume that the subgradients in the SGSP algorithm for the original problem~\eqref{eq.robust.problem} are bounded, then it also holds for \eqref{eq.problem.slater.search}. Thus, one can run the SGSP algorithm for \eqref{eq.problem.slater.search}, knowing that it will converge to the optimal value of $t$.
	A complication is the moment at which the algorithm stops. Assuming that the original problem satisfies the Slater condition with constant $\epsilon$ we would like to stop when both optimality and feasibility conditions are satisfied with $\frac{\epsilon}{3}$ thus ensuring that the $(\hat{\bx},\hat{t})$ obtained by the procedure satisfies
	$$\max\limits_{\bz_i \in Z^i} g_i(\hat{\bx},\bz_i)\leq \hat{t}+\frac{\epsilon}{3}\leq t^*+\frac{\epsilon}{3}+\frac{\epsilon}{3}=-\frac{\epsilon}{3}.$$
	However, in practice, $\epsilon$ can be unknown in advance, and therefore, we construct a search procedure as follows:
	\begin{center}
		\vspace{10pt}
		\fcolorbox{black}{gray!20}{
			\begin{minipage}{0.7\textwidth}
				\noindent \textbf{Slater Point Search}\\
				\textbf{Input:} ${\tau>0}, {\theta_i>0}$ for $i \in [m]$, $\delta>0$, $K=2$\\
				\textbf{Initialization}. Initialize $\bx^0 \in \text{int}(X)$, $t^0=\max_{i \in [m]} f_i({\bx}^0)+\delta$ and $\bu^0_i \in U^i$, for $i \in [m]$. \\
				\textbf{General step:} For $k \in 1,2,\ldots$
				\begin{enumerate}
					\item Run SGSP for the Lagrangian form of problem~\eqref{eq.problem.slater.search} starting from point $(\bx^{k-1},t^{k-1})$ and $\bu^{k-1}$ for $K$ iterations, and obtain the ergodic values $(\bx^k,t^k)$ and $\bu^k$.
					\item Update $t^k=\max_{i \in [m]} f_i({\bx}^k)$.
					\item If $t^k<0$ stop and return the Slater point $\hat{\bx}=\bx^k$ and the Slater value $\epsilon=-t^k$. Otherwise, update $\bar{t}=\min\{\bar{t},t^k+\delta\}$, $k=k+1$ and $K=2K$.
				\end{enumerate}
			\end{minipage}
		}
		\vspace{10pt}	
	\end{center} 
	As stated above, this procedure is guaranteed to converge. Moreover, since $\bx^0 \in \text{int}(X)$ then by properties of convex sets we will also obtain that for each $k$ the iterate ${\bx}^k\in \text{int}(X)$, as the average of points in $X$ with one of them in the interior.
	
	\subsubsection{Determining the subgradients of $\tilde{g}_i$ from subgradients of $g_i$}\label{sec:subgradients}
	To prove the convergence of SGSP we require that the subgradients of the perspective function $\tilde{g}_i$ are bounded. In this section, we show that the subgradients of $\tilde{g}_i$ can be easily derived from subgradients of $g_i$. Moreover, we will also show that under the following mild assumptions the boundedness of the subgradients of $\tilde{g}_i$ follows from the boundedness of the subgradients of $g_i$.
	\begin{assumption}\label{ass:zero_in_int_Zi}
		For every $Z^i$ there exists a constant $\epsilon_i>0$ such that $B(\bzero,\epsilon_i)\subseteq Z^i$.
	\end{assumption}
	\begin{assumption}\label{ass:constraints_bounded_below}
		There exists a constant $\bar{\mu}_i>0$ such that  $-g_i(\bx,\bzero)\leq \bar{\mu}_i$ for any feasible $\bx$ of problem~\eqref{eq.robust.problem}.
	\end{assumption}
	Assumption~\ref{ass:zero_in_int_Zi} is a standard RO assumption that the uncertainty set is full dimensional, 
	note that this is always true, since under a linear transformation we can always reduce the dimension of $Z^i$. Assumption~\ref{ass:constraints_bounded_below} states the following: the feasible set of the robust problem does not contain rays that make any of the constraints of the `nominal problem' (where $\bz = 0$) be arbitrarily satisfied, \ie make $g_i(\bx,\bzero)$ arbitrarily small. We note that this assumption can be verified by checking the following sufficient condition: 
	$$
	\min\left\{g_i(\bx,\bzero):\;\bx\in X,\; g_j(\bx,\bzero)\leq 0, \forall j\in[m]\setminus\{i\}\right\}>-\infty,\;\forall i\in [m],
	$$
	which can, in turn, be shown to hold by solving $m$ convex (non-robust) optimization problems.
	
	In the course of the SGSP algorithm, one needs to compute the subgradients of the perspective functions $\tilde{g}_i(\bx, \bu)$. Ideally, this is done using the subgradients  of the original functions $g_i(\bx, \bz)$, which should typically be available. The following lemma provides a `recipe' for doing exactly this. The `recipe' is based on the convex analysis results for perspective functions of \cite{combettes2018perspective}.
	\begin{lemma}\label{lem:subdif_def}
		Let $\bx\in X$ and $\bu_i=(\tilde{\bz}_i,\lambda_i)\in U^i$ for all $i \in [m]$. 
		\begin{itemize}
			\item[(i)] If $\bv_x\in \partial_\bx\bar{L}(\bx,\bu)$, $\bv_{i}\in \partial_{\bu_i} \left(-\bar{L}(\bx,\bu)\right)$, then they are of the following form
			\begin{align}
				&\bv_{\bx}=\bc+\bd_x\label{eq:d_x}\\
				&\bv_{i}=\begin{cases} \left(\bd_i,-g_{i}\left(\bx,\frac{\tilde{\bz}_i}{\lambda_i}\right)-\frac{\tilde{\bz}_i^\top\bd_i}{\lambda_i}\right),  & \lambda_i>0,\\
					(\bd_i,\phi_i), & \text{otherwise}
				\end{cases}\label{eq:d_z}
			\end{align}
			where $\bd_x\in \sum_{i:\lambda_i>0} \lambda_i\partial_\bx g_{i}\left(\bx,\frac{\tilde{\bz}_i}{\lambda_i}\right)$, $\bd_i\in \partial_{\bz_i}\left(-g_{i}\left(\bx,\frac{\tilde{\bz}_i}{\lambda_i}\right)\right)$ for all $i \in [m]$ such that $\lambda_i>0$, and  $\bd_i\in\cup_{\bz_i\in \real^{d_i}} \left( \partial_{\bz_i} (-g_{i})\left(\bx,\bz_i\right)\right) $, $\phi_i+ (-g_i)^*(\bx,\bd_i)\leq 0$ for all $i \in [m]$ such that $\lambda_i=0$.
			\item[(ii)] Let $\bz_i=\tilde{\bz}_i/\lambda_i$ if $\lambda_i>0$ and $\bz_i=\bzero$ otherwise, and let 
			$\tilde{\bd}_{\bx,i}\in \partial_{\bx} g_i(\bx,\bz_i)$ and $\tilde{\bd}_{\bz,i} \in \partial_{\bz_i } \left(-g_i(\bx,\bz_i)\right)$ for all $i \in [m]$. Then, 
			\begin{align*}\bv_x=\bc+\sum_{i=1}^m\lambda_i\bd_{\bx,i} \in \partial_\bx\bar{L}(\bx,\bu) \text{ and } \bv_i=(\bd_{\bz,i},-g_i(\bx,\bz_i)-(\bz_i)^\top\bd_{\bz,i})\in\partial_{\bu_i} \left(-\bar{L}(\bx,\bu)\right).\end{align*}
		\end{itemize}
	\end{lemma}
	
	\subsubsection{Computing the projections} \label{sec:projection}
	In this section, we discuss the projections needed to apply SGSP and CP. Specifically, we will discuss how to project on set $U^i$ and $\tilde{U}^i$, which are more complicated than $Z^i$. We shall assume that the projection on set $Z^i$ is simple and formulate the projection on $U^i$ in its terms. Note that, in this paper, we use $U^i$
	\begin{equation} 
		U^i=\{(\tilde{\bz}_i,\lambda_i):\tilde{\bz}_i\leq \lambda_iZ^i,\lambda_i\geq 0\}.\label{eq:Ui_firstform}\end{equation}
	for the CP setting,
	and $\tilde{U}^i$
	\begin{equation} \tilde{U}^i=\{(\tilde{\bz}_i,\lambda_i):\tilde{\bz}_i\leq \lambda_iZ^i,\;0\leq \lambda_i\leq \bar{\lambda}\},\label{eq:Ui_secondform}\end{equation}
	which uses an additional upper bound on $\lambda_i$ for the more general SGSP setting. We will start by showing a general way to compute the projection over $U^i$.
	
	\begin{proposition}\label{prop:proj_firstformUi}
		The projection of $\bu_i=(\tilde{\bz}_i,\lambda_i)\in \real^{d_i+1}$ on the set $U^i$ defined by \eqref{eq:Ui_firstform} is given by
		$$P_{U^i}(\bu_i)=\begin{cases}
			\bu_i &\bu_i\in U^i\\
			\left(\mu^*P_{Z^i}\left(\frac{\tilde{\bz}_i}{\mu^*}\right),\mu^*\right)&\bu_i\notin U^i,\sigma_{Z^i}(\tilde{\bz}_i)>-\lambda_i\\
			\bzero &\text{otherwise}
		\end{cases}$$
		where $\mu^*>0$ is the unique solution of 
		$$\mu\norm{P_{Z^i}\left(\frac{\tilde{\bz}_i}{\mu}\right)}^2-\tilde{\bz}_i^\top P_{Z^i}\left(\frac{\tilde{\bz}_i}{\mu}\right)+\mu-\lambda_i=0.$$
	\end{proposition}
	The same technique used to prove Proposition~\ref{prop:proj_firstformUi} can be used for the proof of projection over $\tilde{U}^i$.
	\begin{corollary}\label{cor:proj_secondformUi}
		The projection of $\bu_i=(\tilde{\bz}_i,\lambda_i)\in \real^{d_i+1}$ on the set $\tilde{U}^i$ defined by \eqref{eq:Ui_secondform} is given by
		$$P_{\tilde{U}^i}(\bu_i)=\left(\mu^* P_{Z^i}\left(\frac{\tilde{\bz}_i}{\mu^*}\right),\mu^*\right)$$%\begin{cases}
		%\bu_i &\bu_i\in U^i\\
		%(\tilde{\bz}_i,\bar{\lambda})&\lambda_i>\bar{\lambda},\; \tilde{\bz}_i\in \bar{\lambda}Z^i\\
		%\left(\bar{\lambda} P_{Z^i}\left(\frac{\tilde{\bz}_i}{\bar{\lambda}}\right),\bar{\lambda}\right)&{\psi}'_i(\bar{\lambda})<0\\
		%\left(\mu^* P_{Z^i}\left(\frac{\tilde{\bz}_i}{\mu^*}\right),\mu^*\right)&\sigma_{Z^i}(\tilde{\bz}_i)>-\lambda,\;{\psi}'_i(\bar{\lambda})>0,\; {\psi}'_i(\mu^*)=0\\
		%\bzero &\text{otherwise}
		%\end{cases}$$
		where $\mu^*=\max\{\min\{\tilde{\mu},\bar{\lambda}\},0\}$ with $\tilde{\mu}$ being the unique solution of
		$$\tilde{\mu}\norm{P_{Z^i}\left(\frac{\tilde{\bz}_i}{\tilde{\mu}}\right)}^2-\tilde{\bz}_i^\top P_{Z^i}\left(\frac{\tilde{\bz}_i}{\tilde{\mu}}\right)+\tilde{\mu}-\lambda_i=0,$$
		and $0 P_{Z^i}\left(\frac{\tilde{\bz}_i}{0}\right)=0$.
	\end{corollary}
	
	Proposition~\ref{prop:proj_firstformUi} and Corollary~\ref{cor:proj_secondformUi} suggest that we can always obtain the projection onto sets $U^i$ and $\tilde{U}^i$, which are the conic extension of $Z^i$, by applying a bi-section procedure to find the value of $\mu$ that satisfies the appropriate equality constraint.

	Table~\ref{tbl:proj} illustrates three examples, the $\ell_2$, $\ell_\infty$,  $\ell_1$ balls, for which the projection onto $Z^i$ is obtained either analytically or in $O(n)$, and similarly the value of the optimal $\mu^*$, and therefore the projections onto their conic extensions can also be computed in $O(n)$.
	
	\begin{table}
		\TABLE{Examples of projections onto $\tilde{U}$ \label{tbl:proj}}
		%\begin{threeparttable}
		{\scriptsize
			\begin{tabular}{lll}
				\hline
				$Z^i$ & $P_{Z^i}(\by)$ & $P_{\tilde{U}^i}(\by,\lambda)=\left(\mu^*P_{Z^i}\left(\frac{\by}{\mu^*}\right),\mu^*\right)$\\
				\hline\hline
				$\{\bz:\norm{\bz}_2\leq 1\}$ &$\begin{cases}\by & \by\in Z\\ \frac{\by}{\norm{\by}} &\text{otherwise}\end{cases}$ &
				$\mu^*=\max\left\{\min\left\{\frac{\lambda+{\norm{\by}_2}}{2},\bar{\lambda}\right\},0\right\}$ \footnotemark[2]\\
				\hline
				$\{\bz:\norm{\bz}_\infty\leq 1\}$ &$\min\left\{\be,{|\by|}\right\}\circ\text{sign}(\by)$ &
				\makecell[tl]{$\mu^*=\min\left\{\frac{\sum_{i\leq j^*} |q_{(i)}|+\lambda}{j^*+1},\bar{\lambda}\right\}$ \footnotemark[2]\\
					$j^*=\max\left\{j\in[d]:|q_{(j)}|\geq \frac{\sum_{i\leq j} |q_{(i)}|+\lambda}{j+1}\right\}$ \footnotemark[1].}\\ 
				\hline
				$\{\bz:\norm{\bz}_1\leq 1\}$ &\makecell[tl]{$\max\{|\bq|-\theta^*\be,\bzero\}\circ\text{sign}(\bq)$\\ where 
					$\theta^*= \frac{\sum_{i\leq j^*}	|q_{(i)}|-1}{j^*}$\\
					$j^*=\max\left\{j\in[d]: |q_{(j)}|\geq \frac{\sum_{i\leq j}	|q_{(i)}|-1}{j}\right\}$\footnotemark[1].}&
				\makecell[tl]{$\mu^*=\min\left\{\frac{\sum_{i\leq j^*}|q_{(i)}|+\lambda}{j^*+1},\bar{\lambda}\right\}$\footnotemark[2]\\
					$j^*=\max\left\{j\in[d]:|q_{(j)}|\geq \frac{\sum_{i\leq j} |q_{(i)}|-\lambda}{j+1} \right\}$\footnotemark[1].}\\ 
				\hline
		\end{tabular}}
		{$^1$ If the problem is not feasible then $j^*=-\infty$\\
			\indent $^2$ For projection over $U^i$ set $\bar{\lambda}=\infty$}
	\end{table}
	
	\section{Convergence results} \label{sec:generalizations}
	\subsection{EB algorithms}
	To unify analysis for all forms of the Lagrangian function and the algorithms presented in this paper, we define now sufficient properties needed for an algorithm to prove the feasibility and optimality convergence. These sufficient properties will take the form of the algorithm being an \emph{ergodically bounded} (EB) algorithm, as the following definition states.
	\begin{definition}
		Let $\mathcal{A}$ be an iterative algorithm to solve the saddle point problem~\eqref{eq.robust.problem}. We call the algorithm $\mathcal{A}$ \emph{$\psi$-ergodic bounded} ($\psi$-EB) if there exist constant scalars $\phi,\tau,\beta \geq 0$, $\theta_i, \theta_\bw \in \real_+$, and a function $\psi(N):\mathbb{N}\rightarrow \real_+$ with $\psi(N) \downarrow 0$ as $N \rightarrow + \infty$, such that for any initial point $(\bx^0;\blambda^0,\bw^0)$ the algorithm $\mathcal{A}$ generates a sequence  $\{(\bx^k,\blambda^k,\bw^k)\}_{k\in\mathbb{N}}$ satisfying
		\begin{align*}
			&L(\bar{\bx}^N,(\blambda, \bw)) - L(\bx^*,( {\blambda}^*,\bw^*)) +\alpha\dist(\bar{\bx}^N,X) \leq\\
			&\qquad\qquad \psi(N)\left(\shimrit{\frac{\max\{\norm{\bx^*},\norm{\bx^0}\}^2}{\tau}} + \sum_{i=1}^m\theta_i^{-1}\max\{\lambda_i,\lambda_i^0\}^2+\theta_{\bw}^{-1}\max\{\norm{\bw},\norm{\bw^0}\}^2+\phi(1+\alpha^2) \right),
		\end{align*}
		and 
		$$
		L(\bar{\bx}^N; \blambda^*,\bw^*)- L(\bx^*; \blambda^*,\bw^*) + \beta \dist(\bar{\bx}^N,X) \geq 0.
		$$
		for all iterations $N\in\mathbb{N}$, dual variables $\blambda\in\real^m_+$, $\bw\in\real^r$, parameters $\alpha\geq 0$ and saddle points $(\bx^*,(\blambda^*,\bw^*))$ of \eqref{eq:saddle_point_simple}.
	\end{definition}
	\shimrit{We note that this definition allows to consider algorithms which generate sequences $\{\bx^k\}_{k\in\mathbb{N}}$ such that $\bx^k\notin X$. In this case, the convergence result to a feasible solution also considers the rate of convergence to $X$.}
	The following theorem shows that the ergodic sequence of any $\psi$-EB algorithm converges to feasibility and optimality at the rate at which $\psi$ converges to zero.
	\begin{theorem} \label{thm.EB.algorithms.convergence}
		Let problem~\eqref{eq.robust.problem} satisfy Assumption~\ref{ass:Optimal} and \ref{ass:Slater}. Let $\mathcal{A}$ be an $\psi$-EB algorithm to solve the saddle point problem~\eqref{eq.robust.problem} with parameters $\phi,\tau,\beta$ and $\btheta$.
		For a given starting point $(\bx^0,(\blambda^0,\bw^0))\in\real^n\times\real_+^m\times\real^p$ let $\bar{\bx}^N$ be the ergodic primal sequence generated by the algorithm. Then, for any optimal solution $\bx^*$ to \eqref{eq.robust.problem} we have that
		\begin{align*}
			&\sum_{i=1}^m [f_i(\bar{\bx}^N)]_+ +\norm{\bba\bar{\bx}^N-\bb} + \dist(\bar{\bx}^N,X) \leq\\
			&\quad \psi(N)\left(\frac{\shimrit{\max\{\norm{\bx^*},\norm{\bx^0}\}^2}}{\tau} + \sum_{i=1}^m \frac{\max\{\bar{\lambda}+1, \lambda_i^0\}^2}{\theta_i} + \frac{\max\{R_\bw+1,\norm{\bw^0}\}^2}{\theta_{\bw}} + \phi (1+(1+\beta)^2)\right),
		\end{align*}
		and
		\begin{align*}
			&|\bc^\top(\bar{\bx}^N-\bx^*)|\leq \\
			&\qquad \psi(N)\left(\frac{\shimrit{\max\{\norm{\bx^*},\norm{\bx^0}\}^2}}{\tau} + \sum_{i=1}^m \frac{\max\{2\bar{\lambda}, \lambda_i^0\}^2}{\theta_i} + \frac{\max\{2R_\bw,\norm{\bw^0}\}^2}{\theta_{\bw}} +\phi(1+4\beta^2)\right).
		\end{align*}
	\end{theorem}
	\begin{proof}{Proof.}
		Let $\kappa_1>0$ and let $\br\in\real_+^m$ be a multiplication of indicator vector of constraint violations by $\kappa_1$, \emph{i.e.}:
		$$
		r_i=\begin{cases} \kappa_1, & f_i(\bar{\bx}^N)>0,\\
			0, &\text{otherwise},\end{cases}, \ i \in [m]
		$$
		and let $\blambda=\blambda^*+\br$ and let $\kappa_2>0$ and $\bw=\bw^*+\kappa_2\frac{\left(\bba\bar{\bx}^N-\bb\right)}{\norm{\bba\bar{\bx}^N-\bb}}$ where $(\bx^*;\blambda^*,\bw^*)$ is a saddle point of \eqref{eq:basic_saddle_point}. Then, 
		\begin{align}
			L(\bar{\bx}^N; \blambda,\bw)&=\bc^\top\bar{\bx}^N+\sum_{i=1}^m \lambda_i f_i(\bar{\bx}^N)+\bw^\top(\bba\bar{\bx}^N-\bb)\nonumber\\
			&=\bc^\top\bar{\bx}^N+\sum_{i=1}^m (\lambda_i^* + r_i) f_i(\bar{\bx}^N) + \bw^*\left(\bba\bar{\bx}^N-\bb\right)+ \kappa_2 \norm{\bba\bar{\bx}^N-\bb}\nonumber\\
			&=\bc^\top\bar{\bx}^N+\sum_{i=1}^m \lambda^*_i f_i(\bar{\bx}^N)+\kappa_1 \sum_{i=1}^m [f_i(\bar{\bx}^N)]_+ + \bw^*\left(\bba\bar{\bx}^N-\bb\right)+\kappa_2  \norm{\bba\bar{\bx}^N-\bb}\nonumber\\
			&=L(\bar{\bx}^N; \blambda^*,\bw^*)+\kappa_1 \sum_{i=1}^m [f_i(\bar{\bx}^N)]_+ + \kappa_2  \norm{\bba\bar{\bx}^N-\bb}\label{eq:feas}
		\end{align}
		Since $\mathcal{A}$ is an $\psi$-EB algorithm then
		\begin{equation}
			L(\bar{\bx}^N; \blambda^*,\bw^*)+\beta \dist(\bar{\bx}^N,X)\geq \bc^\top\bx^*=L(\bx^*;\blambda^*,\bw^*),
			\label{eq:lb3}
		\end{equation}
		and for all $\kappa_3>0$
		\begin{align}
			& L(\bar{\bx}^N; \blambda,\bw) - L(\bx^*; {\blambda}^*,\bw^*) +\kappa_3\dist(\bar{\bx}^N,X)\leq \nonumber \\ 
			&  \psi(N)\left( \frac{\shimrit{\max\{\norm{\bx^*},\norm{\bx^0}\}^2}}{\tau} + \sum_{i=1}^m \frac{\max\{\lambda_i,\lambda_i^0\}^2}{\theta_i} + \frac{\max\{\norm{\bw},\norm{\bw^0}\}^2}{\theta_{\bw}} +\phi(1+\kappa_3^2)\right).\label{eq:ub3}
		\end{align}
		Since Assumptions~\ref{ass:Optimal} and \ref{ass:Slater} hold, it follows from Proposition~\ref{prop:bounded_lambda} that $\lambda_i^*\leq \bar{\lambda}$ and thus
		%\begin{equation} \label{eq:ub3_b}
		$$
		\max\{\lambda_i,\lambda_i^0\}= \max\{\lambda_i^*+r_i,\lambda_i^0\} \leq \max\{\bar{\lambda}+\kappa_1,\lambda_i^0\}.
		$$
		%\end{equation}
		Similarly, since $\norm{\bw^*}\leq R_\bw$, thus,
		\begin{align}
			\max\{\norm{\bw},\norm{\bw^0}\} & = \max\left\{\left\|\bw^*+\kappa_2 \frac{\left(\bba\bar{\bx}^N-\bb\right)}{\norm{\bba\bar{\bx}^N-\bb}}\right\|,\norm{\bw^0}\right\} \nonumber \\
			& \leq \max\{\norm{\bw^*}+\kappa_2,\norm{\bw^0}\} \nonumber \\
			& \leq \max\{R_\bw+\kappa_2,\norm{\bw^0}\}.\label{eq:ub3_c}
		\end{align}
		Combining inequalities \eqref{eq:feas}-\eqref{eq:ub3_c} we obtain
		\begin{align}
			& \kappa_1\sum_{i=1}^m [f_i(\bar{\bx}^N)]_+ + \kappa_2 \| \bba\bar{\bx}^N-\bb \| +(\kappa_3-\beta) \dist(\bar{\bx}^N,X) \nonumber \\
			\leq &\psi(N)\left( \frac{\shimrit{\max\{\norm{\bx^*},\norm{\bx^0}\}^2}}{\tau} + \sum_{i=1}^m \frac{\max\{\bar{\lambda}+\kappa_1,\lambda_i^0\}^2}{\theta_i} + \frac{\max\{R_\bw+\kappa_2,\norm{\bw^0}\}^2}{\theta_{\bw}} +\phi(1+\kappa_3^2)\right).\label{eq:feas_bound}
		\end{align}
		Setting $\kappa_1=\kappa_2=1$ and $\kappa_3=\beta+1$ we obtain the first result. For the second result, we need to upper and lower bound $\bc^\top(\bar{\bx}^N-\bx^*)$. First, we have  
		\begin{equation}
			\bc^\top(\bar{\bx}^N-\bx^*)= L(\bar{\bx}^N; \bzero,\bzero) - L(\bx^*; \blambda^*,\bw^*).%\leq  L(\bar{\bx}^N, \bzero) - L(\bx^*, \bar{\blambda}^N).
			\label{eq:ub4}
		\end{equation}
		where the equalities follow from $L(\bx^*; \blambda^*,\bw^*)=\bc^\top\bx^*$ and $L(\bar{\bx}^N; \bzero,\bzero)=\bc^\top\bar{\bx}^N$. Using \eqref{eq:ub3} with $\kappa_3=0$ we obtain
		\begin{equation} 
			L(\bar{\bx}^N; \bzero,\bzero)-L(\bx^*; \blambda^*,\bw^*)\leq \psi(N)\left(\frac{ \shimrit{\max\{\norm{\bx^*},\norm{\bx^0}\}^2}}{\tau} + \sum_{i=1}^m\frac{(\lambda_i^0)^2}{\theta_i}+\frac{\norm{\bw^0}^2}{\theta_{\bw}}+\phi\right).\label{eq:ub6}
		\end{equation}
		Combining \eqref{eq:ub4} with \eqref{eq:ub6} we have that
		\begin{equation}
			\bc^\top(\bar{\bx}^N-\bx^*)\leq \psi(N)\left(\tau^{-1} \| \bx^* - \bx^0 \|^2 + \sum_{i=1}^m\theta_i^{-1}(\lambda_i^0)^2+\theta_{\bw}^{-1}\norm{\bw^0}^2+\phi\right).\label{eq:opt_bound1}
		\end{equation}
		To obtain the other side of the bound, we use \eqref{eq:lb3} to obtain
		\begin{align*}
			\bc^\top(\bx^*-\bar{\bx}^N)-\sum_{i=1}^m \lambda^*_i f_i(\bar{\bx}^N) - (\bw^*)^\top(\bba\bar{\bx}^N-\bb)= L(\bx^*;\blambda^*,\bw^*)-L(\bar{\bx}^N; \blambda^*,\bw^*) \leq \beta \dist(\bar{\bx}^N,X),%\label{eq:ub5}
		\end{align*}
		which in turn implies
		\begin{align}
			\bc^\top(\bx^*-\bar{\bx}^N)&\leq \sum_{i=1}^m \lambda^*_i f_i(\bar{\bx}^N)+(\bw^*)^\top(\bba\bar{\bx}^N-\bb)+ \beta \dist(\bar{\bx}^N,X)\nonumber\\ 
			&\leq\bar{\lambda}\sum_{i=1}^m  [f_i(\bar{\bx}^N)]_+ + R_\bw\norm{\bba\bar{\bx}^N-\bb}+ \beta \dist(\bar{\bx}^N,X),\label{eq:ub7}
		\end{align}
		where the last inequality follows from the the fact that $0\leq\lambda^*_i\leq \bar{\lambda}$ and $\norm{\bw^*}\leq R_\bw$. Using \eqref{eq:feas_bound} with $\kappa_1=\bar{\lambda}$, $\kappa_2=R_\bw$, and $\kappa_3= 2\beta$ to bound \eqref{eq:ub7} we obtain that
		\begin{align}
			\bc^\top(\bx^*-\bar{\bx}^N)\leq \psi(N)\left( \frac{\shimrit{\max\{\norm{\bx^*},\norm{\bx^0}\}^2}}{\tau} + \sum_{i=1}^m\frac{\max\{2\bar{\lambda},\lambda_i^0\}^2}{\theta_i} + \frac{\max\{2R_\bw,\norm{\bw^0}\}^2}{\theta_{\bw}} +\phi(1+(2\beta)^2)\right)\label{eq:opt_bound2}
		\end{align}
		Combining \eqref{eq:opt_bound1} and \eqref{eq:opt_bound2} we obtain the desired result. \hfill \qedsymbol
	\end{proof}
	
	\subsection{Generalized model}
	So far, we kept the problem formulations simple and stated the convergence results without proofs. In this section, we shall show the convergence of both the SGSP and CP algorithms for a more general model using the EB-algorithm definition.
	
	The general model is needed to tackle the fact that projections onto sets $X$ or $Z^i$ are the main tools used in both algorithms, and therefore these projections should be simple. Two issues that might arise is that either the primal domain $X$ or at least one of the sets $Z^i$ is not `simple', but instead, it is an intersection of several simple sets.
	\begin{example}
		One of the popular uncertainty sets in RO is the so-called budgeted uncertainty set, formulated as:
		$$
		Z = \{ \bz: \ -1 \leq \bz \leq 1, \ \| \bz \|_1 \leq \Gamma \}
		$$
		which is an intersection of two simple sets: the $\ell_\infty$ norm and $\ell_1$ norm balls.
	\end{example}
	Our strategy for dealing with these complex sets shall be to disentangle the projections onto the intersected sets. We will achieve this by including `copies' of the respective primal or dual variables, together with the relevant equality constraints, which are to be relaxed in the Lagrangian. The saddle-point algorithm would then be applied to the new Lagrangian problem.
	
	Consider, for example, the case where $X = \cap_{j = 1}^q X_i$ is an intersection of several `simple' sets. Formulation~\ref{eq.robust.problem} is ready to handle this situation easily. One can `expand' the vector $\bx$ by having $q$ copies of it: $\bx \mapsto [\bx_1, \ldots, \bx_q]$. In the next step, each inequality constraint in the problem is made dependent only on one of the $\bx_i$'s and at the same time, equality constraints  $\bx_1 = \bx_i $, $i = 2,\ldots, q$ become embedded in the $\bba \bx = \bb$ system (where the rank condition is easily verified).
	
	Similarly, we can consider the case where ${Z}^i = \cap_{l = 1,\ldots, s_i} Z^{i,l}, \quad i = 1,\ldots, m$ with $Z^{i,l}$ being compact convex sets. In this case, the constraint $\bz_i \in \cap_{l = 1,\ldots, s_i} Z^{i,l}$ can be written using concatenated uncertain parameter vector $\hat{\bz}_i = (\bz_{i, 1}, \ldots, \bz_{i, s_i})$, defining 
	$$
	\hat{Z}^{i}=\left\{(\bz_{i, 1}, \ldots, \bz_{i, s_i})\in Z^{i,1}\times\ldots\times Z^{i,s_i}:\;\bz_{i,l} = \bz_{i, s_i}, \; l=1,\ldots, s_i-1\right\},
	$$ 
	and formulating the constraints as
	$$
	\max\limits_{\hat{\bz}_{i} \in \hat{Z}^i } \hat{g}_i(\bx_1, \hat{\bz}_{i})\equiv {g}_i(\bx_1, \bz_{i, s_i}) \leq 0
	$$
	where $Z^{i,l}$, $l = 1,\ldots, s_i$ are `simple' sets on which it is easy to project. Thus, the general model we are going to address is the following:
	\begin{align}
		\min\limits_{\bx \in X} \ &  \bc^\top\bx  \label{eq.robust.extended} \\
		\text{s.t.} \; &  f_i(\bx)\coloneqq \max\limits_{\bz_i \in Z^i } g_i(\bx,\bz_i)\leq 0\quad i \in [m] \nonumber \\
		\ & \bba \bx = \bb \nonumber
	\end{align}
	where $Z^i = \cap_{l = 1}^{s_i} Z^{i,l}$. For both SGSP and CP we shall give the problem generalizations together with the expanded Lagrangian that is to be solved.

	%\subsection{SGSP convergence} \label{sec:sgsp.is.eb}
	Consider problem \eqref{eq.robust.extended} where we need to formulate a saddle point problem that disentangles different components of the $\bx$ and $\bz_i$.  We start by considering a saddle point $(\bx^*,(\bu^*,\bw^*))$ of $\bar{L}(\bx,(\bu,\bw))$ over the sets $\bx\in X$ and $\bu_i\in U^i$, which we already proved exists.
	Thus, by the definition of the saddle point, we have that
	\begin{align*}\tilde{g}_i(\bx^*,\bu^*_i)=\sup_{\bu_i\in U^{i}} \tilde{g}_i(\bx^*,\bu_i)=\sup_{\substack{\bu_{i,l}\in U^{i,l},\\ u_{i,s_i}=u_{i,l}, l\in[s_i-1]}} \tilde{g}_i(\bx^*,\bu_{i,s_i})\end{align*}
	where the last equality follows from the definition of $U^{i}=\cap_{i\in[s_i]} U^{i,l}$ where $U^{i,l}=\{(\tilde{\bz},\lambda_i):\tilde{\bz}\in \lambda_{i}Z^{i,l}\}$, similarly to what was explained in the previous section. Note that the leftmost supremum has a solution ($\bu_{i,1}=\bu_{i,2}=\ldots=\bu_{i,s_i}=\bu_i^*$). Moreover, since $U^i\subseteq U^{i,l}$ for $l\in[s_i]$, the sets $U^{i,l}$ have a nonempty interior, therefore, dualizing the equality constraints, we have that strong duality holds.
	Thus, denoting $\tilde{\bu}^*_i=({\bu}^*_{i,1},\ldots,{\bu}^*_{i,s_i})\equiv(\bu^*_i,\ldots,\bu^*_i)$
	there exists $\bomega^*_i=(\bomega^*_{i,1},\ldots,\bomega^*_{i,s_i-1})$ such that $(\bomega_i^*,\tilde{\bu}_i^*)$ satisfies
	\begin{align}\tilde{g}_i(\bx^*,\bu^*_{i})&=\sup_{\substack{\bu_{i,l}\in U^{i,l},\\ u_{i,s_i}=u_{i,l}, l\in[s_i-1]}} \tilde{g}_i(\bx^*,\bu_{i,s_i})\nonumber\\
		&=\inf_{\omega_{i,l}\in \real^{d_i+1},l\in [s_i-1]}\sup_{\substack{\bu_{i,l}\in U^{i,l}, l\in[s_i-1]}} \tilde{g}_i(\bx^*,\bu_{i,s_i})+\sum\limits_{l = 1}^{s_i - 1} \bomega_{i,l}^\top (\bu_{i,l}-\bu_{i,s_i})\nonumber\\
		&=\inf_{\bomega_{i,l},\ l\in[s_i-1]} \tilde{g}_i(\bx^*,\bu^*_{i,s_i})+\sum\limits_{l = 1}^{s_i - 1} \bomega_{i,l}^\top (\bu^*_{i,l}-\bu^*_{i,s_i})\nonumber\\
		&=\sup_{\substack{\bu_{i,l}\in U^{i,l},\ l\in[s_i-1]}} \tilde{g}_i(\bx^*,\bu_{i,s_i})+\sum\limits_{l = 1}^{s_i - 1} (\bomega^*_{i,l})^\top (\bu_{i,l}-\bu_{i,s_i}).\label{eq:strong_dual}
	\end{align}
	Due to Proposition~\ref{prop:bounded_lambda} we know that we can restrict $U^{i}$ to $\tilde{U}^{i}$, and thus, the existence of the saddle point in this case follows the same logic where $U^{i,l}$ is replaced by $\tilde{U}^{i,l}=\{\bu_{i,l}\equiv(\tilde{\bz}_{i,l},\lambda_{i,l})\in U^{i,l}: \lambda_{i,l}\leq \bar{\lambda}\}$. In the following, we will show that we can restrict the domain over which we optimize $\bomega$ and still retrieve a saddle point of the original problem. 
	
	Denoting $\bchi=(\bx,\bomega)$ and $\by=(\tilde{\bu},\bw)$ where $\bomega=(\bomega_1,\ldots,\bomega_m)$, $\tilde{\bu}=(\tilde{\bu}_1,\ldots,\tilde{\bu}_m)$, and  $\bomega_i=(\bomega_{i1},\ldots,\bomega_{i(s_i-1)})$, $\tilde{\bu}_i=(\bu_{i1},\ldots,\bu_{is_i})$ for all $i\in[m]$, we can define a saddle point in the lifted space.
	Indeed, let $(\bx^*,(\blambda^*,\bw^*))$ be a saddle point of $L(\bx,(\blambda,\bw))$, then, by Proposition~\ref{prop:eq_saddle_points}, and the reasoning leading to \eqref{eq:strong_dual} above, there exists $\tilde{\bu}^*$ and $\bomega^*$ such that
	\begin{align}
		&\sup\limits_{\blambda \geq 0, \bw} L(\bx^*,(\blambda,\bw))\nonumber\\%&=\sup\limits_{\blambda \geq 0, \bw} \inf\limits_{\bx \in X} \bc^\top 
		&= \sup\limits_{\substack{\bw\in W,\\
				\bu_i\in U^{i},\; i\in[m]}} \bc^\top \bx^* + \bw^\top (\bba \bx^* - \bb) + \sum\limits_{i = 1}^m \tilde{g}_i \left( \bx^*,\bu_i \right) \nonumber  \\ 
		&= \sup\limits_{\substack{\bw\in W,\\
				\bu_{i, l} \in  U^{i,l},\;i\in[m],l\in[s_i-1]}} \bc^\top \bx^* + \bw^\top (\bba \bx^* - \bb) + \sum\limits_{i = 1}^m  \tilde{g}_i \left( \bx, \bu_{i,s_i} \right) + \sum\limits_{i = 1}^m \sum\limits_{l = 1}^{s_i - 1} (\bomega^*_{i, l})^\top (\bu_{i,l} - \bu_{i, s_i}) \nonumber 
		%&= \inf\limits_{\bomega_{i,l}} \inf\limits_{\bx \in X}  \sup\limits_{\bw} \sup_{\bu_{i, l} \in U^{i,l}} \bc^\top \bx + \bw^\top (\bba \bx - \bb) + \sum\limits_{i = 1}^m \tilde{g}_i \left( \bx, \bu_{i,s_i} \right) + \sum\limits_{i = 1}^m \sum\limits_{l = 1}^{s_i - 1} \bomega_{i, l}^\top (\bu_{i,l} - \bu_{i, s_i}) \nonumber \\
		%&= \inf_{\bchi:\ \bx \in X } \sup\limits_{\bnu: \ \bu_{i, l} \in U^{i,l}} \breve{L}(\bchi,\bnu) \label{eq.robust.extended.sp1}
	\end{align}
	and
	\begin{align}
		&\inf_{\bx\in X} L(\bx,(\blambda^*,\bw^*))\nonumber\\%&=\sup\limits_{\blambda \geq 0, \bw} \inf\limits_{\bx \in X} \bc^\top 
		&= \inf_{\bx\in X} \bc^\top \bx + (\bw^*)^\top (\bba \bx - \bb) + \sum\limits_{i = 1}^m \tilde{g}_i \left( \bx,\bu^*_i \right) \nonumber  \\ 
		&= \inf_{\substack{\bx\in X\\
				\bomega_{i, l},\;i\in[m],l\in[s_i-1]}} \bc^\top \bx^* +  (\bw^*)^\top (\bba \bx - \bb) + \sum\limits_{i = 1}^m  \tilde{g}_i \left( \bx, \bu_{i,s_i} \right) + \sum\limits_{i = 1}^m \sum\limits_{l = 1}^{s_i - 1} (\bomega_{i, l})^\top (\bu^*_{i,l} - \bu^*_{i, s_i}). \nonumber 
		%&= \inf\limits_{\bomega_{i,l}} \inf\limits_{\bx \in X}  \sup\limits_{\bw} \sup_{\bu_{i, l} \in U^{i,l}} \bc^\top \bx + \bw^\top (\bba \bx - \bb) + \sum\limits_{i = 1}^m \tilde{g}_i \left( \bx, \bu_{i,s_i} \right) + \sum\limits_{i = 1}^m \sum\limits_{l = 1}^{s_i - 1} \bomega_{i, l}^\top (\bu_{i,l} - \bu_{i, s_i}) \nonumber \\
		%&= \inf_{\bchi:\ \bx \in X } \sup\limits_{\bnu: \ \bu_{i, l} \in U^{i,l}} \breve{L}(\bchi,\bnu) \label{eq.robust.extended.sp1}
	\end{align}
	Defining 
	\begin{align}\breve{L}(\bchi,\by):= \bc^\top \bx + \bw^\top (\bba \bx - \bb) + \sum\limits_{i = 1}^m  \tilde{g}_i \left( \bx, \bu_{i,s_i} \right)+ \sum\limits_{i = 1}^m \sum\limits_{l = 1}^{s_i - 1} \bomega_{i, l}^\top (\bu_{i,l} - \bu_{i, s_i}) \label{eq.robust.extended.sp1}. 
	\end{align}
	we obtain that $(\bchi^*,\by^*)$ is a saddle point of $\breve{L}$. Following similar logic, we can also obtain that given a saddle point $(\bchi^*,\by^*)$ of $\breve{L}$ with $\bu_{i,s_i}^*=(\tilde{\bz}^*_{i,s_i},\lambda_{i,s_i}^*)$ we can obtain a saddle point $(\bx^*,(\blambda^*,\bw^*))$ of $L$ by defining $\lambda_i^*=\lambda_{i,s_i}^*$, $\blambda^*=(\lambda_1^*,\ldots,\lambda_m^*)$.
	
	Since we proved that $\breve{L}$ has a saddle point, we can now show that SGSP algorithm applied to \eqref{eq.robust.extended.sp1} meets the EB-algorithm assumptions.
	
	However, to run the algorithm and prove its convergence, we need to show that the feasible sets of the  variables can be restricted without losing a saddle point. Legitimacy of bounding $\bu_{i, l}$ and $\bw$ follows from Proposition~\ref{prop:bounded_lambda} and \eqref{eq:strong_dual}. The following proposition establishes our ability to bound $\bomega_{i,l}$.
	\begin{proposition} \label{proposition:bounded.omega}
		Let Assumptions~\ref{ass:Optimal}--\ref{ass:constraints_bounded_below} hold. Let $i\in[m]$ and consider the saddle point problem
		\begin{align}
			\min_{\bomega_i } \max_{\substack{\bu_{i,l}\in \tilde{U}^{i,l} ,\ l\in [s_i], \\ \utilde{\lambda}\leq \lambda_{i,s_i} \leq \tilde{\lambda}}} \tilde{g}_i( \bx^* ,\bu_{i,s_i}) + \sum\limits_{l = 1}^{s_i - 1} \bomega_{i, l}^\top (\bu_{i,l} - \bu_{i, s_i})\label{eq.sp.harder}
		\end{align}
		where $\bx^*$ is an optimal solution of \eqref{eq.robust.extended}, and $0 \leq \utilde{\lambda}\leq \tilde{\lambda} \leq \bar{\lambda}$. Then, there exists a saddle point $(\bomega_i^*,\tilde{\bu}_i^*)$ with $\bomega^*_{i, l} = (\bnu^*_{i,l},\mu^*_{i,l})$ such that $\bomega_{i,l}^*$ is contained in the set
		\begin{align*}
			\Omega^{i,l} = \left\{ \bomega_{i, l} = ( \bnu_{i,l},\mu_{i,l}): -\bar{\mu}_i\leq \mu_{i,l} \leq 0, \; \norm{\bnu_{i,l}} \leq \frac{-{\mu}_{i,l}}{\epsilon_i},\; i\in [m],\;l\in[s_i-1] \right\},
		\end{align*}
		where $\epsilon_i$ and $\bar{\mu}_i$ is given by Assumptions \ref{ass:zero_in_int_Zi} and \ref{ass:constraints_bounded_below}, respectively.
	\end{proposition}
	
	Now that we have shown that we can bound all variables in the saddle point of function $\breve{L}(\bchi,\by)$ over the sets $\mathcal{X}=X\times(\times_{i=1}^m\times_{l=1}^{s_i-1}\Omega^{i,l})$ and $\mathcal{Y}=(\times_{i=1}^m\times_{l=1}^{s_i}U^{i,l})\times W$, we can apply SGSP to the problem with these bounded sets as follows.
	\begin{center}
		\vspace{10pt}
		\fcolorbox{black}{gray!20}{
			\begin{minipage}{0.8\textwidth}
				\noindent \textbf{SGSP for $\breve{L}$}\\
				\textbf{Input:} ${\tau>0}, {\theta_{i}>0}, \theta_{\bw} > 0$ for $i \in [m]$, and $N\in\mathbb{N}$\\
				\textbf{Initialization}. Initialize $\bchi^0 \in \mathcal{X}$, and $\by^0\in\mathcal{Y}$ such that $\lambda^0_{i,l}=\lambda^0_{i,1}$ for all $l\in[s_i]$,$i\in[m]$, and define the diagonal matrix $\Theta=\diag(\theta_1\be;\ldots;\theta_m\be;\theta_\bw\be)$.\\
				\textbf{General step:} For $k \in [N]$\\
				Compute subgradients $\bv^{k-1}_\bchi=(\bv^{k-1}_\bx,\bv^{k-1}_\bomega)\in \partial_\bchi\breve{L}(\bchi^{k-1},\by^{k-1})$, and for all $i \in [m]$ $\bv^{k-1}_{\tilde{\bu}_i}=(\bv^{k-1}_{i,1},\ldots,\bv^{k-1}_{i,s_i})\in \partial_{\tilde{\bu}_{i}} \left(-\breve{L}(\bchi^{k-1},\by^{k-1})\right)$.
				\begin{align*}
					\bchi^k&= P_{\mathcal{X}}(\bchi^{k-1} - \tau \bv_\chi^{k-1})&\Longleftrightarrow\quad&
					\bx^{k}= P_X(\bx^{k-1} - \tau \bv_x^{k-1}),&& \\
					&&&\bomega_{i,l}^{k}  = P_{\Omega^{i,l}}(\bomega_{i,l}^{k-1} - \tau(\bu_{i,l}^{k-1}-\bu_{i,s_i}^{k-1})), && i \in [m],\;l\in[s_i-1] &&\\
					\by^k&=P_{\mathcal{Y}}(\by^{k-1} - \Theta \bv_y^{k-1})&\Longleftrightarrow\quad&
					\bu^{k}_{i,l} = P_{\tilde{U}^{i,l}}(\bu_{i,l}^{k-1} + \theta_i \bomega^{k-1}_{i,l}), &&i \in [m],\;l\in[s_i-1] \\
					&&&\bu^{k}_{i,s_i} = P_{\tilde{U}^{i,s_i}}(\bu_{i,s_i}^{k-1} - \theta_i \bv_{i,s_i}^{k-1}), &&i \in [m] \\
					&&&\bw^{k}  = P_{W}(\bw^{k-1} + \theta_\bw (\bba \bx^{k-1} - \bb)).
				\end{align*}
			\end{minipage}
		}
		\vspace{10pt}
	\end{center}
	We will now prove that SGSP is an EB-algorithm and thus, it converges to the optimal solution.
	\begin{proposition} \label{prop.SGSP.convergence}
		Let Assumptions~\ref{ass:Optimal}--\ref{ass:constraints_bounded_below} hold. And let $\{(\bchi^k,\by^k)\}_{k\in\mathbb{N}}$ be the sequence generated by SGSP  algorithm with parameters $\tau=\frac{\tilde{\tau}}{\sqrt{N}}$, $\theta_i =\frac{\tilde{\theta}_i}{\sqrt{N}}$, $\theta_{\bw} =\frac{\tilde{\theta}_{\bw}}{\sqrt{N}}$, where $\bchi^k=(\bx^k,\bomega^k)$ and $\by^k=(\tilde{\bu}^k,\bw^k)$. Then,
		\begin{align}
			& \breve{L}(\bar{\bchi}^N,\by)-\breve{L}(\bchi,\bar{\by}^N) \nonumber \\
			&\quad \leq   \frac{1}{2\sqrt{N}}
			\left(\frac{\norm{\bchi^{0}-\bchi}^2}{\tilde{\tau}} + \tilde{\tau}G_\bchi^2 + \sum\limits_{i=1}^m \left( \frac{\norm{\tilde{\bu}_i^{0} - \tilde{\bu}_i}^2}{\tilde{\theta}_i} + \tilde{\theta}_i G_{\tilde{\bu}_i}^2 \right) + \frac{\norm{\bw^{0}-\bw}^2}{\tilde{\theta}_\bw}+\tilde{\theta}_\bw G_\bw^2 \right) \label{eq.SGSP.EB.result1}
		\end{align}
		where
		\begin{align*}
			G_\bchi& \coloneqq \norm{\bc}+ \|\bba \| R_\bw +\sum_{i\in[m]} \bar{\lambda}G_{\bx,i}+\bar{\lambda}\sum_{i\in[m]}\left((s_i-1)(2+R_{i,s_i})+\sum_{l\in[s_i-1]}R_{i,l}\right) & \\
			G_{\tilde{\bu}_i} & \coloneqq (\bar{\lambda}+R_{i,s_i})G_{\bz,i}+\bar{g}_i+ 2 \bar{\mu}_i(s_i-1)\left( 1+\frac{1}{\epsilon_i} \right) \\
			G_\bw & \coloneqq \norm{\bba} R_\bx + \norm{\bb} 
		\end{align*}
		Moreover, SGSP is an EB-algorithm. More precisely, let $\bx^*$ an optimal solution of \eqref{eq.robust.extended}, $(\blambda^*,\bw^*)$ be the optimal dual variables associated with its constraints, and let $\blambda\in\real^m_+$, and $\bw\in\real^r$ such that $\lambda_i\equiv\lambda_{i,s_i}$ for all $i\in [m]$. Then,
		\begin{align}
			&L(\bar{\bx}^N, ( \blambda,\bw)) - L(\bx^*, ({\blambda}^*, {\bw}^*))  \leq  \nonumber\\
			&\qquad\frac{1}{2\sqrt{N}}\left(\tilde{\tau}^{-1} \| \bx^* - \bx^0 \|^2 + \sum_{i = 1}^m \sigma_i \tilde{\theta}_i^{-1} \max\{\lambda_i,\lambda_i^0\}^2 + 2 \tilde{\theta}_\bw^{-1} \max \{ \|\bw\| , \|\bw_0\| \}^2 + \phi \right), \label{eq.SGSP.EB.result2}
		\end{align}
		and
		\begin{equation}
			L(\bar{\bx}^N; (\blambda^*, \bw^*))\geq \bc^\top\bx^*, \label{eq.SGSP.EB.result3}
		\end{equation}
		where
		\begin{align*}
			\sigma_i = \sum_{l=1}^{s_i} (1+4R_{i,l}^2),\quad \phi & = \tilde{\tau}G_\bchi^2+\sum_{i=1}^m  \tilde{\theta}_i G_{\bu_i}^2+ 4 \frac{ \left( \sum_{i=1}^m\bar{\mu}_i(s_i-1) \left( 1+\frac{1}{\epsilon_i} \right) \right)^2}{ \tilde{\tau}} + \tilde{\theta}_\bw G_\bw^2 \nonumber 
		\end{align*}
		where $R_\bx$ is the radius of $X$, $R_\bw$ the radius of $W$, $R_{i,l}$ is the radius of $Z^{i,l}$, $\bar{g}_i = \max_{\bx\in X,\bz_i\in Z^i} |g_i(\bx,\bz_i)|$, $\bar{\lambda}$, $\bar{\mu}_i$ and $\epsilon_i$ are defined in Proposition~\ref{prop:bounded_lambda}, Assumption \ref{ass:constraints_bounded_below} and Assumption~\ref{ass:zero_in_int_Zi}, respectively, and $G_{\bx, i}, G_{\bz, i} $ are the bounds on the subgradients $\bd^k_{\bx,i}$, $\bd^k_{\bz,i}$ generated as in Lemma~\ref{lem:subdif_def}. 
	\end{proposition}

	\shimrit{In the case of CP, we proceed similarly to SGSP with one change -- as the CP algorithm presumes bilinearity, we assume $g_i$ are of the form given in \eqref{eq:biaffine.g_i}, we obtain:
		\begin{align}
			\check{L}(\bchi,\by) \equiv & \bx^\top\left( \bc+\sum_{i\in[m]} \tilde{\bbq}_i\bu_{i,s_i}+\bba^\top \bw \right) -\bb^\top \bw + \delta_X(\bx) \nonumber \\
			& +  \sum_{i\in[m]} \left( \tilde{\bq}_i^\top \bu_{i,s_i}+\sum_{l\in [s_i-1]}\bomega_{i,l}^\top(\bu_{i,l}- {\bu}_{i,s_i}) - \sum_{l\in[s_i]}\delta_{U}^{i,l}(\bu_{i,l}) \right)-\bb^\top \bw \label{eq.CP.composite.sp}
		\end{align}
		In this case we can define $\bbb$ as the matrix
		$$\bbb=\begin{bmatrix} \bba^\top &\rvline &\begin{matrix}\bzero &  \rvline & \tilde{\bbq}_1\end{matrix}& \rvline &\begin{matrix}\bzero & \rvline  & \tilde{\bbq}_2\end{matrix} &\rvline & \ldots & \rvline &\begin{matrix} \bzero & \rvline & \tilde{\bbq}_m\end{matrix}\\
			\hline
			\bigzero& \rvline & \begin{matrix} \bigI  & \rvline &\begin{matrix}-\bbi\\ \vdots\\ -\bbi\end{matrix}\end{matrix}&\rvline & \bigzero &\rvline&\ldots & \rvline &\bigzero\\
			\hline
			\bigzero& \rvline & \bigzero &\rvline &\begin{matrix} \bigI  & \rvline &\begin{matrix}-\bbi\\ \vdots\\ -\bbi\end{matrix}\end{matrix} &\rvline&\ldots & \rvline &\bigzero\\
			\hline
			\vdots & \rvline& &  \rvline& &  \rvline &\ddots & \rvline& \vdots\\
			\hline
			\bigzero & \rvline& \bigzero & \rvline& \bigzero &\rvline& \ldots & \rvline&\begin{matrix} \bigI  & \rvline &\begin{matrix}-\bbi\\ \vdots\\ -\bbi\end{matrix}\end{matrix}
		\end{bmatrix},$$
		and apply CP to generalized problem as follows.}
	\begin{center}
		\vspace{10pt}
		\fcolorbox{black}{gray!20}{
			\begin{minipage}{0.8\textwidth}
				\shimrit{\noindent \textbf{CP for $\breve{L}$}\\
					\textbf{Input:} ${\tau>0}$, $\theta \in (0,1/\tau\norm{\bbb}^2) $,  and $N\in\mathbb{N}$\\
					\textbf{Initialization}. Initialize $\bomega^0=\bzero\in \real^{\sum_{i\in[m]}(s_i-1)}$,\; $\bx^0\in X$,  $\bchi^0=(\bx^0,\bomega^0) $, \\$\bar{\chi}^0=(\bar\bx^0,\bar\bomega^0):=\bchi^0$, $\by^0=(\bu^0, \bw^0 ) \in \times_{i\in[m]}\times_{l\in[s_i]}U^{i,l} \times \real^r $. \\
					\textbf{General step:} For $k \in [N]$:
					\begin{align*}
						\bu^{k}_{i,l} &= P_{U_{i, l}} \left(  \bu^{k - 1}_{i,l} + \theta \bar\bomega_{i, l}^{k - 1} \right) && i \in [m],\;l\in[s_i-1]\\
						\bu^{k}_{i,s_i} & = P_{U_{i, s_i}} \left(  \bu^{k - 1}_{i,s_i} + \theta \left( \tilde{\bq}_i - \sum_{l = 1}^{s_i - 1} \bar\bomega_{i, l}^{k - 1} \right) \right) && i \in [m] \\
						\bw^k &= \bw^{k - 1} + \theta (\bba \bar{\bx}^{k - 1} - \bb) \\
						\bx^k &= P_X \left( {\bx}^{k - 1} - \tau \left(\bc + \bba^\top \bw^{k} - \sum\limits_{l = 1}^{s_i - 1} \tilde{\bbq}_i \bu_{i, s_i}^k\right)  \right) \\
						\bomega_{i,l}^{k}  & =  \bomega_{i, l}^{k - 1} - \tau \left( \bu^{k}_{i, l} - \bu^{k}_{i, s_i} \right) && i \in [m],\;l\in[s_i-1]\\
						\bar{\bx}^k&=2\bx^k-\bx^{k-1}\\
						\bar{\bomega}^k&=2\bomega^k-\bomega^{k-1}
				\end{align*}}
			\end{minipage}
		}
		\vspace{10pt}
	\end{center}
	We are now ready to state the convergence of \shimrit{CP} for the general formulation via showing that it is an EB-algorithm.
	\shimrit{
		\begin{proposition} \label{prop.CP.convergence}
			Let Assumptions~\ref{ass:Optimal}-\ref{ass:constraints_bounded_below} hold true. Then, applying the CP algorithm to $\check{L}(\bchi,\by)$ with parameters $\tau, \sigma > 0$ such that $\tau\sigma\norm{\bbb}^2<1$ results in
			\begin{equation}
				\max_{\bchi\in \mathcal{B}_1, \by\in \mathcal{B}_2}\check{L}(\bar{\bchi}^N,\by)-\check{L}(\bchi,\bar{\by}^N) \leq \max_{\bchi\in  \mathcal{B}_1,\by\in \mathcal{B}_2} \frac{1}{2N} \left( \tau^{-1} \| \bchi - \bchi^0 \|^2 + \theta^{-1} \| \by- \by^0 \|^2 \right).\label{eq.CP.EB.result1}
			\end{equation}
			for any compact sets $\mathcal{B}_1$ and $\mathcal{B}_2$.
			Furthermore, CP is an EB algorithm. More precisely, let $\bx^*$ be an optimal solution of \eqref{eq.robust.extended} and let $(\blambda^*,\bw^*)$ be the optimal dual variables associated with its constraints. Then, $\blambda\in\real^m_+$ and $\bw$ such that $\lambda_i\equiv\lambda_{i,s_i}$ for all $i\in [m]$ satisfy
			\begin{align}
				& L(\bar{\bx}^N,(\bw,\blambda))-L({\bx}^*,(\bw^*,\blambda^*) \nonumber \\
				&\leq  \frac{1}{2N} \left(\frac{2\max\{\norm{\bx^*},\norm{\bx^0}^2\}}{\tau} + \frac{1}{\theta} \sum\limits_{i = 1}^m \sigma_i \max\{\lambda_i, \lambda_i^0\}^2 + \frac{2}{\theta} \max\{ \norm{\bw}, \norm{\bw^0} \}^2 + \phi \right), \label{eq.CP.EB.result2}
			\end{align}
			and \begin{equation*}
				L(\bar{\bx}^N; (\blambda^*, \bw^*))\geq \bc^\top\bx^*, %\label{eq.SGSP.EB.result3}
			\end{equation*} where $\sigma_i=\sum_{l\in[s_i]} (1+4R^2_{i,l}) $, $R_{i,l}$ is the radius of $Z^{i,l}$, $\phi= 4 \sum_{i=1}^m s_i\bar{\mu}_i^2\left( 1+\frac{1}{\epsilon_i} \right)^2 $, with $\bar{\mu}_i$ and $\epsilon_i$ are defined in Assumptions~\ref{ass:constraints_bounded_below} and \ref{ass:zero_in_int_Zi}, respectively.
	\end{proposition}}
	
	\begin{remark}[\textbf{CP: the non-biaffine case}]\label{rem:non.biaffine}
		The CP algorithm can be extends to the case where $g_i(\bx,\bz_i)$ is not bilinear, but rather has the following general form:
		$$
		g_i(\bx,\bz_i)\coloneqq\bh_i(\bx)^\top\bell_i(\bz_i),
		$$
		where $\bh_i(\cdot):X \rightarrow \real^{k_i}$ and $\bell_i(\cdot):Z^i\rightarrow \real^{k_i}$, where each element of the mapping $\bh_i(\cdot)$ is a convex function, \ie $h_{ij}(\cdot)$ is convex for all $j \in [l_i]$, and each element of $\bell_i(\cdot)$ is a concave function, \ie $k_{ij}(\cdot)$ is concave for all $j \in [l_i]$. To maintain the convex-concave structure, we assume that
		$$
		\min_{j\in [k_i]} \min_{\bz_i\in Z^i} \ell_{ij}(\bz_i)\geq 0, \quad \min_{j\in [k_i]}\min_{\bx\in X} h_{ij}(\bx)\geq 0 \quad \forall i\in[m].
		$$
		We will show that we can transform such problems to the biaffine form. For this, we introduce vectors $\bvarpi_{i}$ ($\bzeta_i$ respectively) whose entries upper (lower) bound the entries of $\bell_{i}$ ($\bh_{i}$). With these extra variables, the constraint $g_i(\bx,\bz_i)\leq 0, \;\forall \bz_i\in Z^i$ from problem~\eqref{eq.robust.problem} can be reformulated as
		\begin{align*}
			& \bvarpi_{i}^\top \bzeta_{i} \leq 0,\; \forall (\bz_i,\bzeta_i)\in \Xi^i\\
			& \bh_{i}(\bx)\leq \bvarpi_{i}
		\end{align*}
		where $\Xi^i=\left\{\bxi_i=(\bz_i,\bzeta_i):\bz_i\in Z^i,\;  \bell_{i}(\bz_i)\geq \bzeta_{i}\right\}$. As we see, the first constraint becomes biaffine in the respective variables. 
		To decouple the constraint in $\bx$, one can also duplicate $\bx$ to $\bx_0,\bx_1,\ldots,\bx_m$, and add equality constraints $\bx_i=\bx_0$ for all $i\in[m]$.
		
		Accordingly, we can define an extended primal variable vector $\bchi=(\bx_0,\ldots,\bx_m,\bvarpi_1,\ldots,\bvarpi_m)\in \real^{n+\sum_{i=1}^m l_i}$ with feasible set
		$$
		\mathcal{X}=\left\{\bchi=(\bx_0,\ldots,\bx_m,\bvarpi_1,\ldots,\bvarpi_m):\bx_0\in X,\; \bh_{i}(\bx_i)\leq \bvarpi_{i}\right\},
		$$
		and an extended new uncertain parameter $\bxi_i=(\bz_i,\bzeta_i)$ for constraint $i$ such that $\bxi_i\in \Xi^i$. The $i$-th constraint of problem~\eqref{eq.robust.problem} becomes then:
		$$
		\tilde{g}_i(\bchi,\bxi)=\bchi^\top\bbq_i\bxi_i 
		$$
		where 
		\begin{align*}\bbq_i&=\begin{bmatrix} 
				\bzero_{(n(m+1) + \sum_{j=1}^{i-1}k_j) \times d_i} & \bzero_{(n(m+1) + \sum_{j=1}^{i-1}k_j) \times k_i}\\
				\bzero_{k_i \times d_i} & \bbi_{k_i}\\
				\bzero_{\sum_{j=i+1}^m k_j \times d_i} & \bzero_{\sum_{j=i+1}^m l_j \times k_i}.\\
			\end{bmatrix}
		\end{align*}
		In the end, problem~\eqref{eq.robust.problem} can be shortly written as
		\begin{align}
			\min\limits_{\bchi \in \mathcal{X}} \ &  \tilde{\bc}^\top\bchi \nonumber \\
			\text{s.t.} \ &  \sup\limits_{\bxi_i \in \Xi^i } \tilde{g}_i(\bchi,\bxi_i)\leq 0,\quad i \in [m], \nonumber\\
			&\bba\bchi =0\nonumber
		\end{align}
		where $\tilde{\bc}=(\bc,\bzero)$ and
		$$\bba=\begin{bmatrix}\bbi_{(n)} &-\bbi_{(n)}&\bzero_{(n\times n)} &&\ldots &\bzero_{(n\times n)}& \bzero_{(n\times \sum_{i\in[m]} k_i)}\\
			\bbi_{(n)} &\bzero_{(n\times n)} & -\bbi_{(n)}&\bzero_{(n\times n)} &\ldots &\bzero_{(n\times n)}& \bzero_{(n\times \sum_{i\in[m]}k_i)}\\
			\vdots & \vdots &\ddots &\ddots &&&\vdots\\
			\vdots & \vdots & &\ddots &\ddots&&\vdots\\
			\bbi_{(n)} &\bzero_{(n\times n)}& &\ldots &\bzero_{(n\times n)}&-\bbi_{(n)}& \bzero_{(n\times \sum_{i\in[m]} k_i)}\\
		\end{bmatrix}$$
		Thus, we are back at the biaffine case for the lifted variables $(\bchi,\bxi)$, and CP can be applied as long as the projections over $\Xi^i$ and $\mathcal{X}$ are easily attainable.
	\end{remark}
	
	\section{Numerical experiment: robust quadratic programming} \label{sec:numerics}
	\subsection{Introduction}
	In this section, we compare the numerical performance of our SGSP algorithm to the various approaches of \cite{ho2018online} and the standard cutting-plane algorithm. To do this, we consider an extension of the experiment in \cite{ho2018online}, solving problems
	\begin{align}
		\min_{\bx\in X} \ & \sup\limits_{\bz \in Z} g_0(\bx, \bz) \label{eq:qcqp} \\ 
		\text{s.t.} \ & g_i(\bx, \bz) \leq 0 && \forall \bz \in Z, \ i \in [m] \nonumber
	\end{align}
	where the objective and constraint functions are:
	$$
	g_i(\bx, \bz) = \left\| \left( \bbp_{i0} + \sum\limits_{k = 1}^K \bbp_{ik} z_k \right) \bx \right\|^2 + \bb_i^\top \bx + c_i.
	$$
	with $\bbp_{ik} \in \mathbb{R}^{L \times n}$, $\bb_i \in \mathbb{R}$ and $c_i \in \mathbb{R}$. We assume that $Z=\{\bz\in \real^K:\|\bz \|_2 \leq 1\}$ and $X=\{\bx\in \real^n:\|\bx \|_2 \leq 1\}$. The experiment is an extension of the one performed in \cite{ho2018online}, since the original considers problem with uncertainty only in the objective, while ours captures the more general setting with uncertain constraints.
	
	Note that in \eqref{eq:qcqp}, the functions $g_i(\bx, \cdot)$ are convex, and therefore the problem does not readily fit our framework. However, in the pessimization problem
	\begin{align}
		\sup\limits_{\|\bz \| \leq 1} g_i(\bx, \bz), \label{eq.trsp}
	\end{align}
	known as the \emph{trust region problem}, the function $g_i(\bx,\cdot)$ can be replaced by equivalent concave function, where the equivalence is in the sense of the same maximal value. For this, $g_i$ is first rewritten as
	$$
	g_i(\bx, \bz) = \bz^\top Q_i(\bx) \bz + 2 \br_i(\bx)^\top \bz + s_i(\bx)
	$$
	where $Q_i(\bx) = P_i(\bx)^\top P_i(\bx)$ with $P_i(\bx) \in \mathbb{R}^{n \times K}$ being a matrix whose columns are the vectors $\bbp_{ik} \bx$ for $k \in [K]$, $\br_i(\bx) = P_i(\bx)^\top P_{i0} \bx$ and $s_i(\bx) = \| P_{i0} \bx \|_2^2 - \bb_i^\top \bx - c_i$. For this function, by result of \cite{jeyakumar2014trust} problem \eqref{eq.trsp} can be reformulated as
	\begin{equation} \label{eq:qcqp.concave}
		\sup\limits_{\bz: \|\bz \| \leq 1} g_i(\bx, \cdot) = \sup\limits_{\bz: \|\bz \| \leq 1} \bz^\top (Q_i(\bx) - \lambda_{\text{max}}(Q_i(\bx)) \bbi )\bz + 2 \br_i(\bx)^\top \bz + s_i(\bx) + \lambda_{\text{max}}(Q_i(\bx)) 
	\end{equation}
	where $\lambda_{\max}(\cdot)$ denotes the largest eigenvalue of the argument matrix. Since 
	$$\bar{g}_i(\bx, \bz):=\bz^\top (Q_i(\bx) - \lambda_{\text{max}}(Q_i(\bx)) \bbi )\bz + 2 \br_i(\bx)^\top \bz + s_i(\bx) + \lambda_{\text{max}}(Q_i(\bx)),$$
	is convex-concave, using $\bar{g}_i(\bx, \bz)$ instead of $g_i(\bx,\bz)$ in the robust formulation is equivalent and our setting applies. Therefore, in this set of experiments, whenever using first-order methods, we solve the modified problem:
	\begin{align}
		\min_{\bx\in X} \ & \sup\limits_{\bz \in Z} \bar{g}_0(\bx, \bz) \label{eq:qcqp.transformed} \\ 
		\text{s.t.} \ & \bar{g}_i(\bx, \bz) \leq 0 && \forall \bz \in Z, \ i \in [m] \nonumber
	\end{align}
	We note that with respect $\bx$, problem \eqref{eq:qcqp.transformed} is semidefinite optimization problem. We are ready to state the four algorithms we compare:
	\begin{itemize}[\textbullet ]
		\item {\bf Cutting planes algorithm} applied directly to \eqref{eq:qcqp}, where pessimization subproblem \eqref{eq.trsp} is solved using a generic solver
		to find $\bz$ violating the constraints.
		\item {\bf The SGSP algorithm} applied to \eqref{eq:qcqp.transformed}. As part of this method, we first apply SGSP to find a Slater point for \eqref{eq:qcqp.transformed}, and given the obtained Slater point we run SGSP to solve problem \eqref{eq:qcqp.transformed}.
		\item {\bf The OCO algorithm} of \cite{ho2018online}, applied to \eqref{eq:qcqp.transformed}, where both the variables $\bx$ and $\bz$ are solved using online gradient descent (OGD).
		\item {\bf The FO-pessimization approach} of \cite{ho2018online} applied to \eqref{eq:qcqp.transformed}, where the primal problem is solved using OGD and for each constraint the worst-case $\bz$ is found by solving \eqref{eq:qcqp.concave} using a generic solver.
	\end{itemize}
	We will compare the algorithms on their speed of reducing the feasibility and optimality gaps in the sense of Theorem~\ref{theorem.sgsp.simple.convergence}. In the following Section, we describe the exact numerical setup and the results.
	
	\begin{remark}
		In the presented experiments we do not compare to the `nominal` approach suggested in \cite{bental2015oracle}.
		Indeed, in the implementation of \cite{bental2015oracle} presented in \cite{ho2018online}, the authors suggests that in the $k$-th iteration, the following problem will be solved:
		\begin{align*}
			\min_\bx \ & g_0(\bx, \bz^k_0) \\ 
			\text{s.t.} \ & g_i(\bx, \bz^k_i) \leq 0 &&  \ i \in [m].
		\end{align*}
		However, convergence of such a method requires for there to be a saddle point 
		$$
		\inf_{\bx\in X} \sup\limits_{\bz_i \in Z} g_i(\bx, \bz_i) = \sup\limits_{\bz_i \in Z} \inf_{\bx\in X} g_i(\bx, \bz_i),
		$$
		which does not necessarily exist due to the convexity of $g_i(\bx, \cdot)$. Thus, to solve the \eqref{eq:qcqp} problem using some kind of a `nominal' approach with updated $\bz^k_i$, one would need to either (i) use the semidefinite program \eqref{eq:qcqp.transformed} as the nominal oracle which contradicts the idea of solving `simple' problems per iteration, or (ii) use the dual-subgradient meta-algorithm of \cite{bental2015oracle} where $\bz_i$ is lifted to a semidefinite matrix and the original nominal oracle \eqref{eq:qcqp} is used w.r.t. $\bx$ -- however, running the dual step would require projections on the spectahedron which, again, contradicts the idea of solving `simple' problems, or (iii) use the dual-perturbation meta-algorithm of \cite{bental2015oracle} -- however, as we focus on deterministic algorithms, we do not include an implementation of \cite{bental2015oracle}. 
	\end{remark}
	\subsection{Experiment setting}
	We explored different problem sizes, with respect to $n$ the dimension of $\bx$, $m$ the number of constraints, $K$ the dimension of uncertainty $\bz_i$, and $\ell$ the dimension of the vector in the norm. For each problem size, we sampled the problem data for 50 problem instances as follows. First, each entry of $\bbp_{ik}$ and $\bb_i$ is sampled uniformly from interval $[-1, 1]$. Fixed value $c_i = -0.05$ is chosen deterministically to ensure Slater feasibility of the problem. Next, $\bbp_{ik}$ and $\bb_i$ are normalized as
	\begin{align*}
		\bbp_{ik} & = \frac{\bbp_{ik}}{S_{i1}} && \text{ where } S_{i1} = \left\| \left[ \bbp_{i0}^\top \ \cdots \ \bbp_{iK}^\top \right]^\top \right\|_{2,2} \\
		\bb_i & = \frac{\bb_i}{S_{i2}} && \text{ where } S_{i2} = \| \bb_i \|_2
	\end{align*}
	To compare the algorithms in a fair way, we will use the same starting point for all of them. This point will be the optimal solution to the nominal problem 
	\begin{align*}
		\min_\bx \ & g_0(\bx, \bzero) \\ 
		\text{s.t.} \ & g_i(\bx, \bzero) \leq 0 && \ i \in [m].
	\end{align*}
	The values $\bz_i$ are initialized as the zero vectors. We consider an $\epsilon$ tolerance of $0.001$ for both feasibility and optimality.
	
	Note that the first-order algorithms require a choice of step-size. The step-size for OGD is chosen to be ${2}/({\norm{\nabla_\bx g_i(\bx,\bz_i)}\sqrt{k}})$ and ${2}/({\norm{\nabla_{\bz_i} g_i(\bx,\bz_i)}\sqrt{k}})$ for the primal and dual steps at iteration $k$, respectively. These stepsizes correspond with the analysis of OGD given in  \cite[Theorem 3.4]{Hazan2016}. Similarly, the step sizes for SGSP is given by  $\tau^k={2}/({\norm{\nabla_{\bx} L(\bx,\bu)}\sqrt{k}})$ and $\theta^k_i={2}/({\norm{\nabla_{\bu_i} L(\bx,\bu)}\sqrt{k}})$ for the primal and dual steps at iteration $k$, respectively \footnote{We note that although the analysis of \cite{nedic2009subgradient} was done for a constant step-size, a similar analysis to the one shown in \cite[Theorem 3.4]{Hazan2016} can be done for SGSP, with similar  theoretical results.}. In both cases, the average solution $\bar{\bx}^N$ is computed as a weighted sum of the iterates $\bx^k$ with weights corresponding to the step-sizes. We note that these step-sizes were chosen instead of constant step-sizes, since they produced better results for all methods while retaining theoretical convergence guarantees.
	
	Table~\ref{tbl:sizes} describes the different settings in which the algorithms were run. We consider small, medium and large problem sizes, with either no constraints or three constraints for each. For each problem size we specify the time limit we gave the methods as well as the sampling frequency for the output (see explanation below).
	
	% \begin{table}\label{tbl:sizes}\caption{Sizes of tested problems}
		% 	\center
		% 	\begin{tabular}{lllllll}
			% 		\hline
			% 		Name&$n$& $K$& $L$ & $m$ &maximum allowed time & Output frequency \\
			% 		&&&&&(seconds) & ($\tilde{N}$ iterations)\\
			% 		\hline\hline
			% 		Small&10 & 10& 10 &0 &600 &100  \footnote{Cutting planes method records every iteration.}\\
			% 		&   & &  &3     &&\\
			% 		Medium &600 & 25 & 15 &0 &1200& 100\\
			% 		& & &  &3 & &\\
			% 		Large &3600 & 30 & 16 & 0 &3600& 100\\
			% 		& & && 3 & &\\
			% 	\hline
			% 	\end{tabular}
		% \end{table} }

\begin{table}
	\TABLE{Sizes of tested problems.\label{tbl:sizes}}{
		\begin{tabular}{lllllll}
			\hline
			Name&$n$& $K$& $L$ & $m$ &maximum allowed time & Output frequency \\
			&&&&&(seconds) & ($\tilde{N}$ iterations)\\
			\hline\hline
			Small&10 & 10& 10 &0 &600 &100  $^\dag$\\
			&   & &  &3     &&\\
			Medium &600 & 25 & 15 &0 &1200& 100\\
			& & &  &3 & &\\
			Large &3600 & 30 & 16 & 0 &3600& 100\\
			& & && 3 & &\\
			\hline
	\end{tabular}}{$^\dag$ Cutting planes method records every iteration.}
\end{table}

In order to measure the feasibility and the optimality gap for each method, for each parameter realization, we first define a constant $\tilde{N}$, such that every $\tilde{N}$ iterations statistics on the solution are gathered.  Specifically, for all $k\in \mathbb{N}$, let $\bx^{k\tilde{N}}$ be the solution obtained after iteration $k\tilde{N}$ of the algorithm, and let $T_k$ be the time it took to run these $k\tilde{N}$ iterations. The feasibility gap at iteration  $k\tilde{N}$ is given by
$$
\text{FG}_k:=\max_{i\in[m]} \max_{\bz_i\in Z} g_i(\bx^{k\tilde{N}},\bz_i).
$$
Defining the optimality gap to be infinity if the feasibility gap is larger than the defined $\epsilon$, the optimality gap ratio at iteration  $k\tilde{N}$ is given by
$$
\text{OGR}_k:=\frac{\delta_{\{p:p\leq \epsilon\}}(\text{FG}_k)\max_{\bz_i\in Z} g_0(\bx^{k\tilde{N}},\bz_i)-\text{LB}}{\text{LB}}.
$$
In the above formula, $\delta_{\{p:p\leq \epsilon\}}$ is the indicator function, and LB is the lower bound on the optimal solution obtained at the end of the cutting planes algorithm, by only considering the cuts added during the algorithm. Thus, for each time $t\geq 0$ we can record the minimal feasibility gap up to time $t$ as
$$
\min_{k:T_k\leq t} \text{FG}_k,
$$
and the minimal optimality gap ratio up to time $t$ as 
$$
\min_{k:T_k\leq t} \text{OGR}_k.
$$
All the code, available in the online repository \citep{our_repo}, was run using Python 3.7, with the CasADi package \citep{Andersson2019} as the optimization interface. The optimization problems were solved using Gurobi 9.0.0 \citep{gurobi}, with the exception of the trust-region subproblem \eqref{eq.trsp} that due to numerical difficulties was solved using the IPOPT solver \citep{wachter2006implementation}. The code was run on a PowerEdge R740xd server with two Intel Xeon Cold 6254 3.1GHz processors, each with 18 cores, 
and a total	RAM of 384GB. 

\subsection{Results}
In Figure~\ref{fig.opt_gap_noconstraints}, we show the optimality convergence for all methods for the small, medium, and large instances without constraints. In the small instances, the computational and memory requirements of the cutting planes algorithm are negligible, and its performance dominates the other methods. Among the first order methods, however, the SGSP algorithm attains the fastest convergence. 
%The slower convergence of the OCO and FO-Pess methods is due to their need to run binary search on the optimal value of the problem. 

When the problem instances become larger, the memory and computational requirements of the cutting planes algorithm become more significant, making its performance similar (medium instances) and then worse than the performance of the first-order methods. Among the first-order methods, our SGSP algorithm consistently dominates the other methods, although the differences with the OCO are rather small. Interestingly, we observe that for all methods it is the medium instances that keep their optimality gaps rather large for the longest. This might be related to the way we sample the problems where up to a certain point the problem size growth effect outweighs the `averaging out' effect of the large matrices that make it easier to find a high-quality solution.
\begin{figure} 
	\centering
	\includegraphics[width=\textwidth]{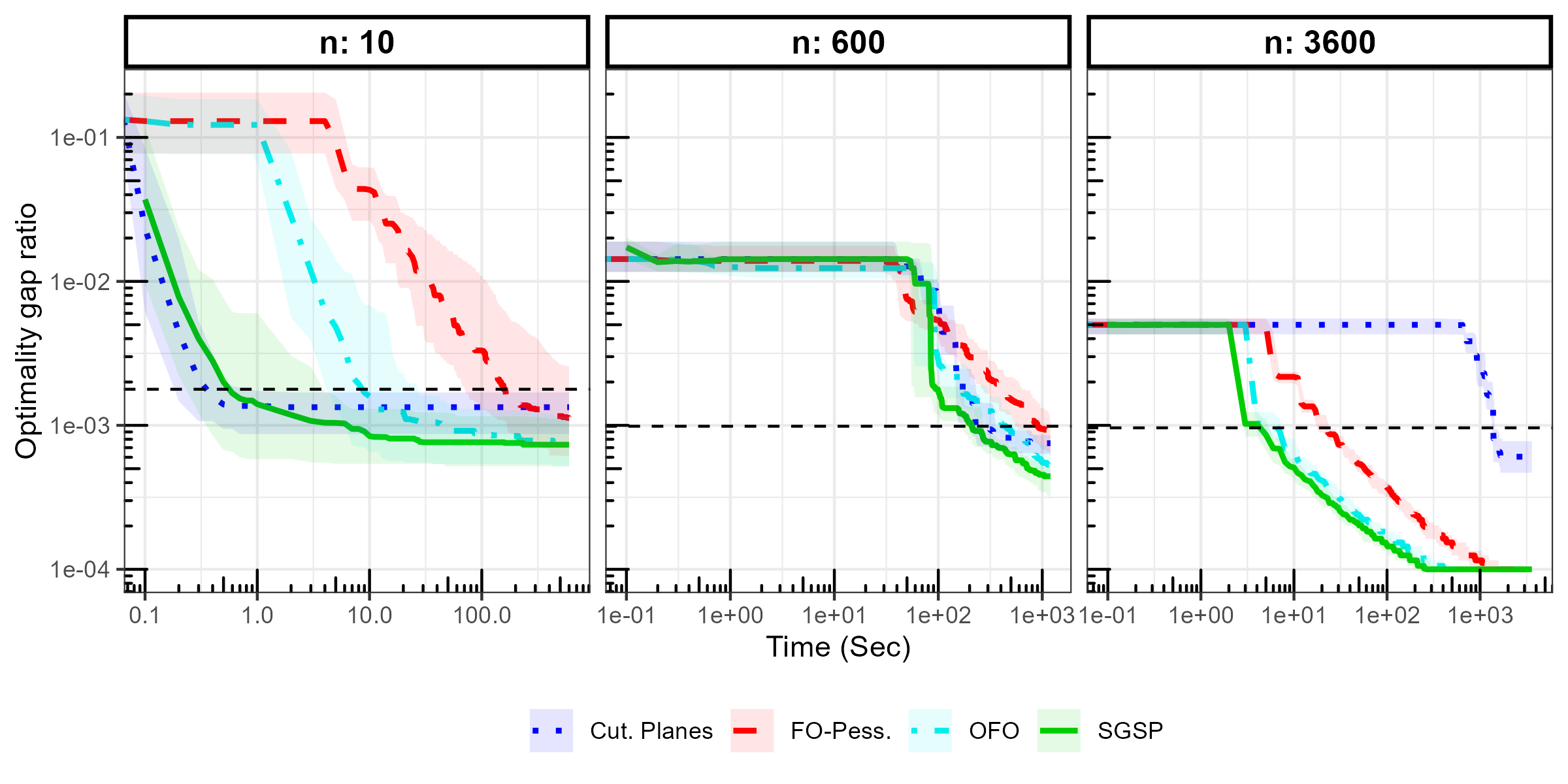}
	\caption{Optimality gap for all problem sizes with $m = 0$.} \label{fig.opt_gap_noconstraints}
\end{figure}

We now proceed to discuss the constrained problems. In Figures~\ref{all_graphs_n10_m4}-\ref{all_graphs_n3600_m4_k30_l16} we show the results for the small, medium, and large instances, respectively. We start with discussing the small examples. Similarly to the unconstrained case, the cutting planes algorithm is the fastest of all. Among the first-order methods, OFO is better than SGSP at finding feasible solutions fast, but having found them, it gets `stuck' on improving optimality due to the need to run binary search on the objective value. Specifically, the problem is with the bi-section iterations which do not have a feasible solution, but this infeasibility can only be identified after performing a very large number of iterations. The FO-Pess method is the slowest across all three measures.

\begin{figure} 
	\centering
	\includegraphics[width=\textwidth]{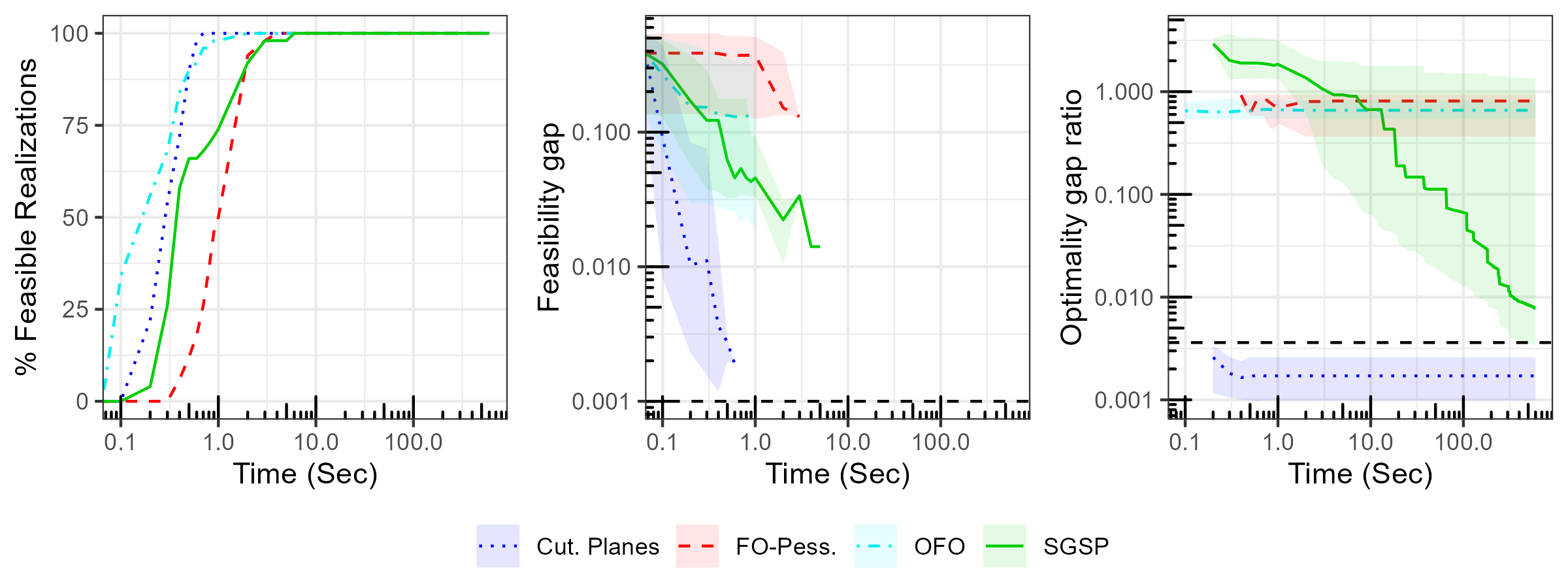}
	\caption{Small instances, $m = 3$. Percentage of instances with a feasible solution found, feasibility gap among infeasible instances, and optimality gap among the feasible instances.} \label{all_graphs_n10_m4}
\end{figure}

For the medium instances in Figure~\ref{all_graphs_n600_m4}, we would expect the first order methods to already perform better than the cutting planes algorithm. Indeed, the cutting planes algorithm becomes worse in decreasing the feasibility gap and the number of infeasible instances compared to all first-order methods. Among the feasible first-order methods, again, it is the SGSP algorithm that manages to reach optimality guarantees within the prescribed amount of time which are equivalent to those of the cutting planes, while OFO and FO-Pess do get stuck due to the need for running the binary search on the objective function value.

\begin{figure} 
	\centering
	\includegraphics[width=\textwidth]{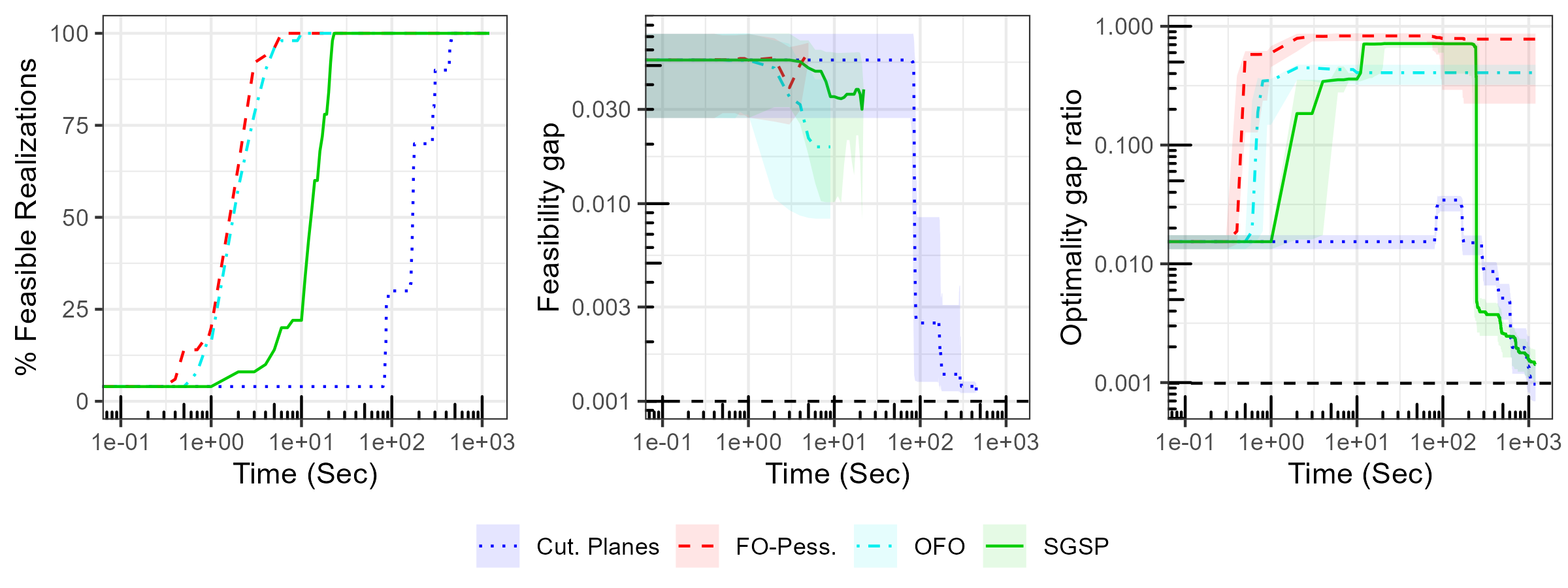}
	\caption{Medium instances, $m = 3$. Percentage of instances with a feasible solution found, feasibility gap among infeasible instances, and optimality gap among the feasible instances.} \label{all_graphs_n600_m4}
\end{figure}

Finally, for the largest instances in Figure~\ref{all_graphs_n3600_m4_k30_l16}, 
predictably, the cutting planes does not only have inferior performance to all first order based methods.
We therefore focus on the comparison of computational performance of the three first-order methods. We observe that all first order methods are able to find feasible solutions faster than in the medium instances. We again believe that this is due to the problem generation process. Among the first-order methods, SGSP is relatively the slowest one to find a feasible solution for all the instances. With respect to the optimality gap, we observe that OFO slightly dominates the SGSP. The SGSPs performance on these instances is affected by the fact that at each iteration, it requires gradient computations with respect to all the constraints, whereas for the OFO this is done only for the constraint with largest current value of the left-hand side. Since for large dimensional problems these computations are substantial, the SGSP performs slightly slower than OFO, while dominating over the FO-Pess method.

\begin{figure} 
	\centering
	\includegraphics[width=\textwidth]{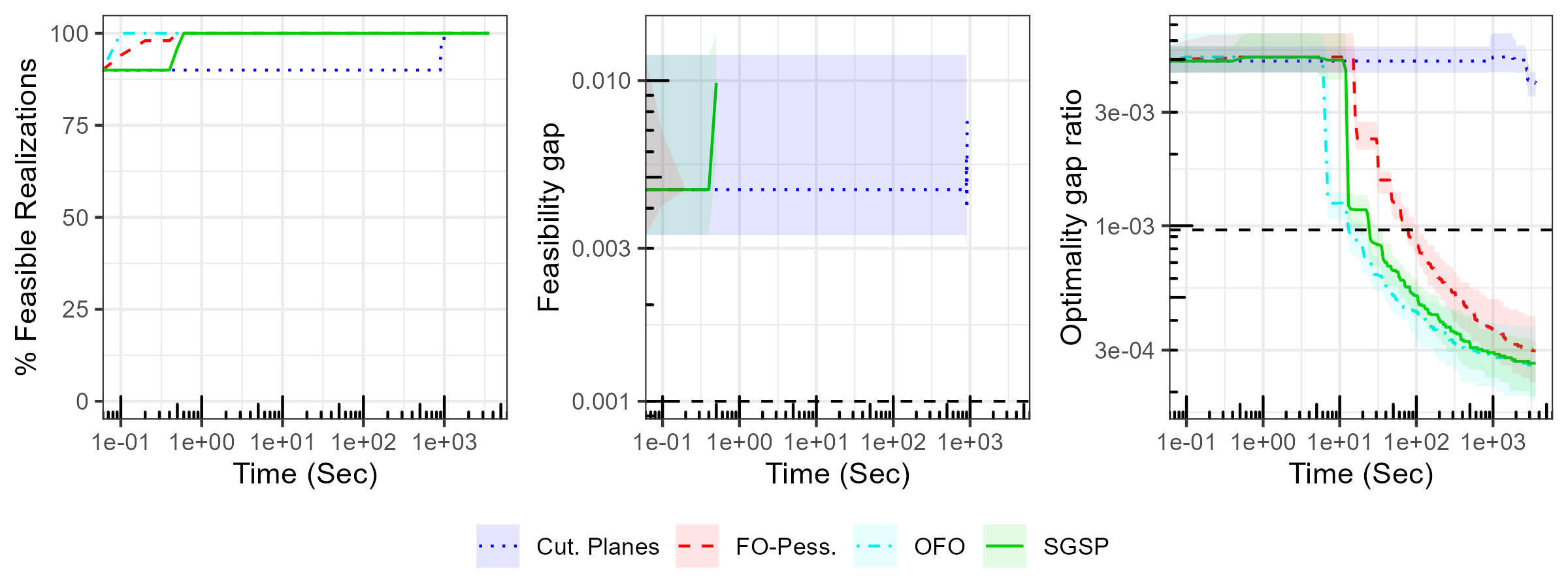}
	\caption{Large instances, $m = 3$. Percentage of instances with a feasible solution found, feasibility gap among infeasible instances, and optimality gap among the feasible instances.} \label{all_graphs_n3600_m4_k30_l16}
\end{figure}

\section{Conclusions} \label{sec:conclusions}
In this paper, we have proposed a first-order optimization approach to robust optimization problems based on a convex-concave saddle-point reformulation of the problem's Lagrangian. Our approach recovers the $\mathcal{O}(1/ \epsilon^2)$ convergence rate for general problems considered also by \cite{bental2015oracle,ho2018online}, and offers an improved $\mathcal{O}(1/ \epsilon)$ convergence guarantee for problems with a biaffine function structure. Similar to those algorithms, our method allows for a convenient parallelization of the computations related to different constraint functions and avoids problem size increase typical for the cutting planes and robust counterparts approaches. At the same time, our approach has the numerical benefit of avoiding a binary-search procedure for the optimal value of the objective as in \cite{ho2018online}, while providing a deterministic algorithm which does not have to solve the nominal problem, contrary to \cite{bental2015oracle}.

\section*{Acknowledgements} 
We thank Fatma K{\i}l{\i}n\c{c}-Karzan and Nam Ho-Nguyen for their discussions and sharing with us the implementation of their experiments. We also thank Shoham Sabach for his work at the early stage of this work and valuable suggestions.

\begin{APPENDIX}{}
	
	\section{Proofs} \label{appendix.proofs}
	\begin{proof}{Proof of Proposition~\ref{prop:eq_saddle_points}.}
		\textbf{Proof of $\Rightarrow$.} We first note that from the definition of $f_i$ and the construction of the lifted uncertainty set $U^i$ we have that for any $\blambda\in\real^m_+$ and ${\bx}\in X$ and $i\in[m]$. 
		\begin{align*}
			{\lambda}_i f_i({\bx}) &=  {\lambda}_i \sup\limits_{\bz_i\in Z^i} g_i(\bx,\bz_i) = \sup\limits_{\tilde{\bz}_i:(\tilde{\bz}_i,{\lambda}_i)\in U^i} \lambda_i g_i\left(\bx,\frac{\tilde{\bz}_i}{\lambda_i}\right).
		\end{align*}
		Indeed, if $(\bx^*,(\bu^*,\bw^*))$ is a saddle point of $\bar{L}$ then, defining $\bz_i^*= \tilde{\bz}^*_i / \lambda_i^*$ if $\lambda^*>0$ and $\bz_i^*=0$ otherwise, gives: 
		\begin{equation}
			L(\bx^*,(\blambda^*, \bw^*))=\bar{L}(\bx^*,(\bu^*, \bw^*)),
			\label{eq:same_value}
		\end{equation}
		and 
		\begin{align}
			& \bar{L}(\bx^*,(\bu^*, \bw^*)) \nonumber \\
			=&  \sup_{\bu\in U, \bw \in \mathbb{R}^r} \bar{L}(\bx^*,(\bu, \bw)) \nonumber \\
			=& \sup_{\blambda\in\real^m_+,\; \tilde{\bz}_i\in \lambda_iZ^i,\;i \in [m], \bw \in \mathbb{R}^r } \bc^\top{\bx}^* + \sum\limits_{i=1}^m  \lambda_i g_i\left(\bx^*,\frac{\tilde{\bz}_i}{\lambda_i}\right) + \bw^\top (\bba \bx^* - \bb)  \nonumber \\
			=& \sup_{\blambda\in\real^m_+, \bw \in \mathbb{R}^r} \bc^\top{\bx}^* + \sum\limits_{i=1}^m  \lambda_i\sup\limits_{\bz_i\in Z^i} g_i(\bx^*,\bz_i)  + \bw^\top (\bba \bx^* - \bb)  \nonumber\\
			=& \sup_{\blambda\in\real^m_+, \bw \in \mathbb{R}^r} \bc^\top{\bx}^*  + \sum_{i=1}^m {\lambda}_i f_i({\bx}^*)  + \bw^\top (\bba \bx^* - \bb) \nonumber \\
			=& \sup_{\blambda\in\real^m_+, \bw \in \mathbb{R}^r} L(\bx^*,(\blambda, \bw)), \label{eq:proof1eq1}
		\end{align}
		where the subsequent equalities follow from (i) the definition of $\bar{L}$, (ii) definition of $U^i$, (iii) the definition of $f_i$, and (iv) the definition of $L$. Moreover,
		\begin{align}
			\bar{L}(\bx^*,(\bu^*, \bw^*)) 
			=& \inf_{\bx\in X} \bar{L}(\bx,(\bu^*, \bw^*)) \nonumber \\
			=& \inf_{\bx\in X}\bc^\top\bx + \lambda^*_i g_i\left(\bx,\frac{\tilde{\bz}^*_i}{\lambda^*_i}\right) + \bw^{*\top} (\bba \bx - \bb) \nonumber \\
			\leq & \inf_{\bx\in X}\bc^\top\bx + \sum\limits_{i=1}^m\sup_{\tilde{\bz}_i:(\tilde{\bz}_i,\lambda^*_i)\in U^i}\lambda^*_i g_i\left(\bx,\frac{\tilde{\bz}_i}{\lambda^*_i}\right) + \bw^{*\top} (\bba \bx - \bb) \nonumber \\
			=& \inf_{\bx\in X}\bc^\top\bx+ \sum\limits_{i=1}^m\sup_{\bz_i\in Z^i}\lambda^*_i g_i(\bx,\bz_i)  + \bw^{*\top} (\bba \bx - \bb) \nonumber \\
			=& \inf_{\bx\in X} L(\bx,(\blambda^*, \bw^*))\label{eq:proof1eq7},
		\end{align}
		where the subsequent steps follow from (i) the definition of $\bar{L}$, (ii) the definition of $U^i$, and (iii) the definition of $L$. Combining \eqref{eq:same_value}, \eqref{eq:proof1eq1} and \eqref{eq:proof1eq7} we obtain that
		\begin{align*}
			\inf_{\bx\in X} \bar{L}(\bx,(\bu^*, \bw^*)) = & \inf_{\bx\in X} L(\bx,(\blambda^*, \bw^*)) \\
			= & L(\bx^*,(\blambda^*, \bw^*)) \\
			= & \bar{L}(\bx^*,(\bu^*, \bw^*))\\
			= & \sup_{\bu\in U, \bw \in \mathbb{R}^r} \bar{L}(\bx^*,(\bu, \bw)) \\
			= & \sup_{\blambda\in\real^m_+, \bw \in \mathbb{R}^r} L(\bx^*,(\blambda, \bw)),
		\end{align*}
		\emph{i.e.}, $(\bx^*,(\blambda^*, \bw^*))$ is a saddle point of $L$.
		
		\textbf{Proof of $\Leftarrow$.}
		We shall show that $(\bx^*, (\blambda^*, \bw^*))$ can be extended to a saddle point of the lifted Lagrangian $\bar{L}$. Note that defining $\tilde{\bz}_i=\lambda_i\bz_i$ we have that
		\begin{align*}
			\bc^\top \bx^* + \sum_{i = 1}^m \lambda_i^* f_i \left( \bx^*\right) + \bw^{*\top}(\bba \bx^* - \bb)
			& = \inf_{\bx \in X}  \bc^\top \bx + \sum_{i = 1}^m \lambda_i^* \max \limits_{(\tilde{\bz}_i, \lambda_i^*) \in U^i} g_i \left( \bx, \frac{\tilde{\bz}_i}{\lambda_i^*} \right) + \bw^{*\top}(\bba \bx - \bb) \\
			& = \max\limits_{\substack{(\tilde{\bz}_i, \lambda_i^*) \in U^i,\;i\in[m]}} \inf_{\bx \in X}  \bc^\top \bx + \sum_{i = 1}^m \lambda_i^* g_i \left( \bx, \frac{\tilde{\bz}_i}{\lambda_i^*} \right) + \bw^{*\top}(\bba \bx - \bb) \\
			& = \inf_{\bx \in X}  \bc^\top \bx + \sum_{i = 1}^m \lambda_i^* g_i \left( \bx, \frac{\hat{\bz}_i}{\lambda_i^*} \right) + \bw^{*\top}(\bba \bx - \bb).
		\end{align*}
		The first equality is due to $(\bx^*,(\blambda^*, \bw^*))$ being a saddle point of $L$ and the definition of $f_i$. The third equality follows from Sion's theorem, applicable due to boundedness of $\{ (\tilde{\bz}, \lambda_i) \in U^i: \ \lambda_i = \lambda_i^*\}$ and where we define $\hat{\bz}_i$ as a (necessarily-existing) maximizer:
		$$
		(\hat{\bz}_i)_{i \in[m]} \in \arg\max_{\tilde{\bz}_i:(\tilde{\bz}_i, \lambda_i^*) \in U^i} \left\{ \inf_{\bx \in X}  \bc^\top \bx + \sum_{i = 1}^m \lambda_i^* g_i \left( \bx, \frac{\tilde{\bz}_i}{\lambda_i^*} \right) + \bw^{*\top}(\bba \bx - \bb) \right\}.
		$$
		Thus, defining $\bu^*_i=(\hat{\bz}^*_i,\lambda^*_i)$ we have that
		\begin{equation*}
			L(\bx^*,(\blambda^*, \bw^*))=\bar{L}(\bx^*, (\bu^*, \bw^*))=\inf_{\bx \in X}  \bc^\top \bx + \sum_{i = 1}^m \lambda_i^* g_i \left( \bx, \frac{\hat{\bz}_i^*}{\lambda_i^*}\right) + \bw^{*\top}(\bba \bx - \bb) =\inf_{\bx\in X} \bar{L}(\bx,(\bu^*, \bw^*))
		\end{equation*}
		Moreover,
		\begin{align*}
			\bar{L}(\bx^*,(\bu^*, \bw^*))=L(\bx^*,(\blambda^*, \bw^*))&=\bc^\top \bx^* + \sum_{i = 1}^m \lambda_i^* f_i(\bx^*) + \bw^{*\top}(\bba \bx^* - \bb) \\
			% & =  \bc^\top \bx^* + \sup\limits_{\blambda \geq 0, \bw}\sum_{i = 1}^m \lambda_i \sup\limits_{\bz \in Z^i} f_i(\bx^*) + \bw^{\top}(\bba \bx - \bb) \\
			& = \bc^\top \bx^* + \sup\limits_{\blambda \geq \real^m_+, \bw\in \real^r}\sum_{i = 1}^m \lambda_i \sup\limits_{\bz_i \in Z^i} g_i(\bx^*, \bz) + \bw^{\top}(\bba \bx^* - \bb) \\
			%& = \bc^\top \bx^* + \sup\limits_{\blambda \geq 0, \bw}\sum_{i = 1}^m \lambda_i \sup\limits_{\bz \in Z^i} g_i(\bx^*, \bz) + \bw^{\top}(\bba \bx^* - \bb) \\
			&=  \bc^\top \bx^* + \sup\limits_{\bu_i  = (\lambda_i, \tilde{\bz}_i) \in U^i,\;i\in[m], \bw} \sum_{i = 1}^m \lambda_i g_i \left( \bx^*, \frac{\tilde{\bz}_i}{\lambda_i} \right) + \bw^{\top}(\bba \bx^* - \bb) \\
			&=\max_{\bu_i\in U^i,\;i\in [m], \bw\in \real^r}\bar{L}(\bx^*,(\bu, \bw)),
		\end{align*}
		where we used the fact that $(\bx^*, (\blambda^*, \bw^*))$ is a saddle point of the original Lagrangian $L$, and the definition of $f_i$, $L$, and $\bar{L}$. Thus, we showed that $(\bx^*, (\bu^*, \bw^*))$ is a saddle point of $\bar{L}$. \hfill \qedsymbol
	\end{proof}

	\begin{proof}{Proof of Lemma~\ref{lem:subdif_def}.}
		\begin{itemize}
			\item[(i)]  Applying \cite[Lemma 2.3]{combettes2018perspective} and using the fact that for any $\bu_i=(\tilde{\bz}_i,\lambda_i)\in U^i$ we have that either both $\lambda_i=0$ and $\tilde{\bz}_i=0$, or $\bz_i=\tilde{\bz}_i/\lambda_i\in Z^i$.
			\item[(ii)] We use the fact that $\bz_i=\tilde{\bz}_i/\lambda_i$ for all $i$ such that $\lambda_i>0$, and define
			$$
			\bd_x\coloneqq\sum_{i=1}^m \lambda_i\tilde{\bd}_{\bx,i}=\sum_{i:\lambda_i>0} \lambda_i\tilde{\bd}_{\bx,i}\in \sum_{i:\lambda_i>0} \lambda_i\partial_\bx g_{i}(\bx,{\bz}_i)=\sum_{i:\lambda_i>0} \lambda_i\partial_\bx g_{i}\left(\bx,\frac{\tilde{\bz}_i}{\lambda_i}\right).
			$$ 
			Thus, we see that \eqref{eq:d_x} holds, i.e., $\bv_x\in \partial_\bx\bar{L}(\bx,\bu)$. For all $i$ such that $\lambda_i>0$ define 
			$$
			\bd_i\coloneqq \tilde{\bd}_{\bz,i}\in \partial_{\bz_i } \left(-g_i(\bx,\bz_i)\right)=\partial_{\bz_i } \left(-g_i\left(\bx,\frac{\tilde{\bz}_i}{\lambda_i}\right)\right).
			$$
			For such $\bd_i$ we  have that
			$$
			\bv_i=(\bd_{\bz,i},-g_i(\bx,\bz_i)-\bz_i^\top\bd_{\bz,i})=\left(\bd_i,-g_{i}\left(\bx,\frac{\tilde{\bz}_i}{\lambda_i}\right)-\frac{\tilde{\bz}_i^\top\bd_i}{\lambda_i}\right)\in\partial_{\bu_i} \left(-\bar{L}(\bx,\bu)\right).
			$$
			i.e, the first case of \eqref{eq:d_z} holds. 
			
			Finally, for all $i$ such that $\lambda_i=0$, using the fact that $\bz_i=\bzero$ define
			$$
			\bd_i\coloneqq\bd_{\bz,i}\in \partial_{\bz_i } \left(-g_i(\bx,\bzero)\right)\subseteq\cup_{\bz_i\in Z^i}\partial_{\bz_i} \left(-g_{i}\left(\bx,\bz_i\right)\right) + T_{Z^i}(\bz_i)^*,
			$$
			where the tangent cone $T_{Z^i}(\bz_i)^*$ always includes the zero vector. Moreover, defining $\phi_i\coloneqq-g_i(\bx,\bzero)$ we have that
			\begin{align*}
				(-g_i)^*(\bx,\bd_i)&=\sup_{\bzeta_i} \{\bzeta_i^\top\bd_i+g(\bx,\bzeta_i)\}\\
				&=\sup_{\bzeta_i} \{\bd_i^\top(\bzeta_i-\bzero)-g_i(\bx,\bzero)+g_i(\bx,\bzeta_i)\}+g_i(\bx,\bzero)\\
				&\leq g_i(\bx,\bzero)=-\phi_i\\
			\end{align*}
			where the first equality follows from the definition of convex conjugate, and the inequality follows from the convexity of $-g_i(\bx,\cdot)$ and the fact that $\bd_i\in \partial_{\bz_i } \left(-g_i(\bx,\bzero)\right)$. Thus, meeting the definition \eqref{eq:d_z} we have that for the case where $\lambda_i=0$ we obtain that $\bv_i=(\bd_i,\phi_i)\in \partial_{\bu_i} \left(-\bar{L}(\bx,\bu)\right)$. \hfill \qedsymbol
		\end{itemize}
	\end{proof}
	
	\begin{proof}{Proof of Proposition~\ref{prop:proj_firstformUi}.}
		The projection over set $U^i$ is given by computing the minimizer of the following optimization problem
		\begin{equation}\min\left\{\frac{1}{2}\norm{\bu_i-\bv}^2: \bv\in U^i\right\}=\min\left\{\frac{1}{2}\norm{\tilde{\bz}_i-\bzeta}^2+\frac{1}{2}(\lambda_i-\mu)^2: \bzeta\in \mu Z^i,\;\mu\geq 0\right\}
			%=\argmin\left\{\mu\delta_{Z^i}\left(\frac{{\bzeta}}{\mu}\right)+\frac{1}{2}\norm{\tilde{\bz}_i-\bzeta}^2+\frac{1}{2}(\lambda_i-\mu)^2\right\}=\prox_{\tilde{\delta}_{Z^i}}(\tilde{\bz}_i,\lambda_i),
			\label{eq:proj_problem}\end{equation}
		
		It is clear that if $\bu_i\in U^i$ its projection onto $U^i$ is the vector $\bu_i$ itself. Otherwise, we can rewrite the projection problem as follows
		\begin{equation*}\min\left\{\norm{\tilde{\bz}_i-\bzeta}^2+(\lambda_i-\mu)^2: \bzeta\in \mu Z^i,\;\mu\geq 0\right\}=\min\left\{\norm{\tilde{\bz}_i-\mu\bz_i}^2+(\lambda_i-\mu)^2:\bz_i\in Z^i,\;\mu\geq 0\right\}.\end{equation*}
		Computing this minimum first over $\bz_i\in Z^i$ we obtain that if $\mu>0$ then
		$$\argmin\left\{\norm{\tilde{\bz}_i-\mu\bz_i}^2:\bz_i\in Z^i\right\}=\argmin\left\{\norm{\frac{\tilde{\bz}_i}{\mu}-\bz_i}^2:\bz_i\in Z^i\right\}=P_{Z^i}\left(\frac{\tilde{\bz}_i}{\mu}\right)$$
		in which case the optimal ${\bzeta}$ is given in by $\mu P_{Z^i}\left(\frac{\tilde{\bz}_i}{\mu}\right)$.
		Otherwise, if $\mu=0$, then all points $\bz_i\in Z^i$ are optimal and the optimal $\bzeta$ is $\bzero$.
		Moreover, 
		\begin{equation*}\min\{\norm{\bu_i-\bv}^2: \bv\in U^i\}
			=\min\left\{\inf_{\mu> 0} {\psi}_i(\mu),\norm{\bu_i}^2\right\}	 \end{equation*}
		since 
		\begin{align*}
			{\psi}_i(\mu) & =\inf_{\bzeta\in\mu Z^i}\frac{1}{2}\norm{\tilde{\bz}_i-\bzeta}^2+\frac{1}{2}(\lambda_i-\mu)^2 \\ 
			& = \frac{\mu^2}{2}\norm{\frac{\tilde{\bz}_i}{\mu}-P_{Z^i}\left(\frac{\tilde{\bz}_i}{\mu}\right)}^2+\frac{1}{2}(\lambda_i-\mu)^2 \\
			& =\frac{1}{2}\mu^2\dist\left(\frac{\tilde{\bz}_i}{\mu},Z^i\right)^2+\frac{1}{2}(\lambda_i-\mu)^2.
		\end{align*}
		We first show that $\psi_i$ is convex on the domain $\mu>0$. Indeed, since \eqref{eq:proj_problem} is jointly convex in $\bzeta$ and $\mu>0$, $\psi_i$ is a convex function as a partial minimization of a convex problem. In particular, $\psi_i$ is strongly convex since it is a sum of a convex and strongly convex functions. 
		Using \cite[Proposition 18.22]{bauschke2011convex} and the chain rule, we obtain the following derivative of ${\psi}_i$
		$$ {\psi}'_i(\mu)=\mu\dist\left(\frac{\tilde{\bz}_i}{\mu},Z^i\right)^2-\tilde{\bz}_i^\top\left(\frac{\tilde{\bz}_i}{\mu}-P_{Z^i}\left(\frac{\tilde{\bz}_i}{\mu}\right)\right)+\mu-\lambda_i=\mu\norm{P_{Z^i}\left(\frac{\tilde{\bz}_i}{\mu}\right)}^2-\tilde{\bz}_i^\top P_{Z^i}\left(\frac{\tilde{\bz}_i}{\mu}\right)+\mu-\lambda_i.$$
		%-\mu P_{Z^i}\left(\frac{\tilde{\bz}_i}{\mu}\right)^\top\left(\frac{\tilde{\bz}_i}{\mu}-P_{Z^i}\left(\frac{\tilde{\bz}_i}{\mu}\right)\right)+\mu-\lambda_i$$	
		Due to the strong convexity of $\psi(\mu)$, its infimum over $\mu>0$ is attained if and only if there exists $\mu^*>0$ such that ${\psi}'_i(\mu^*)=0$.
		%Since $\bzero\in Z^i$, it follows from \cite[Theorem 9.8]{beck2014introduction} that for any $\mu>0$ it holds that
		Since $\lim_{\mu\rightarrow\infty}{\psi}'_i(\mu)=\infty$, and due the monotonicity of the gradient of convex functions, such a $\mu^*>0$ exists if and only if 
		$\lim_{\mu\rightarrow 0^+}{\psi}'_i(\mu)<0$ or equivalently $\lim_{\alpha\rightarrow \infty} \tilde{\bz}_i^\top P_{Z^i}(\alpha\tilde{\bz}_i)>-\lambda_i$. In the rest of the proof we show that the latter is always true. 
		
		We first note that since $\bzero\in Z^i$ it follows from \cite[Theorem 9.9]{beck2014introduction} that for any $\alpha>0$ 
		$$\alpha\tilde{\bz}_i^\top P_{Z^i}(\alpha\tilde{\bz}_i)\geq \norm{ P_{Z^i}(\alpha\tilde{\bz}_i)}^2\geq 0,$$
		and thus,
		$\lim_{\alpha\rightarrow \infty} \tilde{\bz}_i^\top P_{Z^i}(\alpha\tilde{\bz}_i)\geq 0$ which implies that the condition is satisfied for all $\lambda_i>0$.
		Moreover, it follows from the definition of the support function that 
		$\tilde{\bz}_i^\top P_{Z^i}(\alpha\tilde{\bz}_i)\leq \sigma_{Z^i}(\tilde{\bz}_i)$ for any $\alpha>0$.
		We will now show that
		$\lim_{\alpha\rightarrow \infty} P_{Z^i}(\alpha\tilde{\bz}_i)\in \partial\sigma_{Z^i}(\tilde{\bz}_i)=\argmax_{\bp\in Z^i}\tilde{\bz}_i^\top\bp$.
		From optimality of the projection we have that $\bq=P_{Z^i}(\alpha\tilde{\bz}_i)$ if and only if
		$\by\in \alpha\tilde{\bz}_i-\partial\delta_{Z^i}(\by)$, so $\alpha\tilde{\bz}_i-\by\in \partial\delta_{Z^i}(\by)$. 
		Moreover, note that by definition of the indicator function and the subdifferential, if $\by \in\partial\delta_{Z^i}(\by)$ the $\alpha\by\in \partial\delta_{Z^i}(\by)$ for all $\alpha>0$, and thus $\tilde{\bz}_i-\frac{\by}{\alpha}\in \partial\delta_{Z^i}(\by)$. 
		Since $\bp\in Z^i$ is bounded, it follows that as $\alpha\rightarrow \infty$ we have that $\tilde{\bz}_i\in \partial\delta_{Z^i}(\by)$, which is equivalent to $\by\in  \argmax_{\bp\in Z^i}\tilde{\bz}_i^\top\bp$. Thus, we established that the condition $\lim_{\alpha\rightarrow \infty} \tilde{\bz}_i^\top P_{Z^i}(\alpha\tilde{\bz}_i)>-\lambda_i$ is equivalent to $\sigma_{Z^i}(\tilde{\bz}_i)>-\lambda$, concluding our proof. \hfill \qedsymbol
	\end{proof} 
	
	\begin{proof}{Proof of Proposition~\ref{proposition:bounded.omega}.}
		We look at the saddle point problem \eqref{eq.sp.harder} where the optimization problem is done over the nonrestricted sets $U^{i,l}$, that is
		\begin{align}
			\min_{\bomega_i } \max_{\substack{\bu_{i,l}\in {U}^{i,l}, \ l\in[s_i]\\\utilde{\lambda}\leq \lambda_{i,s_i} \leq \tilde{\lambda}}}\tilde{g}_i( \bx^* ,\bu_{i,s_i}) + \sum\limits_{l = 1}^{s_i - 1} \bomega_{i, l}^\top (\bu_{i,l} - \bu_{i, s_i}).\label{eq.sp.harder2}
		\end{align}
		We will run the proof in three steps:
		\begin{itemize}
			\item Proving that \eqref{eq.sp.harder} has a saddle point.
			\item Proving that any saddle point of \eqref{eq.sp.harder2} is also a saddle point of \eqref{eq.sp.harder}.
			\item Proving the boundedness of $\bomega^*$ for the problem without the restriction.
		\end{itemize} 
		It will therefore follow that after restricting $\bomega$ problem \eqref{eq.sp.harder} still has saddle points.
		%We begin by defining 
		%$$\bar{U}_{i,s^i}=\{\bu_{i,s_i}=\left\{\tilde{\bz}_{i,s_i},\lambda_{i,s_i}):\bu_{i,s_i}\in\tilde{U}^{i,s_i}, \lambda_{i,s_i}\leq %\tilde{\lambda}\right\}.$$
		We begin with the first step. Indeed, since
		$$\sup_{\bu_i\in U^i,\; \utilde{\lambda}\leq \lambda_i\leq \tilde{\lambda}}\tilde{g}_i(\bx^*,\bu_{i})\leq \sup_{\bu_i\in U^i}\tilde{g}_i(\bx^*,\bu_{i})<\infty$$
		by the same logic as \eqref{eq:strong_dual}, \eqref{eq.sp.harder2} must have a saddle point.
		
		Moving to the second step, we use the necessary and sufficient optimality conditions of the saddle point formulation to obtain that $(\bomega^\dag,\tilde{\bu}^\dag)$ is a saddle point of \eqref{eq.sp.harder2} where $\bu_{i,l}\in U^{i,l}$ if and only if
		\begin{align*}
			\bu^\dag_{i,l} & = \bu^\dag_{i,s_i} \quad &&l \in [s_i - 1] \\
			0 &\in\bomega^\dag_{i, l} - \partial \delta_{{U}^{i,l}}(\bu^\dag_{i,l}), \quad&& l \in [s_i - 1] \\
			0&\in -\partial_{\bu}(-\tilde{g}_i(\bx^*, \bu^\dag_{i,s_i})) - \sum_{l \in [s_i - 1]} \bomega^\dag_{i,l} - \partial \delta_{\bar{U}^{i,s_i}}(\bu^\dag_{i,s_i}),&&
		\end{align*}
		where $\bar{U}^{i,s_i}=\{\bu_{i,s_i}\in {U}_{i,s}: \utilde{\lambda}\leq\lambda_{i,s_i}\leq \tilde{\lambda}\}\subseteq \tilde{U}^{i,s_i}$. Since $0\leq\utilde{\lambda}\leq \tilde{\lambda} \leq \bar{\lambda}$ it follows that $\bu^\dag_{i,l}\in \tilde{U}^{i,l}\subseteq {U}^{i,l}$ and thus, $\bomega^\dag_{i, l}\in  \partial \delta_{U^{i,l}}(\bu^\dag_{i,l})\subseteq  \partial \delta_{\tilde{U}^{i,l}}(\bu^\dag_{i,l})$  and $- \sum_{l \in [s_i - 1]} \bomega^\dag_{i,l}\in  \partial \delta_{\bar{U}^{i,s_i}}(\bu^\dag_{i,s_i})+\partial_{\bu}(-\tilde{g}_i(\bx^*, \bu^\dag_{i,s_i}))$. Therefore,
		\begin{align*}
			\bu^\dag_{i,l} & = \bu^\dag_{i,s_i} \quad &&l \in [s_i - 1] \\
			0&\in \bomega^\dag_{i, l} - \partial \delta_{\tilde{U}^{i,l}}(\bu^\dag_{i,l}), \quad &&l \in [s_i - 1] \\
			0&\in-\partial_{\bu}(-\tilde{g}_i(\bx^*, \bu^\dag_{i,s_i})) - \sum_{l \in [s_i - 1]} \bomega^\dag_{i,l} - \partial \delta_{\bar{U}^{i,s_i}}(\bu^\dag_{i,s_i}),&&
		\end{align*}
		which are exactly the optimality conditions when using $\tilde{U}{^i,l}$ instead of ${U}^{i,l}$, and so  $(\bomega_i^\dag,\tilde{\bu}_i^\dag)$  is also a saddle point of the restricted problem.
		
		We now move to the last step, showing that $\bomega^\dag$ must be bounded. Note that if for some $i\in[m]$ we have that $\lambda_{i,s_i}^\dag=0$, then the pair $(\bomega_i^\dag,\tilde{\bu}_i^\dag)=(\bzero, \bzero)$ is a saddle point for the $i$th element of the sum, and so restricting the norm of $\bomega_i$ is possible. We therefore continue with bounding $\bomega_i$ for the case where $\tilde{\lambda}>0$. By definition, we have that $\bu^\dag_{i,s_i}=(\bz^\dag_{i,s_i}{\lambda}^\dag_{i,s_i},{\lambda}^\dag_{i,s_i})$ where $0\leq\utilde{\lambda}\leq {\lambda}^\dag_{i,s_i}\leq \tilde{\lambda}\leq \bar{\lambda}$. Also note that it must be that
		$ \bu^\dag_{i,s_i}=\bu^\dag_{i,l}$, otherwise the minimization over $\bomega_i$ would yield minus infinity. Thus, we have that
		$\bz^\dag_{i,s_i}\in \cap_{l=1}^{s_i} Z^{i,l}$.
		Finally, we have that
		\begin{align}
			0&\geq \lambda_{i,s_i}^\dag g_i(\bx^*,\bz^\dag_{i,s_i})  \nonumber \\
			&=\tilde{g}_i( \bx^* ,\bu^\dag_{i,s_i}) +  \sum\limits_{l = 1}^{s_i - 1} (\bomega^\dag_{i, l})^\top (\bu^\dag_{i,l} - \bu^\dag_{i, s_i}) \nonumber \\
			&=\max_{\bu_{i,l}\in \tilde{U}^{i,l}, \ \lambda_{i,s_i} = \lambda_{i,s_i}^\dag}   \tilde{g}_i( \bx^* ,\bu_{i,s_i}) + \sum\limits_{l = 1}^{s_i - 1} (\bomega^\dag_{i, l})^\top (\bu_{i,l} - \bu_{i, s_i}) \label{eq:bound_mu1}\\
			&\geq \max_{\bz_{i,l}\in {Z}^{i,l}}\lambda_{i,s_i}^\dag\left(g_i(\bx^*,\bz_{i,s_i})+ \sum\limits_{l = 1}^{s_i - 1} (\bnu^\dag_{i, l})^\top (\bz_{i,l} - \bz_{i, s_i})\right) \nonumber \\
			&\geq \lambda_{i,s_i}^\dag\left(g_i(\bx^*,\bzero)+\epsilon_i\norm{\bnu^\dag_{i, l'}}\right), \forall l'\in[s_i-1]\label{eq:bound_nu1}
		\end{align}
		where the first inequality comes from the feasibility of $\bx^*$ for all $\bz_i\in Z^i=\cap_{l=1}^{s_i} Z^{i,l}$  and the non-negativity of $\lambda_{i,s_i}^\dag$, the first equality follows from  $ \bu^\dag_{i,s_i}=\bu^\dag_{i,l}$, the second equality follows from the optimality of  $\bu^\dag$ and $\bomega^\dag$ for the saddle point problem, the second inequality follows from choosing $\bu_{i,l}=(\bz_{i,l}\lambda_{i,s_i}^\dag,\lambda_{i,s_i}^\dag)$, and the third inequality follows from choosing $\bz_{i,l}=0$ for all $l\in [s_i]/\{l'\}$ and $\bz_{i,l'}=\epsilon_i\bnu^\dag_{i, l'}$.
		Therefore, if $\lambda_{i,s_i}^\dag>0$ it follows from \eqref{eq:bound_nu1} that
		$\norm{\bnu^\dag_{i, l}}\leq -g_i(\bx^*,\bzero) / \epsilon_i \leq \bar{\mu}_{i} /\epsilon_i$.
		
		Similarly, taking an arbitrary $l'\in [s_i-1]$, we can bound \eqref{eq:bound_mu1} from below by choosing $\bu_{i,l}=(\bzero,{\lambda}^\dag_{i,s_i})$ for all $l\in[s_i]\setminus\{l'\}$ and
		$\bu_{i,l'}=(\bzero,\lambda_{i,l'}^\star)$.
		\begin{align}
			0&\geq 
			\max_{ \bu_{i,l}\in {U}_{i,l}, \ \utilde{\lambda}\leq \lambda_{i,s_i} \leq \tilde{\lambda}}   \tilde{g}_i( \bx^* ,\bu_{i,s_i}) + \sum\limits_{l = 1}^{s_i - 1} (\bomega^\dag_{i, l})^\top (\bu_{i,l} - \bu_{i, s_i}) \nonumber \\ %\label{eq:bound_mu2}
			&\geq \lambda_{i,s_i}^\dag g_i(\bx^*,\bzero)+ \sum_{l=1}^{s_i-1}\mu^\dag_{i, l}(\lambda_{i,l}-\lambda_{i,s_i}^\dag)\nonumber \\
			&= \lambda_{i,s_i}^\dag(g_i(\bx^*,\bzero)-\mu^\dag_{i, l'})+\mu^\dag_{i, l'}\lambda_{i,l'}^\star. \nonumber %\label{eq:bound_mu3}
		\end{align}
		Since $\lambda_{i,l'}^\star$ can be taken to infinity, it follows that the equality holds only if $\mu^\dag_{i, l'}\leq 0$.
		Moreover, choosing $\lambda_{i,l'}^\star=0$ implies that $-\mu^\dag_{i, l'}\leq  -g_i(\bx^*,\bzero)\leq \bar{\mu}_i$. Since $l'$ was arbitrarily chosen the proof is complete. \hfill \qedsymbol
	\end{proof}
	
	\begin{proof}{Proof of Proposition~\ref{prop.SGSP.convergence}.} In the proof the the proposition, claim \eqref{eq.SGSP.EB.result1} plays the key role, from which \eqref{eq.SGSP.EB.result2} and \eqref{eq.SGSP.EB.result3} follow.
		
		{\bf Proof of \eqref{eq.SGSP.EB.result1}} Denoting $G_\bchi$, $G_{\bu, i}$, $G_{\bw}$ as bounds on the subgradients $\bv_\bchi^k \in \partial_\bchi \breve{L}(\bchi^{k-1},\by^{k-1})$, $\bv_{\tilde{\bu}}^k \in \partial_{\tilde{\bu}}(-\breve{L}(\bchi^{k-1},\by^{k-1}))$, $\bv_\bw^k \in \partial_\bw \breve{L}(\bchi^{k-1},\by^{k-1})$ used throughout the algorithm, the first of the corollary follows directly from \cite[Lemmas 3.1 and 3.2]{nedic2009subgradient}. It is left to prove that under the chosen assumptions these bounds exist and are equal to the stated values. 
		
		We begin with the primal variables $\bchi$. Denoting $\bu^{k-1}_{i,l}=(\tilde{\bz}_{i,l}^{k-1},\lambda_{i,l}^{k-1})$ and defining $\bz_{i,s_i}^{k-1}$ as 
		$$
		%\begin{equation}\label{eq:z_defin}
		\bz^k_i=\begin{cases}\frac{\tilde{\bz}^k_i}{\lambda^k_i}, &\lambda_i>0\\
			\bzero, &\lambda_i=0.
		\end{cases}
		$$
		%\end{equation}
		we have that 
		$$
		\bv_\bchi^k=\begin{bmatrix} \bc +\bba^\top \bw^{k-1} +\sum_{i=1}^m\lambda_{i,s_i}^{k-1}\bd^k_{\bx,i}\\
			(\bu^{k-1}_{1,s_1}-\bu^{k-1}_{1,s_1-1})\\
			\vdots\\
			(\bu^{k-1}_{m,s_m}-\bu^{k-1}_{m,s_m-1})\end{bmatrix}\in\partial_\bchi\breve{L}(\bchi^{k-1},\bu^{k-1}),
		$$
		where $\bd^k_{\bx,i}\in \partial_\bx g_i(\bx^{k-1},\bz^{k-1}_{i,s_i})$. By boundedness of the subgradients, $\lambda_i$ and $\bw$ we can bound the first component
		$$
		\norm{\bc + \bba^\top \bw^{k-1} + \sum_{i=1}^m\lambda_i^{k-1}\bd^k_{\bx,i}}\leq \norm{\bc}+ \|\bba \| R_\bw + \sum_{i=1}^m \bar{\lambda}G_{\bx,i}.
		$$
		By the definition of $\tilde{U}^{i,l}$ we have that for all $i\in[m]$ and $l\in[s_i-1]$
		\begin{align*}
			\norm{\bu^{k-1}_{i,s_i}-\bu^{k-1}_{i,l}} & \leq (\lambda^{k-1}_{i,s_i}-\lambda_{i,l}^{k-1}) +\norm{\tilde{\bz}_{i,s_i}-\tilde{\bz}_{i,l}} \\
			& \leq 2 \bar{\lambda}+\bar{\lambda}\max_{\bz_{i,s_i}\in Z^{i,s_i},\bz_{i,l}\in Z^{i,l}}\norm{\bz_{i,s_i}-\bz_{i,l}}\leq \bar{\lambda}(2+R_{i,s_i}+R_{i,l}).
		\end{align*}
		Adding these two bounds, we obtain the following bound:
		\begin{align*}
			\norm{\bv^{k}_{\bchi}} %\leq \norm{\bc +\bba^\top \bw^{k-1} + \sum_{i\in[m]}\lambda_i^{k-1}\bv^k_{\bx,i}}+\sum_{i\in[m]}\sum_{l\in[s_i-1]}\norm{\bu^{k-1}_{i,s_i}-\bu^{k-1}_{i,l}}\\
			&\leq \norm{\bc}+ \|\bba \| R_\bw +\sum_{i\in[m]} \bar{\lambda}G_{\bx,i}+\bar{\lambda}\sum_{i\in[m]}\left((s_i-1)(2+R_{i,s_i})+\sum_{l\in[s_i-1]}R_{i,l}\right)=G_\bchi\end{align*}	
		Now, we will bound the norms of the subgradients corresponding to the dual variables $\bu_{i,l}$, $\bw$:
		$$
		\bv_{\tilde{\bu}_i}^k = \left[ \begin{array}{c} \bv_{i,1}^k \\ \vdots \\ \bv_{i, s_i}^k \end{array} \right]
		$$
		First, consider the subgradients $\bv_{i, s_i}$ with respect to $\bu_{i, s_i}$. According to Lemma~\ref{lem:subdif_def} 
		\begin{align*}
			\bv_{i,s_i}^k= \begin{bmatrix} \lambda^{k-1}_{i,s_i} \bd^k_{\bz,i}\\
				(-g_i)(\bx^{k-1},\bz^{k-1}_{i,s_i})+(\bd^{k}_{\bz,i})^\top\bz^{k-1}_{i,s_i}
			\end{bmatrix}-\sum_{l\in[s_i-1]} \bomega^{k-1}_{i,l}\in \partial_{\bu_{i,s_i}}(-\breve{L})(\bchi^{k-1},\bu^{k-1}),
		\end{align*}
		where $\bd^{k}_{\bz,i}=\partial_{\bz_i } (-g_i)(\bx^{k-1},\bz_{i,s_i}^{k-1})$. From the definitions of the sets $\tilde{U}^{i,l}$,$\Omega^{i,l}$, and $W$ we get:
		\begin{align*}
			\norm{\bv_{i,s_i}^k}& \leq \norm{\lambda^{k-1}_{i,s_i} \bd^k_{\bz,i}}+|{(-g_i)(\bx^{k-1},\bz^{k-1}_{i,s_i})|+\norm{(\bd^{k}_{\bz,i})}\norm{\bz^{k-1}_{i,s_i}}}+\sum_{l\in[s_i-1]}\norm{\bomega^{k-1}_{i,l}}\\
			&\leq (\bar{\lambda}+R_{i,s_i})G_{\bz,i}+\bar{g}_i+\bar{\mu}_i(s_i-1)\left(1+\frac{1}{\epsilon_i} \right).
		\end{align*}
		Moreover, for any $l\in[s_i-1]$ it follows from the definition $\Omega^{i,l}$ that
		$$
		\bv_{i,l}^k=-\bomega^{k-1}_{i,l}\in \partial_{\bu_{i,l}}(-\breve{L})(\bchi^{k-1},\by^{k-1}) \quad \Rightarrow \quad \norm{\bv_{i,l}^k}\leq \bar{\mu}_i \left(1+\frac{1}{\epsilon_i} \right).
		$$
		In the end, for the subgradients w.r.t. $\bw$, we have $\bv_\bw^k = -\bba \bx^{k-1} + \bb$ so by boundedness of $X$ we obtain the bound $\norm{\bv_\bw^k} \leq \norm{\bba} R_\bx + \norm{\bb}$. Using the defined $\bv^k_\chi$, $\bv^k_{i,l}$ and $\bv_\bw^k$ in the algorithm we obtain the desired result. \hfill \qedsymbol
		
		{\bf Proof of \eqref{eq.SGSP.EB.result2}} To prove this claim, we will use \eqref{eq.SGSP.EB.result1}. First, define $\bu_i^\dag = \lambda_i \bz_i^\dag$ where $\bz_i^\dag = \argmax_{\bz_i \in Z^i} g(\bar{\bx}^N, \bz_i)$. Following \eqref{eq:strong_dual} we have that
		\begin{align}
			L(\bar{\bx}^N; (\blambda, \bw ) ) &= \bc^\top \bar{\bx}^N +\sum_{i\in[m]}\tilde{g}_i( \bar{\bx}^N,\bu_i^\dag)+ \bw^\top \left( \bba \bar{\bx}^N - \bb \right) \nonumber\\
			&= \min_{\bomega}\max_{\substack{\bu_{i,l}\in \tilde{U}^{i,l},\ \lambda_{i,l}=\lambda_i\\ l\in[s_i],\ i\in[m]}} \bc^\top \bar{\bx}^N +\sum_{i\in[m]}\tilde{g}_i( \bar{\bx}^N,\bu_{i,s_i})+ \bw^\top (\bba \bar{\bx}^N - \bb) + \sum\limits_{i = 1}^m \sum\limits_{l = 1}^{s_i - 1} {\bomega}_{i, l}^\top (\bu_{i,l} - \bu_{i, s_i}) \nonumber\\
			&\leq \max_{\substack{\bu_{i,l}\in \tilde{U}^{i,l},\ \lambda_{i,l}=\lambda_i\\ l\in[s_i],\ i\in[m]}}  \bc^\top \bar{\bx}^N +\sum_{i\in[m]}\tilde{g}_i( \bar{\bx}^N,\bu_{i,s_i})+ \bw^\top (\bba \bar{\bx}^N - \bb) + \sum\limits_{i = 1}^m \sum\limits_{l = 1}^{s_i - 1} (\bar{\bomega}_{i, l}^N)^\top (\bu_{i,l} - \bu_{i, s_i}) \nonumber\\
			&= \max_{\substack{\bu_{i,l}\in \tilde{U}^{i,l},\ \lambda_{i,l}=\lambda_i\\ l\in[s_i],\ i\in[m]}}  \breve{L}((\bar{\bx}^N,\bar{\bomega}^N ),(\tilde{\bu},\bw)). \label{eq:firstbound}
		\end{align}
		We move on to lower bounding $ L(\bx^*, (\bar{\blambda}^{N}, \bar{\bw}^{N}))$ with a term of the form $\breve{L}((\bx^*,\bomega ),\bar{\by}^N)$:
		\begin{align} 
			&   L(\bx^*,( \bar{\blambda}^{N}, \bar{\bw}^N)) \nonumber \\
			= & \max_{\bu_{i} \in\bar{U}_i, \ \bar{\lambda}^N_{i} = \bar{\lambda}^N_{i,s_i} } \bc^\top \bx^* + \sum_{i\in[m]}\tilde{g}_i( \bx^* ,\bu_i) \nonumber \\
			= & \min_{\bomega } \max_{\substack{\bu_{i,l}\in\tilde{U}^{i,l},\ l\in[s_i] \\ \lambda_{i,s_i} = \bar{\lambda}^N_{i,s_i}}} \bc^\top \bx^* + \bar{\bw}^{N\top} (\bba \bx^* - \bb) + \sum_{i\in[m]}\tilde{g}_i( \bx^* ,\bu_{i,s_i}) + \sum\limits_{i = 1}^m \sum\limits_{l = 1}^{s_i - 1} \bomega_{i, l}^\top (\bu_{i,l} - \bu_{i, s_i}) \nonumber \\
		\end{align}
		\begin{align}
			= & \min_{\bomega\in \Omega } \max_{\substack{\bu_{i,l}\in\tilde{U}^{i,l},\ l\in[s_i] \\ \lambda_{i,s_i} = \bar{\lambda}^N_{i,s_i}}}  \bc^\top \bx^* + \bar{\bw}^{N\top} (\bba \bx^* - \bb) + \sum_{i\in[m]}\tilde{g}_i( \bx^* ,\bu_{i,s_i}) + \sum\limits_{i = 1}^m \sum\limits_{l = 1}^{s_i - 1} \bomega_{i, l}^\top (\bu_{i,l} - \bu_{i, s_i}) \nonumber\\
			= & \max_{\substack{\bu_{i,l}\in\tilde{U}^{i,l},\ l\in[s_i] \\ \lambda_{i,s_i} = \bar{\lambda}^N_{i,s_i}}}  \min_{\bomega \in \Omega} \bc^\top \bx^* + \bar{\bw}^{N\top} (\bba \bx^* - \bb) + \sum_{i\in[m]}\tilde{g}_i( \bx^* ,\bu_{i,s_i}) + \sum\limits_{i = 1}^m \sum\limits_{l = 1}^{s_i - 1} \bomega_{i, l}^\top (\bu_{i,l} - \bu_{i, s_i}) \nonumber\\
			\geq & \min_{\bomega \in \Omega} \bc^\top \bx^* + \bar{\bw}^{N\top} (\bba \bx^* - \bb) + \sum_{i\in[m]}\tilde{g}_i( \bx^* , \bar{\bu}^N_{i,s_i}) + \sum\limits_{i = 1}^m \sum\limits_{l = 1}^{s_i - 1} \bomega_{i, l}^\top (\bar{\bu}^N_{i,l} - \bar{\bu}^N_{i, s_i}) \nonumber\\
			= & \min_{\bomega \in \Omega} \breve{L}((\bx^*,\bomega), \bar{\by}^N) \label{eq:secondbound}
		\end{align}
		where the first equality follows \eqref{eq:strong_dual}, the second equality follows from Proposition~\ref{proposition:bounded.omega}, the third equality follows from the fact that existence of a saddle point, established in \eqref{eq:strong_dual}, and the final inequality follows from the definition of $\breve{L}$. Moreover, for any $\bar{\blambda}_i^{N}\equiv \bar{\blambda}_{i,s_i}^{N}$, and $\bar{\bw}^N$, it follows from $(\bx^*,(\blambda^*,\bw^*))$ being a saddle point of $L$ that
		\begin{equation}\label{eq:thirdbound} L(\bx^*,( \bar{\blambda}^{N}, \bar{\bw}^N))\leq L(\bx^*,(\blambda^*,\bw^*)).\end{equation}
		Combining \eqref{eq:firstbound}, \eqref{eq:secondbound}, and \eqref{eq:thirdbound} we have that for any $\lambda_i\in [0,\bar{\lambda}]$ and any $\bw\in W$ the following holds
		\begin{align}
			& L(\bar{\bx}^N, (\blambda, \bw)) - L(\bx^*, ({\blambda}^*, {\bw}^*)) \nonumber \\
			\leq & L(\bar{\bx}^N, (\blambda, \bw)) - L(\bx^*, (\bar{\blambda}^N, \bar{\bw}^N)) \nonumber\\
			\leq & \max_{\bomega \in \Omega} \max_{\substack{\by=(\tilde{\bu},\bw):\\\bu_{i,l}\in \tilde{U}^{i,l},\ \lambda_{i,l}=\lambda_i\\ l\in[s_i],\ i\in[m]}} \left( \breve{L}((\bar{\bx}^N,\bar{\bomega}^N ),\by)   - \breve{L}((\bx^*,\bomega),\bar{\by}^N) \right) \nonumber  \\
			\leq &   \max_{\bomega \in \Omega} \max_{\substack{\bu_{i,l}\in \tilde{U}^{i,l},\ \lambda_{i,l}=\lambda_i\\ l\in[s_i],\ i\in[m]}}  \frac{1}{2\sqrt{N}} \left(\frac{\norm{\bx^{0}-\bx^*}^2+\norm{\bomega^{0}-\bomega}^2}{\tilde{\tau}} + \tilde{\tau}G_\bchi^2 \right.\nonumber\\
			& \left.\qquad\qquad\qquad\qquad\qquad\qquad+ \sum\limits_{i=1}^m \left( \sum_{l=1}^{s_i} \frac{\norm{\bu_{i,l}^{0} - \bu_{i,l}}^2}{\tilde{\theta}_i} 
			+ \tilde{\theta}_i G_{\tilde{\bu}_i}^2 \right) + \frac{\norm{\bw^{0}-\bw}^2}{\tilde{\theta}_\bw}+\tilde{\theta}_\bw G_\bw^2 \right), \label{eq:SPSG_is_EB1a}
		\end{align}
		where the last inequality follows from \eqref{eq.SGSP.EB.result1}.
		Now we can use the fact that for any $\bomega, \bomega^{0} \in\Omega$
		$$
		\norm{\bomega^{0}-\bomega}\leq 2 \sum_{i=1}^m\bar{\mu}_i(s_i-1)( 1 + 1 / \epsilon_i)
		$$
		and for any $\bu_{i,l}^0,\bu_{i,l}\in \tilde{U}^{i,l}$ such that $\bu_{i,l}=(\tilde{\bz}_{i,l},\lambda_{i,l}),\;\lambda_{i,l}=\lambda_i$ 
		$$
		\norm{\bu_{i,l}^{0}-\bu_{i,l}}^2\leq \max\{\lambda_i^0,\lambda_i\}^2(1+4R^2_{i,l}).
		$$
		In the end we have
		$$
		\norm{\bw^{0}-\bw}^2 \leq \frac{2 \max \{ \|\bw\| , \|\bw_0\| \}^2}{ \tilde{\theta}_\bw }
		$$
		We can use these inequalities and the bounds on variables to bound \eqref{eq:SPSG_is_EB1a} and the definition of $\phi$ and $\sigma_i$ we obtain the desired result. 
		
		{\bf Proof of \eqref{eq.SGSP.EB.result3}} For the last claim we use the fact that if $((\bx^*,\bomega^*),(\tilde{\bu}^*,\bw^*))$ is a saddle point of $\breve{L}$ then $(\bx^*,(\blambda^*,\bw^*))$ where $\lambda_i^*=\lambda_{i,s_i}^*$ is a saddle point of $L$. Thus,
		$$L(\bar{\bx}^N, (\blambda^*, \bw^*))\geq L({\bx}^*, (\blambda^*, \bw^*))=\bc^\top\bx^*,$$
		where the inequality follows from $({\bx}^*, (\blambda^*, \bw^*))$ being a saddle point, and the equality follows from the saddle point value for the Lagrangian being equal to the optimal primal objective value. \hfill \qedsymbol
	\end{proof}
	\shimrit{
		\begin{proof}{Proof of Proposition~\ref{prop.CP.convergence}.}
			The first part of the proposition is a direct result of \cite[Theorem 1]{chambolle2011first}. To prove the second part, the strategy is the same as in the SGSP case: we upper (lower) bound the first (second) terms in the left-hand side of \eqref{eq.CP.EB.result2} with terms that are like the ones in the LHS of \eqref{eq.CP.EB.result1} albeit over a bounded domain, and then suprimize the RHS of \eqref{eq.CP.EB.result1} over that domain. Defining $\bu^\dag_{i,l} = \lambda_i(\bz^\dag_i,1)$ for all $l = 1,\ldots, s_i$, where $\bz^\dag_i := \argmax_{\bz_i\in Z^i} g_i(\bar{\bx}^N,\bz_i)$, we then have:
			\begin{align}
				L(\bar{\bx}^N,(\blambda, \bw))
				= & \bc^\top \bar{\bx}^N + \bw^\top (\bba \bar{\bx}^N - \bb) + \sum_{i \in [m]}  \tilde{g}_i(\bar{\bx}^N, \bu^\dag_i)  \nonumber \\ 
				= & \bc^\top \bar{\bx}^N + \bw^\top (\bba \bar{\bx}^N - \bb) + \sum_{i \in [m]}  \tilde{g}_i(\bar{\bx}^N, \bu^\dag_{i,s_i})+\sum_{l\in[s_i-1]}(\bar{\bomega}^N_{i,l})^\top(\bu^\dag_{i,l}-\bu^\dag_{i,s_i}) \nonumber \\
				= & \check{L}(\bar{\bchi}^N, (\tilde{\bu}^\dag, \bw)). \label{eq.papc.composite.bound.1}
			\end{align}
			where the first equality follows from the definition of $L$, the second equality follows from $\bu^\dag_{i,l}=\bu^\dag_{i,s_i}$ for all $l\in[s_i-1]$ and $i\in[m]$, and the third equality follows from definition of $\check{L}$. Since $(\bx^*, (\blambda^*, \bw^*))$ is a saddle point of $L$, it can be extended to $(\bchi^*,\by^*)$ which is a saddle point of $\check{L}$, and by definition
			$$
			L(\bx^*, (\blambda^*, \bw^*)) = \check{L}({\bchi}^*, \by^*) \geq \check{L}({\bchi}^*, \bar{\by}^N),
			$$
			which combined with inequalities \eqref{eq.CP.EB.result1} (with $\mathcal{B}_1=\{\bchi^*\}$ and $\mathcal{B}_2=\{\by=({\bu}^\dag,\bw): \bu^\dag_{i,l}=\bu^\dag_{i,s_i},\;
			\bu^\dag_{i,s_i}=\lambda_i(\bz^\dag_i,1),\;
			\bz_{i}^\dag\in \argmax_{\bz\in Z^i} g_i(\bar{\bx}^N,\bz),\;i\in[m],\;l\in[s_i-1]\}$) and \eqref{eq.papc.composite.bound.1} gives
			\begin{align}
				L(\bar{\bx}^N,(\blambda, \bw)) - L(\bx^*, (\blambda^*, \bw^*)) &\leq \max_{\by\in \mathcal{B}_2}\check{L}(\bar{\bchi}^N,\by)-\check{L}(\bchi^*,\bar{\by}^N) \nonumber\\
				&\leq  \max_{\by\in \mathcal{B}_2}\frac{\tau^{-1} \| \bchi^* - \bchi^0 \|^2 + \sigma^{-1}\| \by- \by^0 \|^2 }{2N},  \label{eq.cor.6.gap_bound}
			\end{align}
			From Proposition~\ref{proposition:bounded.omega} and $\bomega^0=\bzero$ we have that
			\begin{align}\label{eq.cor.6.primal_bound}
				\norm{ \bchi^* - \bchi^0 }^2 &=\norm{\bx^* - \bx^0}^2+ \norm{\bomega^*-\bomega^0}^2\leq \norm{\bx^* - \bx^0}^2 +  \sum_{i=1}^m (s_i-1)\bar{\mu}_i^2\left( 1+\frac{1}{\epsilon_i} \right)^2.
			\end{align}
			The boundedness of $Z^{i}$ implies that
			\begin{align} \nonumber
				\sup_{\substack{\by=(\tilde{\bu}^\dag, \bw): \bu^\dag_{i,l}=\bu^\dag_{i,s_i}\\
						\bu^\dag_{i,s_i}=\lambda_i(\bz^\dag_i,1)\\
						\bz_{i}^\dag\in \argmax_{\bz\in Z^i} g_i(\bar{\bx}^N,\bz)}}\norm{ \by- \by^0 }^2  %\nonumber
				% \end{align}
			% \begin{align}
				& \leq \sum_{i\in[m]}  \sum_{l\in[s_i]} \max_{\bu_{i,l} = \lambda_i \bz_{i,l}:\bz_{i,l}\in Z^{i,l}}\norm{\bu_{i,l} - \bu^0_{i,l}}^2 +  \norm{\bw - \bw^0}^2 \nonumber \\
				& \leq \sum_{i\in[m]}  \sum_{l\in[s_i]} (1+4R^2_{i,l})\max\{\lambda_i,\lambda_{i}^0\}^2 + 2 \max\{ \norm{\bw}, \norm{\bw^0} \}^2 \nonumber\\
				& = \sum_{i\in[m]}  \sigma_i \max\{\lambda_{i}, \lambda_{i}^0\}^2 + 2 \max\{ \norm{\bw}, \norm{\bw^0} \}^2 \label{eq.cor.6.dual_bound}
			\end{align}
			Combining \eqref{eq.cor.6.gap_bound}, \eqref{eq.cor.6.primal_bound}, and \eqref{eq.cor.6.dual_bound} we obtain \eqref{eq.CP.EB.result2}. The last part follows by the same argument as in the case of SGSP.
			
			% To prove the last part, we use the fact that $\norm{\bpi^*} \leq R_\pi$ and so by \eqref{eq.papc.composite.bound.1} choosing $\alpha = R_\bpi$ we have
			% \begin{equation}\label{eq.cor.6.lower_bound}
				% L(\bar{\bx}^N, (\blambda^*, \bw^*)) + R_\bpi \dist(\bar{\bx}^N, X) = \sup_{\substack{(\tilde{\bu}^\dag,\bpi):\norm{\bpi} \leq R_\bpi\\ \bu^\dag_{i,l}=\bu^\dag_{i,s_i}, l\in[s_i-1]\\
						% 		\bz_{i}^\dag\in \argmax_{\bz\in Z^i} g_i(\bar{\bx}^N,\bz)\\
						% \bu^\dag_{i,s_i}=\lambda^*_i(1,\bz_i^\dag)\;	, i\in[m]}} \check{L}(\bar{\bchi}^N, (\bpi, \tilde{\bu}^\dag, \bw^*)) \geq \check{L}(\bar{\bchi}^N, \by^*).
				% \end{equation}
			% Moreover, from the definition of the saddle point we have that
			% \begin{equation}\label{eq.cor.6.lower_bound2}
				% \check{L}(\bar{\bchi}^N,\by^*) \geq \check{L}({\bchi}^*,\by^*) = L(\bx^*,(\blambda^*, \bw^*)) = \bc^\top\bx^*.
				% \end{equation}
			% Combining \eqref{eq.cor.6.lower_bound} and \eqref{eq.cor.6.lower_bound2} concludes the proof. \hfill \qedsymbol
	\end{proof}}
	
\end{APPENDIX}


\begin{thebibliography}{31}
	\providecommand{\natexlab}[1]{#1}
	\providecommand{\url}[1]{\texttt{#1}}
	\providecommand{\urlprefix}{URL }
	
	\bibitem[{Alacaoglu et~al.(2021)Alacaoglu, Malitsky, \protect\BIBand{}
		Cevher}]{alacaoglu2021forward}
	Alacaoglu A, Malitsky Y, Cevher V (2021) Forward-reflected-backward method with
	variance reduction. \emph{Computational optimization and applications}
	80(2):321--346.
	
	\bibitem[{Andersson et~al.(2019)Andersson, Gillis, Horn, Rawlings,
		\protect\BIBand{} Diehl}]{Andersson2019}
	Andersson JAE, Gillis J, Horn G, Rawlings JB, Diehl M (2019) {CasADi} -- {A}
	software framework for nonlinear optimization and optimal control.
	\emph{Mathematical Programming Computation} 11(1):1--36,
	\urlprefix\url{http://dx.doi.org/10.1007/s12532-018-0139-4}.
	
	\bibitem[{Auslender \protect\BIBand{} Teboulle(2009)}]{auslender2009projected}
	Auslender A, Teboulle M (2009) Projected subgradient methods with non-euclidean
	distances for non-differentiable convex minimization and variational
	inequalities. \emph{Mathematical Programming} 120(1):27--48.
	
	\bibitem[{Bauschke \protect\BIBand{} Combettes(2011)}]{bauschke2011convex}
	Bauschke HH, Combettes PL (2011) \emph{Convex analysis and monotone operator
		theory in Hilbert spaces} (Springer).
	
	\bibitem[{Ben-Tal et~al.(2015{\natexlab{a}})Ben-Tal, den Hertog,
		\protect\BIBand{} Vial}]{BenTal2015}
	Ben-Tal A, den Hertog D, Vial JP (2015{\natexlab{a}}) Deriving robust
	counterparts of nonlinear uncertain inequalities. \emph{Mathematical
		Programming} 149(1):265--299, ISSN 1436-4646,
	\urlprefix\url{http://dx.doi.org/10.1007/s10107-014-0750-8}.
	
	\bibitem[{Ben-Tal et~al.(2009)Ben-Tal, {El Ghaoui}, \protect\BIBand{}
		Nemirovski}]{bental2009robust}
	Ben-Tal A, {El Ghaoui} L, Nemirovski A (2009) \emph{Robust optimization}
	(Princeton University Press).
	
	\bibitem[{Ben-Tal et~al.(2015{\natexlab{b}})Ben-Tal, Hazan, Koren,
		\protect\BIBand{} Mannor}]{bental2015oracle}
	Ben-Tal A, Hazan E, Koren T, Mannor S (2015{\natexlab{b}}) Oracle-based robust
	optimization via online learning. \emph{Operations Research} 63(3):628--638.
	
	\bibitem[{Bertsimas et~al.(2011)Bertsimas, Brown, \protect\BIBand{}
		Caramanis}]{bertsimas2011theory}
	Bertsimas D, Brown DB, Caramanis C (2011) Theory and applications of robust
	optimization. \emph{SIAM {R}eview} 53(3):464--501.
	
	\bibitem[{Bienstock(2007)}]{bienstock2007histogram}
	Bienstock D (2007) Histogram models for robust portfolio optimization.
	\emph{Journal of Computational Finance} 11(1):1.
	
	\bibitem[{Chambolle \protect\BIBand{} Pock(2011)}]{chambolle2011first}
	Chambolle A, Pock T (2011) A first-order primal-dual algorithm for convex
	problems with applications to imaging. \emph{Journal of mathematical imaging
		and vision} 40(1):120--145.
	
	\bibitem[{Combettes \protect\BIBand{}
		M{\"u}ller(2018)}]{combettes2018perspective}
	Combettes PL, M{\"u}ller CL (2018) Perspective functions: Proximal calculus and
	applications in high-dimensional statistics. \emph{Journal of Mathematical
		Analysis and Applications} 457(2):1283--1306.
	
	\bibitem[{Gabrel et~al.(2014)Gabrel, Murat, \protect\BIBand{}
		Thiele}]{gabrel2014recent}
	Gabrel V, Murat C, Thiele A (2014) Recent advances in robust optimization: An
	overview. \emph{European Journal of Operational Research} 235(3):471--483.
	
	\bibitem[{Gidel et~al.(2017)Gidel, Jebara, \protect\BIBand{}
		Lacoste-Julien}]{gidel2017frank}
	Gidel G, Jebara T, Lacoste-Julien S (2017) Frank-wolfe algorithms for saddle
	point problems. \emph{Artificial Intelligence and Statistics}, 362--371
	(PMLR).
	
	\bibitem[{Hazan(2016)}]{Hazan2016}
	Hazan E (2016) Introduction to online convex optimization. \emph{Foundations
		and Trends in Optimization} 2(3-4):157--325, ISSN 2167-3888,
	\urlprefix\url{http://dx.doi.org/10.1561/2400000013}.
	
	\bibitem[{Ho-Nguyen \protect\BIBand{}
		K{\i}l{\i}n{\c{c}}-Karzan(2018)}]{ho2018online}
	Ho-Nguyen N, K{\i}l{\i}n{\c{c}}-Karzan F (2018) Online first-order framework
	for robust convex optimization. \emph{Operations Research} 66(6):1670--1692.
	
	\bibitem[{Ho-Nguyen \protect\BIBand{}
		K{\i}l{\i}n{\c{c}}-Karzan(2019)}]{ho2019exploiting}
	Ho-Nguyen N, K{\i}l{\i}n{\c{c}}-Karzan F (2019) Exploiting problem structure in
	optimization under uncertainty via online convex optimization.
	\emph{Mathematical Programming} 177(1):113--147.
	
	\bibitem[{Jeyakumar \protect\BIBand{} Li(2014)}]{jeyakumar2014trust}
	Jeyakumar V, Li G (2014) Trust-region problems with linear inequality
	constraints: exact sdp relaxation, global optimality and robust optimization.
	\emph{Mathematical Programming} 147(1-2):171--206.
	
	\bibitem[{Juditsky \protect\BIBand{} Nemirovski(2011)}]{juditsky2011first}
	Juditsky A, Nemirovski A (2011) First order methods for nonsmooth convex
	large-scale optimization, ii: utilizing problems structure. S SN Sra, Wright
	SJ, eds., \emph{Optimization for Machine Learning}, chapter~6, 149--181 (MIT
	press Cambridge, MA).
	
	\bibitem[{Korpelevich(1976)}]{korpelevich1977extragradient}
	Korpelevich G (1976) Extragradient method for finding saddle points and other
	problems. \emph{Ekonomika i Matematicheskie Metody} 12(4):747--756.
	
	\bibitem[{Malitsky \protect\BIBand{} Tam(2020)}]{Malitsky2020}
	Malitsky Y, Tam MK (2020) A forward-backward splitting method for monotone
	inclusions without cocoercivity. \emph{SIAM Journal on Optimization}
	30(2):1451--1472.
	
	\bibitem[{Mutapcic \protect\BIBand{} Boyd(2009)}]{mutapcic2009cutting}
	Mutapcic A, Boyd S (2009) Cutting-set methods for robust convex optimization
	with pessimizing oracles. \emph{Optimization Methods \& Software}
	24(3):381--406.
	
	\bibitem[{Nedi{\'c} \protect\BIBand{} Ozdaglar(2009)}]{nedic2009subgradient}
	Nedi{\'c} A, Ozdaglar A (2009) Subgradient methods for saddle-point problems.
	\emph{Journal of Optimization Theory and Applications} 142(1):205--228.
	
	\bibitem[{Nemirovski(1994)}]{LectureNem94}
	Nemirovski A (1994) Information based complexity of convex programming. Lecture
	notes, \urlprefix\url{http://www2.isye.gatech.edu/~nemirovs/Lec_EMCO.pdf}.
	
	\bibitem[{Nemirovski(2004)}]{nemirovski2004prox}
	Nemirovski A (2004) Prox-method with rate of convergence o (1/t) for
	variational inequalities with lipschitz continuous monotone operators and
	smooth convex-concave saddle point problems. \emph{SIAM Journal on
		Optimization} 15(1):229--251.
	
	\bibitem[{Nesterov(2007)}]{nesterov2007dual}
	Nesterov Y (2007) Dual extrapolation and its applications to solving
	variational inequalities and related problems. \emph{Mathematical
		Programming} 109(2):319--344.
	
	\bibitem[{Optimization(2020)}]{gurobi}
	Optimization G (2020) Gurobi optimizer reference manual.
	\urlprefix\url{http://www.gurobi.com}.
	
	\bibitem[{Ouyang \protect\BIBand{} Xu(2021)}]{ouyang2021lower}
	Ouyang Y, Xu Y (2021) Lower complexity bounds of first-order methods for
	convex-concave bilinear saddle-point problems. \emph{Mathematical
		Programming} 185(1):1--35.
	
	\bibitem[{Postek \protect\BIBand{} Shtern(2024)}]{our_repo}
	Postek K, Shtern S (2024) First-order algorithms for robust optimization
	problems via convex-concave saddle-point lagrangian reformulation.
	\urlprefix\url{http://dx.doi.org/10.1287/ijoc.2022.0200}, available for
	download at http://github.com/INFORMSJoC/2022.0200.
	
	\bibitem[{Tseng(1991)}]{tseng1991applications}
	Tseng P (1991) Applications of a splitting algorithm to decomposition in convex
	programming and variational inequalities. \emph{SIAM Journal on Control and
		Optimization} 29(1):119--138.
	
	\bibitem[{W{\"a}chter \protect\BIBand{}
		Biegler(2006)}]{wachter2006implementation}
	W{\"a}chter A, Biegler LT (2006) On the implementation of an interior-point
	filter line-search algorithm for large-scale nonlinear programming.
	\emph{Mathematical Programming} 106(1):25--57.
	
	\bibitem[{Zhang et~al.(2022)Zhang, Hong, \protect\BIBand{}
		Zhang}]{zhang2022lower}
	Zhang J, Hong M, Zhang S (2022) On lower iteration complexity bounds for the
	convex concave saddle point problems. \emph{Mathematical Programming}
	194(1-2):901--935.
	
\end{thebibliography}
\end{document}